\documentclass[12pt]{amsart}
\usepackage{amsmath,amssymb,amsfonts}
\usepackage{enumerate}
\usepackage{amscd, amssymb, latexsym, amsmath, amscd}
\usepackage[all]{xy}
\usepackage{pb-diagram}

\newtheorem{thm}[equation]{Theorem}
\newtheorem{prop}[equation]{Proposition}
\newtheorem{cor}[equation]{Corollary}
\newtheorem{con}[equation]{Conjecture}
\newtheorem{lemma}[equation]{Lemma}

\numberwithin{equation}{section}

\newcommand{\Q}{\mathbb Q}
\newcommand{\Z}{\mathbb Z}
\newcommand{\R}{\mathbb R}

\newcommand{\G}{\mathfrak G}

\newcommand{\A}{\mathbb A}

\def\Hom{{\rm Hom}}
\def\Aut{{\rm Aut}}
\def\ord{{\rm ord}}
\def\G{{\rm G}}
\def\SL{{\rm SL}}
\def\disc{{\rm disc}}

\def\Sp{{\rm Sp}}
\def\Spin{{\rm Spin}}

\def\U{{\rm U}}

\def\GL{{\rm GL}}
\def\PGL{{\rm PGL}}

\def\Gal{{\rm Gal}}
\def\SO{{\rm SO}}
\def\Ind{{\rm Ind}}
\def\Sp{{\rm Sp}}
\def\sign{{\rm sign}}
\def\O{{\rm O}}
\def\Sym{{\rm Sym}}

\usepackage[all]{xy}

\def\A{{\mathbb A}}
\def\ws{{\widetilde{\Sp}}}
\def\wpi{{\widetilde{\pi}}}
\def\GG{{\mathbb G}}
\def\NN{{\mathbb N}}

\def\R{{\mathbb R}}

\def\Va{{V_{\alpha}}}

\def\Z{{\mathbb Z}}

\def\vr{{\varphi}}

\def\ep{{\epsilon}}

\def\AA{{\mathbb A}}

\def\QQ{{\mathbb Q}}
\def\RR{{\mathbb R}}

\def\SS{{\mathbb S}}
\def\CC{{\mathbb C}}

\def\wp{{W^{\perp}}}

\def\G{{\mathbb G}} 
\def\wg{{\widehat{G}}}

\title[Restriction Problems for Classical Groups]{Symplectic local root numbers, central critical L-values, and restriction problems in the representation theory of classical groups}

\author{Wee Teck Gan, Benedict H. Gross and Dipendra Prasad}

\address{W.T.G.: Department of Mathematics, University of California at San Diego, 9500 Gilman Drive, La Jolla, 92093} \email{wgan@math.ucsd.edu}
\address{B.H.G: Department of Mathematics, Harvard University, 
Cambridge, MA 02138}\email{gross@math.harvard.edu}
\address{D.P.: School of Mathematics, Tata Institute of Fundamental
Research, Colaba, Mumbai-400005, INDIA}
\email{dprasad@math.tifr.res.in}

\begin{document}
\maketitle

\tableofcontents
\section{Introduction}

It has been almost 20 years since two of us proposed a rather speculative
approach to the problem of restriction of irreducible representations from $\SO_n$ to $\SO_{n-1}$ [GP1, GP2].  Our predictions depended on the Langlands parameterization of
irreducible representations, using $L$-packets and $L$-parameters.  Since then, there has
been considerable progress in the construction of local $L$-packets, as
well as on both local and global aspects of the restriction problem.  We
thought it was a good time to review the precise conjectures which
remain open, and to present them in a more general form, involving
restriction problems for all of the classical groups.
\vskip 10pt

Let $k$ be a local field equipped with an involution $\sigma$ with fixed field $k_0$. Let $V$ be a vector space over $k$ with a non-degenerate sesquilinear form and let $G(V)$ be the identity
component of the classical subgroup of $\GL(V)$ over $k_0$ which preserves this form.
There are four distinct cases, depending on whether the space $V$ is orthogonal,
symplectic, hermitian, or skew-hermitian. In each case, for certain 
non-degenerate
subspaces $W$ of $V$, we define a subgroup $H$ of the locally compact group
$G = G(V) \times G(W)$ containing the diagonally embedded subgroup $G(W)$, 
and a unitary representation $\nu$ of $H$. The local restriction
problem is to determine  
\[  d(\pi) = \dim_{\CC} \Hom_H(\pi \otimes \overline{\nu},\CC), \]
where $\pi$ is an irreducible complex representation of $G$.  
\vskip 5pt

The basic cases are when
$\dim V  - \dim W = 1$ or $0$, where $\nu$ is the trivial representation or a Weil representation respectively. When 
$\dim V - \dim W \geq 2$,
this restriction problem is also known as the existence and uniqueness of Bessel or Fourier-Jacobi models in the literature. As in [GP1] and [GP2], our
predictions involve the Langlands parameterization, in a form suggested by Vogan [Vo],
and the signs of symplectic root numbers.
\vskip 10pt

We show that the Langlands parameters for irreducible representations of classical groups (and for
genuine representations of the metaplectic group) are complex representations of the Weil-Deligne group of $k$, of specified dimension and with certain duality properties. We describe these parameters
and their centralizers in detail, before using their symplectic root numbers to construct certain distinguished characters of the component group. Our local conjecture states that
there is a unique representation $\pi$ in each generic Vogan $L$-packet, such that the
dimension $d(\pi)$ is equal to $1$. Furthermore,  this representation corresponds to a
distinguished character $\chi$ of the component group. For all other representations $\pi$ in the
$L$-packet, we predict that $d(\pi)$ is equal to $0$. The precise statements are contained in Conjectures \ref{conj-mult1} and \ref{conj-character}.
\vskip 10pt

Although this material is largely conjectural, we prove a number
of new results in number theory and representation theory along the way:
\vskip 5pt

\begin{enumerate}[(i)]
\item In Proposition \ref{P:orthogonal-ep}, we give  a generalization of
a formula of Deligne on orthogonal root numbers to the root numbers of conjugate
orthogonal representations. 
\vskip 5pt

\item We describe the L-parameters of classical groups,
and unitary groups in particular, in a much simpler way than currently exists in the literature; this is contained in Theorem \ref{T:classical-par}.

\vskip 5pt

\item We show in Theorem \ref{T:KR} that the irreducible representations of the metaplectic group can be classified in terms of the irreducible representations of odd special orthogonal groups; this largely follows from fundamental results of  Kudla-Rallis [KR], though the statement of the theorem did not appear explicitly in [KR]. 
\vskip 5pt

\item We prove two theorems (cf. Theorems \ref{T:SOVPS} and \ref{T:FJPS}) that allow us to show the uniqueness of general Bessel and Fourier-Jacobi models over non-archimedean local fields. More precisely, we show  that $d(\pi) \leq 1$ in almost all cases (cf. Corollaries \ref{C:orthogonal-hermitian}, \ref{C:symplectic} and \ref{C:skew-hermitian}), reducing this to the basic cases when $\dim W^{\perp} =0$ or $1$, which were recently established by  [AGRS] and [S].   The same theorems allow us to reduce our local conjectures to these basic cases, as shown in Theorem \ref{T:reduction}. 
 \end{enumerate}
 \vskip 10pt
 
 One subtle point about our local conjecture is its apparent dependence on the choice of an additive character $\psi$ of $k_0$ or $k/k_0$. Indeed, 
the choice of such a character $\psi$  is potentially used in 3 places:
\vskip 5pt
\begin{enumerate}[(a)]
\item the Langlands-Vogan parametrization (which depends on fixing a
quasi-split pure inner form $G_0$ of $G$, a Borel subgroup $B_0$ of $G_0$, and a
non-degenerate character on the unipotent radical of $B_0$);
\vskip 5pt
 
\item the definition of the distinguished character $\chi$ of the component group;
\vskip 5pt

\item the representation  $\nu$ of $H$  in the restriction problem.
\end{enumerate}
\vskip 5pt

\noindent Typically, two of the above depend on the choice of $\psi$, whereas the third one doesn't. More precisely, we have:
 \vskip 10pt
 
 \begin{enumerate}[-]
 \item in the orthogonal case, none of (a), (b) or (c) above depends on $\psi$; this explains why this subtlety does not occur in [GP1] and [GP2].
\vskip 5pt

\item in the hermitian case, (a) and (b) depend on the choice of $\psi: k/k_0 \to  \SS^1$, but (c) doesn't.
\vskip 5pt
\item in the symplectic/metaplectic  case,  (a) and (c) depend on $\psi: k_0 \to \SS^1$, but (b) doesn't.
\vskip 5pt

\item in the odd skew-hermitian case, (b) and (c) depend on $\psi: k_0 \to \SS^1$, but (a) doesn't.
\vskip 5pt

\item in the even skew-hermitian case, (a) and (c) depend on $\psi: k_0 \to \SS^1$ but (b) doesn't.
 \end{enumerate}
 \vskip 10pt
 
 Given this, we check in \S \ref{S:compatible} that the dependence on $\psi$ 
 cancels out in each case, so that our local conjecture is internally consistent with respect to changing $\psi$. There is, however, a variant of our local conjectures 
 which is less sensitive to the choice of $\psi$, but is slightly weaker. This variant is given in Conjecture \ref{C:refined2}. Finally, when all the data involved are unramified, we state a more refined conjecture; this is contained in Conjecture \ref{conj-unramified}. 
 
\vskip 10pt
After these local considerations, we study the global restriction problem, for cuspidal tempered representations
 of adelic groups. Here our predictions involve the central values of automorphic $L$-functions,
associated to a distinguished symplectic representation $R$ of the  $L$-group. More precisely, 
let $G=G(V) \times G(W)$ and assume that $\pi$ is an
irreducible cuspidal representation of $G({\Bbb A})$, where ${\Bbb A}$ 
is the ring of ad\`eles of a global field $F$. If the vector space 
$\Hom_{H(\Bbb A)}(\pi \otimes \bar{\nu}, \CC)$ 
is nonzero, our local conjecture implies that the global root number
$\epsilon(\pi,R,\frac{1}{2})$ is equal to 1. If we assume $\pi$ to be tempered,
then our calculation of global root numbers and the general conjectures of 
Langlands and Arthur predict that $\pi$ appears with multiplicity one in the
discrete spectrum of $L^2(G(F)\backslash G(\AA))$. We conjecture that the period
integrals on the corresponding space of functions 
\[  f \mapsto  \int_{H(k) \backslash H(\A)}  f(h)\cdot \overline{\nu(h)} \, dh \]
gives a nonzero element in $\Hom_{H(\Bbb A)}(\pi \otimes \bar{\nu}, \CC)$ 
if and only if the central critical $L$-value $L( \pi, R, \frac{1}{2})$ is 
nonzero. 
\vskip 5pt

This first form of our global conjecture is given in \S \ref{S:global-conj}, after which we examine the  global restriction problem in the framework of Langlands-Arthur's conjecture on the automorphic discrete spectrum, and formulate a more refined global conjecture in \S \ref{S:central}. For this purpose, we formulate an extension of Langlands' multiplicity formula for metaplectic groups; see Conjecture \ref{conj:arthur-meta}.

\vskip 5pt

One case in which all of these conjectures are known to be true is when 
$k = k_0 \times k_0$ is the split quadratic \'etale algebra over $k_0$,  and $V$ is a hermitian space over $k$ of dimension $n$ containing a codimension one nondegenerate subspace $W$.
Then
\[  G \cong \GL_n(k_0) \times \GL_{n-1}(k_0) \quad \text{and} \quad
H \cong \GL_{n-1}(k_0). \]
Moreover, $\nu$ is the trivial representation.  When $k_0$ is local, and $\pi$ is a 
generic representation of $G = \GL_n(k_0) \times \GL_{n-1}(k_0)$, the local theory of Rankin-Selberg integrals [JPSS], together with the multiplicity one theorems of [AGRS], [AG] and [SZ], shows that
\[  \dim \Hom_{H}(\pi, \CC) = 1. \]
This agrees with our local conjecture, as the Vogan packets for $G=\GL_n(k_0)
\times \GL_{n-1}(k_0)$ are singletons. If $k_0$ is global and $\pi$ 
is a cuspidal representation of $G({\Bbb A})$, then $\pi$
appears with multiplicity one in the discrete spectrum. The global theory of Rankin-Selberg integrals [JPSS] implies that the 
period integrals over $H(k)\backslash H({\Bbb A})$ give a nonzero 
linear form on $\pi$ if and only if  
\[  L(\pi,{\rm std}_n \otimes {\rm std}_{n-1},1/2) \ne 0, \] 
where $ L(\pi,{\rm std}_n \otimes {\rm std}_{n-1},s)$ denotes the tensor product L-function.
Again, this agrees with our global conjecture, since  in this case,  the local and global root numbers are all equal to 1, and 
\[   R= {\rm std}_n \otimes {\rm std}_{n-1} +  {\rm std}^{\vee}_n \otimes {\rm std}^{\vee}_{n-1}. \]
  
\vskip 5pt

In certain cases where the global root number $\epsilon =  -1$, so that the central value is zero, we also make a prediction for the first derivative in \S \ref{S:derivative}.   The cases we treat are certain orthogonal and hermitian cases, with $\dim W^{\perp} = 1$. We do not know if there is an analogous conjecture for the first derivative in the symplectic or skew-hermitian cases. 
\vskip 10pt

In a sequel to this paper,  we will
present some evidence for our conjectures, for groups of small rank and for certain discrete $L$-packets where one can calculate the distinguished character explicitly.

\vskip 10pt

\noindent{\bf Acknowledgments:}  W. T. Gan is partially supported by NSF grant DMS-0801071. 
B. H. Gross is partially supported by NSF grant DMS 0901102. 
 D. Prasad was partially  supported by a  Clay Math Institute fellowship during the course of this work. We also thank P. Deligne, S. Kudla, M. Reeder, D. Rohrlich, and J.-L. Waldspurger for their help.

 \vskip 15pt
 
\section{Classical groups and restriction of representations}  \label{S:classical}

Let $k$ be  a field, not of characteristic 2.  Let $\sigma$ be an
involution of $k$ having $k_0$ as the fixed 
field.  If $\sigma = 1$, then $k_0 = k$.  If $\sigma \not= 1$, $k$ is a 
quadratic extension of $k_0$ and  $\sigma$ is the nontrivial element in the 
Galois group $\text{Gal}(k/k_0)$. 

Let $V$ be a finite dimensional vector space over $k$.  Let 
\[ \langle,\rangle : V  \times V \to k \]
be a non-degenerate, $\sigma$-sesquilinear form on $V$, 
which is $\epsilon$-symmetric (for $\epsilon = \pm 1$ in $k^\times$):
\begin{center}
$$
\begin{array}{rcl}
\langle \alpha v + \beta w, u\rangle & = & \alpha \langle v, u\rangle+ \beta \langle w, 
u\rangle \\
\langle u, v\rangle & = & \epsilon \cdot \langle v, u\rangle^{\sigma}.
\end{array}$$
\end{center}

\noindent Let $G(V) \subset \GL(V)$ be the algebraic subgroup of elements $T$ in $\GL(V)$ which 
preserve the form $\langle,\rangle$:  
\[ \langle Tv,Tw\rangle  =  \langle v,w \rangle. \]
Then $G(V)$ is a classical group, defined over the field $k_0$. 
The different possibilities for $G(V)$ are given in the following table.

\vskip 15pt

\begin{tabular}{|c|c|c|c|}
\hline   
& & &  \\
$(k,\epsilon)$ & $ k = k_0$, $\epsilon = 1$ & $k = k_0$, $\epsilon = -1$ & $k/k_0$ quadratic, $\epsilon = \pm 1$ \\
\hline 
& & &  \\
$G(V)$ & orthogonal group $\O(V)$ & symplectic group $\Sp(V)$ & unitary group $\U(V)$  \\
\hline 
\end{tabular}

\vskip 15pt


 In our formulation, a classical group will 
always be associated to a space $V$, so the hermitian and skew-hermitian 
cases are distinct.  Moreover, the group $G(V)$ is connected except in the orthogonal case. In that case,
we let $\SO(V)$ denote the connected component, which consists of elements $T$ of determinant $+1$, and shall refer to $\SO(V)$ as a connected classical group.  We will only work with connected classical groups in this paper.

\vskip 5pt

If one takes $k$ to be the quadratic algebra $k_0 \times k_0$ with involution  
$\sigma(x,y) = (y, x)$ and $V$ a free $k$-module, then a non-degenerate 
form $\langle,\rangle$ identifies the $k = k _0\times k_0$ module $V$ with the sum 
$V_0 + V_0^\vee$, where $V_0$ is a finite dimensional vector space over $k_0$ 
and $V_0^\vee$ is its dual.  In this case $G(V)$ is isomorphic to the 
general linear group $\GL(V_0)$ over $k_0$. 
\vskip 5pt

If $G$ is a connected, reductive group over $k_0$, the pure inner forms of
$G$ are the groups $G'$ over $k_0$ which are obtained by inner twisting by elements
in the pointed set $H^1(k_0,G)$. If $\{g_{\sigma}\}$ is a one cocycle on the Galois group
of the separable closure $k_0^s$ with values in $G(k_0^s)$, the corresponding pure inner form $G'$ has points
$$ G'(k_0) = \{ a \in G(k_0^s) : a^{\sigma} = g_{\sigma} a g_{\sigma}^{-1} \}.$$
The group $G'$ is well-defined up to inner automorphism over $k_0$ by the cohomology class 
of $g_{\sigma}$, so one can speak of a representation of $G'(k_0)$.
\vskip 10pt

For connected, classical groups $G(V) \subset \GL(V)$, the pointed set $H^1(k_0,G)$
and the pure inner forms $G'$ correspond bijectively to forms $V'$ of the space $V$ with
its sesquilinear form $\langle,\rangle$ (cf. [KMRT, \S 29D and \S 29E]).

\begin{lemma}
\begin{enumerate}
\item If $G = \GL(V)$ or $G = \Sp(V)$, then the pointed set $H^1(k_0,G) = 1$ and there are no nontrivial pure inner forms of $G$.

\item If $G = \U(V)$, then elements of the pointed set $H^1(k_0,G)$ correspond bijectively to the isomorphism classes of hermitian (or skew-hermitian) spaces $V'$ over $k$ with $\dim(V') = \dim(V)$. The corresponding pure inner form $G'$ of $G$ is the unitary group $\U(V')$.

\item If $G = \SO(V)$, then elements of the pointed set $H^1(k_0,G)$ correspond bijectively to the isomorphism classes of orthogonal spaces $V'$ over $k$ with $\dim(V') = \dim(V)$ and $\disc(V') = \disc(V)$. The corresponding pure inner form $G'$ of $G$ is the special orthogonal group $\SO(V')$. 
\end{enumerate}
\end{lemma}

Now let $W \subset V$ be a subspace, which is non-degenerate for the 
form $\langle,\rangle$.  Then $V = W + W^{\perp}$. We assume that 
$$\begin{array}{ll}
1) & \epsilon \cdot (-1)^{\dim W^{\perp}} = -1 \\
2) & W^{\perp} \ {\rm is \ a \ split \ space}.
\end{array}
$$

When $\epsilon = -1,$ so dim $W^{\perp} = 2n$ is even, condition 2) 
means that $W^{\perp}$ contains an isotropic subspace $X$  
of dimension $n$.  It follows that $W^{\perp}$ is a direct sum 
\[  \wp = X + Y, \] 
with $X$ and $Y$ isotropic. The pairing $\langle-, - \rangle$ induces a natural map
\[  Y \longrightarrow \Hom_k(X,k) = X^{\vee} \]
which is a $k_0$-linear isomorphism (and $k$-anti-linear if $k \ne k_0$). 
When $\epsilon = + 1$, so dim  $W^{\perp} = 2 n + 1$ is odd, condition 2) means that $\wp$ contains an 
isotropic subspace $X$ of dimension $n$.  It follows that 
\[  \wp = X + Y + E, \]
where $E$ is a non-isotropic line orthogonal to $X + Y$, and $X$ and $Y$ are isotropic. As above, one has a $k_0$-linear isomorphism $Y \cong X^{\vee}$.  
\vskip 5pt

Let $G(W)$ be the subgroup of $G(V)$ which acts trivially on $\wp$.  
This is the classical group, of the same type, associated to the space 
$W$.  Choose an $X \subset \wp$ as above, and let $P$ be the 
parabolic subgroup of $G(V)$ which stabilizes a complete flag of 
(isotropic) subspaces in $X$.  Then $G(W)$, which acts trivially
on both $X$ and $X^\vee$, is contained in a Levi
subgroup of $P$, and acts by conjugation on the unipotent radical $N$ of $P$.  
\vskip 5pt

The semi-direct product $H = N \rtimes G(W)$ embeds as a subgroup of the 
product group $G = G(V) \times  G(W)$ as follows.  We use the defining 
inclusion $H \subset P \subset G(V)$ on the first factor, and the 
projection $H\to H/N = G(W)$ on the second factor.  When
$\epsilon = +1$, the dimension of $H$ is equal to the dimension of
the complete flag variety of $G$. When $\epsilon = -1$, the dimension of $H$ is
equal to the sum of the dimension of the complete flag variety of $G$ and half of
the dimension of the vector space $W$ over $k_0$.
\vskip 10pt

We call a pure inner form $G' = G(V') \times G(W')$ of the group $G$ {\em relevant} if the 
space $W'$ embeds as a non-degenerate subspace of $V'$, with orthogonal
complement isomorphic to $W^{\perp}$. Then a subgroup $H' \to G'$ can be
defined as above.
\vskip 5pt

In this paper, we 
will study the restriction of irreducible complex representations
of the groups $G' = G(V') \times G(W')$ to the subgroups $H'$, when $k$
is a local or a global field.
\vskip 10pt

\vskip 15pt

\section{Selfdual and conjugate-dual representations}  \label{S:selfdual}

Let $k$ be a local field, and let $k^s$ be a separable closure of $k$.  
In this section, we will define selfdual and conjugate-dual representations 
of the Weil-Deligne group $WD$ of $k$. 

When $k = \RR $ or $\CC$, we define $WD$ as the Weil group $W(k)$ of 
$k$, which is an extension of $\Gal(k^s/k)$ by $\CC^\times$, and has
abelianization isomorphic to $k^\times$. A representation of $WD$
is, by definition, a completely reducible (or semisimple) continuous homomorphism 
$$\vr : WD\to \GL(M),$$
where $M$ is a finite dimensional complex vector space.  When $k$ is 
non-archimedean, the Weil group $W(k)$ is the dense subgroup $I \rtimes 
F^{\Z}$ of $\Gal(k^s/k)$, where $I$ is the inertia group and $F$ is a 
geometric Frobenius.   We normalize the isomorphism 
\[  W(k)^{ab} \to k^\times \] 
of local class field theory as in Deligne [D], taking $F$ to a 
uniformizing element of $k^\times$. This defines the norm character
\[  | -|: W(k) \to \RR^\times,\quad  \text{with $|F| = q^{-1}$.}\]
We define $WD$ as the product of $W(k)$ with the group $\SL_2(\CC)$.  A representation is a homomorphism
$$\vr : WD \to \GL(M)$$ 
with 
\begin{enumerate}[(i)]
\item $\vr$ trivial on an open subgroup of $I$,

\item  $\vr(F)$ semi-simple,

\item $\vr : \SL_2(\CC) \to  \GL(M)$ algebraic.

\end{enumerate}
The equivalence of this formulation of representations 
with that of Deligne [D], in which a representation is a homomorphism 
$\rho : W(k) \to \GL(M)$ and a nilpotent endomorphism $N$ of $M$ which 
satisfies $Ad \rho(w) (N) = |w| \cdot N$, is given in [GR, \S 2, Proposition 2.2]
\vskip 10pt

We say two representations $M$ and $M'$ of $WD$ are isomorphic if there 
is a linear isomorphism $f: M \to M'$ which commutes with the action of 
$WD$. If $M$ and $M'$ are two representations of $WD$, we have the
direct sum representation $M \oplus M'$ and the tensor product representation
$M \otimes M'$.  The dual representation $M^\vee$ is defined by the  
natural action on $\Hom(M, \CC)$, and the determinant representation 
$\det(M)$ is defined by the action on the top exterior power.  Since 
$\GL_1(\CC) = \CC^\times$ is abelian, the representation $\det(M)$ factors
through the quotient  $W(k)^{ab} \to k^\times$ of $WD$. 
\vskip 5pt

We now define certain selfdual representations of $WD$.  We say the 
representation $M$ is orthogonal if there is a non-degenerate bilinear 
form 
$$B : M \times M \to \CC$$
which satisfies
$$\begin{cases}
B(\tau m, \tau n)  =  B(m, n) \\ 
B(n, m)  =  B(m, n),
\end{cases}
$$
for all $\tau$  in $ WD$.
\vskip 5pt

We say $M$ is symplectic if there is a non-degenerate bilinear form $B$ on 
$M$ which satisfies 
$$\begin{cases}
B(\tau m, \tau n)  =  B(n, n) \\
B(n, m)  =  -B(m, n),
\end{cases}
$$
for all $ \tau$  in $  WD$.
\vskip 5pt

In both cases, the form $B$ gives an isomorphism of representations 
\[ f :  M \to M^\vee,\]
whose dual
\[  f^{\vee}: M = M^{\vee\vee} \to M^{\vee} \]
satisfies
\[  f^{\vee} = b \cdot f, \quad \text{with $b =$ the sign of $B$.} 
\]
\vskip 5pt

We now note:
\vskip 5pt

\begin{lemma} \label{L:BB'}
Given any two non-degenerate forms $B$ and $B'$ on $M$ preserved by $WD$ with the same sign $ b = \pm 1$, there is  
an automorphism $T$ of M which commutes with $WD$ and such that $B'(m,n) = B(Tm,Tn)$.
\end{lemma}

\begin{proof}
Since $M$ is semisimple as a representation of $WD$, we may write
\[  M = \bigoplus_i V_i\otimes  M_i \]
as a direct sum of irreducible representations with multiplicity spaces  $V_i$. Each $M_i$ is either selfdual or else $M_i^{\vee} \cong M_j$ for some $i \ne j$, in which case $\dim V_i = \dim V_j$. 
So we may write
\[  M = \left( \bigoplus_i V_i \otimes M_i \right)  \oplus  \left( \bigoplus_j V_j \otimes (P_j + P_j^{\vee}) \right) \]
with $M_i$ irreducible selfdual and $P_j$ irreducible but $P_j \ncong P_j^{\vee}$. Since any non-degenerate form $B$ remains non-degenerate on each summand above, we are reduced to the cases:
\vskip 5pt

\begin{enumerate}
\item[(a)] $M = V \otimes N$ with $N$ irreducible and selfdual, in which case the centralizer of the action of $WD$ is $\GL(V)$;
\vskip 5pt

\item[(b)] $M = (V \otimes P)  \oplus (V \otimes  P^{\vee})$, with $P$ irreducible and $P \ncong P^{\vee}$, in which case the centralizer of the action of $WD$ is $\GL(V) \times \GL(V)$. 
\end{enumerate}

\vskip 5pt
In case (a), since $N$ is irreducible and selfdual, there is a unique (up to scaling) $WD$-invariant non-degenerate bilinear form on $N$; such a form on $N$ has some sign $b_N$. Thus, giving a $WD$-invariant non-degenerate bilinear form $B$ on $M$ of sign $b$ is equivalent to giving a non-degenerate bilinear form on $V$ of sign $b \cdot b_N$. But it is well-known that any two non-degenerate 
bilinear forms of a given sign are conjugate under $\GL(V)$. This takes care of (a).
\vskip 5pt

In case (b), the subspaces $V \otimes P$ and $V \otimes P^{\vee}$ are necessarily totally isotropic. Moreover, there is a unique (up to scaling) $WD$-invariant pairing on $P \times P^{\vee}$. Thus to give a $WD$-invariant non-degenerate bilinear form $B$ on $M$ of sign $b$ is equivalent to giving a non-degenerate bilinear form on $V$. But any two such forms are conjugate under the action of 
$\GL(V) \times \GL(V)$ on $V \times V$. This takes care of (b) and the lemma is proved.
\end{proof}

\vskip 10pt

 When $M$ is symplectic, $\dim (M)$ is even and $\det (M) = 1$.  When $M$ 
is orthogonal, $\det (M)$ is an orthogonal representation of dimension 
$1$.  These representations correspond to the quadratic characters 
$$\chi : k^\times \to \langle  \pm  1\rangle . $$ 
Since ${\rm char}(k) \not =  2$, the Hilbert symbol gives a perfect pairing 
$$(- , -) : k^\times / k^{\times 2} \times k^\times / k^{\times 2} \to \langle \pm 1\rangle .$$ 
We let $\CC(d)$ be the one dimensional orthogonal 
representation given by the character $\chi_d(c) = (c,d)$. 
 \vskip 10pt

We also note the following elementary result:
\vskip 10pt

\begin{lemma} 
If $M_i$ is selfdual with sign $b_i$, for $i = 1$ or $2$, then $M_1 \otimes M_2$ is selfdual with sign $b_1 \cdot b_2$.  
\end{lemma}

\begin{proof}
If $M_i$ is selfdual with respect to a form $B_i$ of sign $b_i$, then $M_1 \otimes M_2$ is selfdual with respect to the tensor product $B_1 \otimes B_2$ which has sign $b_1 \cdot b_2$.
\end{proof}
\vskip 10pt

Next, assume that $\sigma$ is a nontrivial involution of $k$, with 
fixed field $k_0$.  Let $s$ be an element of $WD(k_0)$ which generates 
the quotient group 
\[ WD(k_0)/WD(k) = \Gal(k/k_0) = \langle 1,\sigma\rangle. \]  
If $M$ is 
a representation of $WD$, let $M^s$ denote the conjugate 
representation, with the same action of $\SL_2(\CC)$ and the action 
$\tau_s(m) = s \tau s^{-1}(m)$ for $\tau$ in $W(k)$. 
\vskip 10pt

We say the representation $M$ is conjugate-orthogonal if there is a 
non-degenerate bilinear form $B : M \times M \to \CC$ which satisfies 
$$\begin{cases}
B(\tau m,s \tau s^{-1} n)  =  B(m, n) \\
B (n, m)  =  B(m, s^2n), 
\end{cases}
$$ for all $\tau$   in   $WD$. 
We say $M$ is conjugate-symplectic if there is a non-degenerate bilinear 
form on $M$ which satisfies
$$\begin{cases}
B(\tau m, s\tau s^{-1}n)  =  B(m, n) \\
B(n, m)  =  -B(m, s^2n),
\end{cases}
$$ for all $\tau$ in $WD$.
In both cases, the form $B$ gives an isomorphism of representations 
\[  f  : M^s \to M^\vee,\]
whose conjugate-dual
\[ \begin{CD}
  (f^{\vee})^s: M^s @>>> ((M^s)^{\vee})^s @>\varphi(s)^2>>M 
  \end{CD} \]
  satisfies 
\[  (f^{\vee})^s = b \cdot f \quad \text{with $b = $ the sign of $B$}. \]  
 
We now note:
\vskip 5pt

\begin{lemma} \label{L:BB'2}
Given two such non-degenerate forms $B$ and $B'$ on $M$ with the same sign and preserved by $WD$, there is an automorphism of $M$ which commutes with $WD$ and such  that  $B'(m,n) = B(Tm,Tn)$.
\end{lemma}

\begin{proof}
The proof is similar to that of Lemma \ref{L:BB'}. As before, we may reduce to the following two cases:
\vskip 5pt

\begin{enumerate}
\item[(a)] $M = V \otimes N$ with $N$ irreducible and conjugate-dual, in which case the centralizer of the action of $WD$ is $\GL(V)$;
\vskip 5pt

\item[(b)] $M = (V \otimes P) \oplus (V \otimes (P^s)^{\vee})$ with $P$ irreducible and $P \ncong (P^s)^{\vee}$, in which case the centralizer of the action of $WD$ is $\GL(V) \times \GL(V)$. 
\end{enumerate}
\vskip 5pt

In case (a), if the conjugate-duality of $N$ has sign $b_N$, then giving a $WD$-invariant non-degenerate bilinear form on $M$ of sign $b$ is equivalent to giving a non-degenerate bilinear form on $V$ of sign $b \cdot b_N$, and all such are conjugate under $\GL(V)$. Similarly, in case (b), 
giving a $WD$-invariant non-degenerate bilinear form on $M$ of sign $b$ is equivalent to giving a non-degenerate bilinear form on $V$, and all such are conjugate under $\GL(V) \times \GL(V)$.
\end{proof}
\vskip 5pt

 \vskip 10pt

The isomorphism class of the representation $M^s$ is independent of the 
choice of $s$ in $WD(k_0) - WD$.  If $s' = ts$ is another choice, then 
the map 
$$\begin{array}{rcl}
f & : & M^s \to M^{s'} \\
& & m \mapsto t(m)
\end{array}
$$
is an isomorphism of representations of $WD$.  We denote the isomorphism 
class of $M^s$ and $M^{s'}$ simply by $M^{\sigma}$.  
If $M$ is conjugate-orthogonal or conjugate-symplectic by the pairing $B$  relative to $s$, then it is 
conjugate-orthogonal or conjugate-symplectic by the pairing
$$B' (m, f(n)) = B(m, n)$$ 
relative to $s'$.  In both cases, $M^{\sigma}$ is isomorphic to the dual 
representation $M^\vee$.
\vskip 5pt

If $M$ is conjugate-dual via a pairing $B$ 
with sign $b = \pm 1$, then $\det (M)$ is conjugate-dual with sign = 
$(b)^{\dim(M)}$.  Any conjugate-dual representation of $WD$ of dimension 
1 gives a character $\chi : k^\times \to \CC^\times$ which satisfies 
$\chi^{1+\sigma} = 1$.  Hence $\chi$ is trivial on the subgroup $\NN 
k^\times$, which has index $2$ in $k^\times_0$.  We denote this $1$-dimensional 
representation by $\CC(\chi)$.
 
\begin{lemma} \label{L:1-dim}

The representation $\CC(\chi)$ is  conjugate-orthogonal  if  and  only  if    
$ \chi$   is  trivial  on $ k^\times_0 $, and conjugate-symplectic  if  and 
 only  if $ \chi $ is  nontrivial  on  $ k_0^\times$ but trivial on $\NN k^\times$.

\end{lemma}
\begin{proof}: Since the action of WD on $\CC(\chi)$ factors through the quotient
$W(k)^{ab}$, we may compute with the quotient $W(k/k_0)$ of $W(k_0)$.
The Weil group $W(k/k_0)$ is isomorphic to the 
normalizer of $k^\times$ in the multiplicative group of the quaternion division
algebra over $k_0$ [We1, Appendix III, Theorem 2]. It is therefore
generated by $k^\times$ and $s$, with $s \alpha = \alpha^{\sigma} s$ for 
$\alpha \in k^\times$, and $s^2$ in $k^\times_0$ generating the quotient $k_0^\times 
/ \NN k^\times$. 

If $\chi(s^2) = +1$, then the form $B(z,w) = zw$ on $\CC(\chi)$ is 
conjugate-orthogonal.  If $\chi(s^2) = -1$, then this form is 
conjugate-symplectic. 
\end{proof}
\vskip 10pt

We also note:
\vskip 10pt

\begin{lemma} \label{L:ind-tensor}
\noindent (i) If $M$ is conjugate-dual with sign $b$, then $N = \Ind_{WD(k)}^{WD(k_0)} M$ is selfdual with sign $b$.
\vskip 10pt

\noindent (ii) If $M_i$ is conjugate-dual with sign $b_i$, for $i = 1$ or $2$, then $M_1 \otimes M_2$ is conjugate-dual with sign $b_1 \cdot b_2$.
\end{lemma}
\vskip 5pt

\begin{proof}
For (i), suppose that $M$ is conjugate-dual with respect to a form $B$. As a vector space, $N = M \oplus s^{-1} \cdot M$ for $s \in WD(k_0) \smallsetminus WD(k)$. We define a non-degenerate bilinear form $B_N$ on $N$ by decreeing that $M$ and $s^{-1} \cdot M$ are isotropic spaces and setting 
\[  \begin{cases}
B_N(m, s^{-1} \cdot m') = B(m , m'), \\
B_N(s^{-1} \cdot m', m) = b \cdot B_N(m, s^{-1} \cdot m').
\end{cases}  \]
It is easy to check that $B_N$ is preserved by $WD(k_0)$; this proves (i). The assertion (ii) is also straightforward: if $M_i$ is conjugate-dual with respect to the form $B_i$ of sign $b_i$, then $M_1 \otimes M_2$  is conjugate-dual with respect to the tensor product $B_1 \otimes B_2$ which has sign $b_1 \cdot b_2$. 
\end{proof}

\vskip 15pt

\noindent{\bf Remark:} M. Weissman has pointed out that the representation $M^{\sigma}$ can be more canonically defined (without resorting to the choice of $s \in WD(k) \smallsetminus WD(k_0)$) in the following way. Consider the induced representation $\Ind_{WD(k)}^{WD(k_0)} M$ which can be realized on:
\[  \{ f: WD(k_0) \to M: f(\tau \cdot s) = \tau (f(s)) \, \, 
\text{for all $\tau \in WD(k)$ and $s \in WD(k_0)$} \}. \]
Then the representation $M^{\sigma}$ of $WD(k)$ can be realized on the subspace of such functions which are supported on $WD(k_0) \smallsetminus WD(k)$. We note:
\vskip 5pt
\begin{enumerate}[(i)]
\item Any $WD$-equivariant map $M \to N$ induces a natural $WD$-equivariant map $M^{\sigma} \to N^{\sigma}$. 
\vskip 5pt

\item There is a natural isomorphism $(M^{\sigma})^{\vee} \to (M^{\vee})^{\sigma}$, via the perfect duality on $M^{\sigma} \times (M^{\vee})^{\sigma}$ defined by
\[  (f, f^{\vee}) \mapsto \langle f(s), f'(s) \rangle, \quad \text{for $f \in M^{\sigma}$ and $f^{\vee} \in (M^{\vee})^{\sigma}$}, \]
for any $s \in WD(k_0) \smallsetminus WD(k)$ and where $\langle-,-\rangle$ denotes the natural pairing on $M \times M^{\vee}$. The above pairing is clearly independent of the choice of $s$.
\vskip 5pt

\item On this model of $M^{\sigma}$, there is a canonical isomorphism
\[  (M^{\sigma})^{\sigma} \longrightarrow M \]
given by
\[  F \mapsto F[s](s^{-1}) \]
for any $s \in WD(k_0) \smallsetminus WD(k)$. This isomorphism is independent of the choice of $s$.
\end{enumerate}
\vskip 5pt

Thus a conjugate-duality with sign $b$ is a $WD$-equivariant isomorphism 
\[  f: M^{\sigma} \to M^{\vee} \]
whose conjugate-dual
\[      {^\sigma}f^{\vee}: M^{\sigma} \to  ((M^{\sigma})^{\vee})^{\sigma} \cong ((M^{\vee})^{\sigma})^{\sigma} \cong M^{\vee}\]
satisfies 
\[    {^\sigma}f^{\vee} = b \cdot f. \]
This treatment allows one to suppress the somewhat mysterious looking  identity $B(n,m) = b \cdot B(m, s^2 n)$. 
\vskip 15pt

\section{The centralizer and its group of components}  \label{S:centralizer}

The centralizer $C(M)$ of a representation $M$ of $WD$  is the subgroup of 
$\GL(M)$ which centralizes the image.  Write 
$$M = \bigoplus m_i M_i$$ 
as a direct sum of irreducible representations $M_i$, with 
multiplicities $m_i \ge 1$.  Then by Schur's lemma 
$$C(M)\simeq \prod \GL(m_i,\CC).$$  
In particular, $C(M)$ is a connected reductive group.  
\vskip 5pt
The situation is more interesting for representations $M$ which are 
either selfdual or conjugate-dual, via a pairing $B$ with sign $b = \pm 
1$.  We define $C = C(M,B)$ as the subgroup of $\Aut(M, B) \subset \GL(M)$ which 
centralizes the image of $WD$.  Up to isomorphism, the reductive group $C$ depends only on
the representation $M$.
\vskip 5pt

If we write $M$ as a direct sum of irreducible representations $M_i$, 
with multiplicities $m_i$, and consider their images in $M^\vee$ under the 
isomorphism $M^{\sigma} \to M^\vee$ provided by $B$, we find that there are 
three possibilities:
\begin{enumerate}
\item $M^{\sigma}_i$ is isomorphic to $M^\vee_i$, via a pairing $B_i$ of 
the same sign $b$ as $B$. 
\item $M^{\sigma}_i$ is isomorphic to $M^\vee_i$, via pairing $B_i$ of the 
opposite sign $-b$ as $B$.  In this case the multiplicity $m_i$ is even. 
\item $M^{\sigma}_i$ is isomorphic to $M^\vee_j$, with $j \not= i$.  In 
this case $m_i = m_j$. 
\end{enumerate}
Hence, we have a decomposition
$$M = \bigoplus m_i M_i + \bigoplus 2n_i N_i + \bigoplus p_i(P_i + (P^\sigma_i)^\vee)$$
where the $M_i$ are selfdual or conjugate-dual of the same sign $b$, 
the $N_i$ are selfdual or conjugate-dual of the opposite sign $-b$, 
and $P^{\sigma}_i$ is not isomorphic to $P^\vee_i$, so that $P_i$ and $P_j = 
(P^{\sigma}_i)^\vee$ are distinct irreducible summands. 
\vskip 5pt

In this case, the proofs of Lemmas \ref{L:BB'} and \ref{L:BB'2} give (cf. [GP, $\S$6-7], [P1] and [P3]): 

$$C \simeq \prod \O(m_i,\CC) \times \prod \Sp(2n_i,\CC) \times \prod 
\GL(p_i,\CC).$$ 
In particular, the component group of $C$ is 
$$A = \pi_0(C) \simeq (\Z/2)^k,$$ 
where $k$ is the number of irreducible summands $M_i$ of the same type as 
$M$.  

For each such $M_i$, let $a_i$ be a simple reflection in the orthogonal 
group $\O(m_i)$.  The images of the elements $a_i$ in $A$ give a 
basis over $\Z/2\Z$.  For any element $a$ in $C$, we 
define
$$M^a = \{m \in M:  am = -m\}$$ 
to be the $-1$ eigenspace for $a$ on $M$.  This is a representation of  
$WD$, and the restricted pairing $B : M^a \times M^a \to \CC$ is 
non-degenerate, of the same type as $M$.  For the simple reflections 
$a_i$ in $C$, 
$$M^{a_i} = M_i$$ 
are the irreducible summands of the same type as $M$. 

We can use these representations to define characters $\chi : A \to 
\langle \pm 
1\rangle $.  The basic idea is to define signed invariants $d(M) = \pm 1$ of 
representations $M$ of $WD$, which are either selfdual or conjugate-dual.  

\begin{prop}  \label{P:d(M)}

Let $d(M)$ be an invariant of selfdual or conjugate-dual representations, taking
values in $\pm 1$. Assume that
\begin{enumerate}
\item $d(M+ M') = d(M) \cdot d(M')$
\item the value $d(M^a)$ depends only on the image of $a$ in the quotient group $A = C / 
C^0$.
\end{enumerate}
Then the function $$\chi (a) = d(M^a)$$
defines a character of $A$. 
\end{prop}

Indeed, the different classes in $A$ are all represented by commuting involutions in 
$C$, and for two commuting involutions $a$ and $b$ we have the formula:
$$M^{ab} + 2(M^a \cap M^b) = M^a + M^b$$
as representations of $WD$. Hence $\chi(ab) = \chi(a)\cdot \chi(b)$.

\vskip 10pt

The simplest example of such an invariant, which applies in both the 
conjugate-dual and the selfdual cases, is 
$$d(M) = (-1)^{\dim M}.$$
Since $\dim M^a \pmod{2}$ depends only on the coset of $a \pmod{C^0}$, 
this gives a character of $A$:
$$\eta (a) = (-1)^{\dim M^a}.$$
\vskip 5pt

Now assume $M$ is selfdual.  The character $\eta$ is trivial when $M$ 
is symplectic, as $\dim M^a$ is even for all $a$ in $\Sp(M) = G(M, B)$.  
In the orthogonal case, $\dim M^a$ is even precisely for elements $a$ in 
the centralizer which lie in the subgroup $\SO(M,B)$ of index $2$ in $\O(M,B)$. 
We denote this subgroup by $C^+$.
\vskip 5pt

An element  $c$ of $k^\times/k^{\times 2}$ gives a character 
$$\eta_c(a) = (\det M^a)(c)$$
of $A$. Indeed, the quadratic character $\det M^a$ depends only on the coset of
$a \pmod{C^0}$.  Since $\eta_{cd} = \eta_c \eta_d$, we get a pairing 
$$(c, a) : k^\times /k^{\times 2} \times A \to \langle \pm 1\rangle $$
which is trivial in the symplectic case.  

To construct other characters of $A$, we need  more sophisticated 
signed invariants $d(M)$ of selfdual or conjugate-dual 
representations.   We will obtain these from local root numbers, after 
recalling that theory in the next section. 
\vskip 15pt

\section{Local root numbers}  \label{S:rootnumber}

Let $M$ be a representation of the Weil-Deligne group $WD$ of a local 
field $k$, and let $\psi$ be a nontrivial additive character of $k$.  
In this section, we define the local root number $\ep(M, \psi)$, following 
the articles of Tate [T] and Deligne [De1].  We then study the properties 
of these constants for selfdual and conjugate-dual representations,
and give explicit formulae in the orthogonal and conjugate-orthogonal
cases. The local root numbers are more mysterious in the symplectic
and conjugate-symplectic cases. Indeed, they form the basis of our conjectures
on the restriction of representations of classical groups over local fields.

Let $dx$ be the unique Haar measure on $k$ which is selfdual for 
Fourier transform with respect to $\psi$.  For a representation $M$ of 
the Weil group $W(k)$, we define 
$$\ep(M, \psi) = \ep(M, \psi, dx, 1/2) \ {\rm \ in \ \ }\ \CC^\times,$$
in the notation of [De1, $\S$4-5].  This is the local constant $\ep_L(M, 
\psi)$ in [T, 3.6.1].  In the non-archimedean case, if $M$ is a 
representation of $WD = W(k) \times \SL_2(\CC)$, we may write 
$$M =\sum_{n\geq 0} M_n \otimes {\rm Sym}^n$$
with each $M_n$ a representation of the Weil group.  We define (cf. 
[GR, $\S$2]): 
$$\ep(M, \psi) = {\displaystyle{\prod_{n \geq 0}}} \ep (M_n, \psi)^{n+1} \cdot 
\det (-F|M_n^I)^n.$$
This constant depends only on the isomorphism class of $M$. 

The following formulae involving $\ep(M, \psi)$ are well-known [T, 3.6], for representations
$M$ of the Weil group $W(k)$. For $a$ in $k^\times$, let $\psi_a$ be the nontrivial additive character $\psi_a(x) = 
\psi (ax)$. Then

$$\ep(M, \psi_a) = \det M(a)\cdot \ep(M, 
\psi),$$

$$\ep(M, \psi)\cdot  \ep(M^\vee, \psi^{-1}) =1.$$
Since $\psi^{-1} = \psi_{-1}$, we conclude that

$$\ep(M, \psi) \cdot \ep (M^\vee, \psi) = \det M(-1).$$

\bigskip  
For representations $M = \sum M_n \otimes {\rm Sym}^n$
of $WD$ in the non-archimedean case, we have

$$M^\vee  = \sum M^\vee_n \otimes {\rm Sym}^n,$$

$$\det (M)  =  \prod_{n \geq 0}  \det (M_n)^{n+1} .$$
This allows us to extend the above formulas to the local root numbers
$\ep(M,\psi)$ of representations of $WD$.

\bigskip

Now let $\sigma$ be an involution of $k$, and define $\psi^{\sigma}(x) 
= \psi(x^{\sigma})$. Then $\ep(M^{\sigma}, \psi^{\sigma}) = \ep (M, \psi)$.
If we assume further that $\psi^{\sigma} = 
\psi^{-1}$, then 
$$\begin{array}{rcl}
\ep(M, \psi) \cdot \ep(^{\sigma}M^\vee, \psi) & = & \ep(M, \psi) \cdot \ep(M^\vee,  
\psi^{\sigma}) \\
& = & \ep(M, \psi) \cdot \ep (M^\vee, \psi^{-1})\\
& = & 1. 
\end{array}
$$
When we apply these formulas to selfdual and conjugate-dual representations, we
obtain the following.

\begin{prop}  \label{P:epsilon-WD}
\begin{enumerate}
\item Assume that $M$ is a selfdual representation of $WD$ with $\det (M) = 1$. Then $\ep(M) = \ep(M, 
\psi)$ is independent of the choice of $\psi$ and satisfies 
\[  \ep(M)^2 = 1. \]  
Furthermore, if $M$ is of the form $M = N + N^\vee$, then $\ep(M) = \det N(-1)$. 
\vskip 5pt

\item Assume that  $M$ is a conjugate-dual representation of $WD$ and that the additive character $\psi_0$  of $k$ satisfies $\psi^{\sigma}_0 = 
\psi^{-1}_0$. Then 
\[  \ep(M, \psi_0)^2 = 1.\]
Furthermore, if $M$ is of the form $M = N + {^{\sigma}N}^\vee$, 
then $\ep (M, \psi_0) = 1$. 
\end{enumerate}
\end{prop}

\vskip 5pt

Since we are assuming that the
characteristic of $k$ is not equal to $2$, the characters $\psi_0$ of $k$ which satisfy 
$\psi^{\sigma}_0 = \psi^{-1}_0$ are precisely those characters which are 
trivial on $k_0$, the fixed field of $\sigma$.  These 
characters form a principal homogeneous space for the group $k^\times_0$, and 
the value $\ep(M, \psi_0)$ depends only on the $\NN k^\times$-orbit of 
$\psi_0$.  Indeed $\det M$ is conjugate-dual and hence trivial on 
$\NN k^\times \subset k_0^\times$.  If $\det M$ is conjugate-orthogonal,
 the restriction of $\det M$ to $k_0^\times$ is trivial, and hence the 
value $\ep(M) = \ep(M, \psi_0)$ is independent of the choice of $\psi_0$.
\vskip 10pt

Following Deligne [De2], we can say more when the selfduality or conjugate-duality of $M$ is 
given by a pairing $B$ with sign $b = + 1$. Recall the spin covering of the special orthogonal group, which gives an exact sequence:
$$ 1 \to \Z/2 \to \Spin(M) \to \SO(M) \to 1.$$
\vskip 5pt

\begin{prop} \label{P:orthogonal-ep}
\begin{enumerate}
\item Assume that $M$ is an orthogonal representation and that $\det (M) = 1$.  Then the root number $\ep(M) =\ep(M,\psi)$ is independent of the choice of $\psi$ and satisfies $\ep(M)^2 = 1$. Furthermore $\ep(M) = + 1$ if 
and only if the representation
$\vr : WD  \to \SO(M) $ 
lifts to a homomorphism
$\vr : WD \to \Spin(M)$.
\vskip 5pt
 
\item Assume that $M$ is a conjugate-orthogonal representation and that  $\psi^{\sigma}_0 = 
\psi^{-1}_0$. Then the root number $\ep (M) = \ep(M, \psi_0)$ is 
independent of the choice of $\psi_0$ and satisfies $\ep(M) = +1$.
\end{enumerate}
\end{prop}
 \vskip 5pt
 
 \begin{proof}
The orthogonal case was proved by Deligne [De2]; we note that in our case the characteristic of $k$ is
not equal to $2$. We will deduce the second result for conjugate-orthogonal representations of $W(k)$ from Deligne's formula, combined with the work of Frohlich and Queyrut [FQ]. The extension of the second result to conjugate-orthogonal representations of the Weil-Deligne group $WD$ is then an amusing exercise, which we leave to the reader. 
\vskip 10pt

If $M$ is conjugate-orthogonal and $\dim(M) = 1$, we have seen that $M$ corresponds to a complex character $\chi$ of the group $k^{\times}/k_0^{\times}$. 
By [FQ, Thm 3], we have the formula
\[ \epsilon(\chi, \psi({\rm Tr})) = \chi(e) \]
where $\psi$ is any non-trivial additive character of $k_0$ and $e$ is any nonzero element of $k$ with ${\rm Tr}(e) = 0$ in $k_0$. The element $e$ is well-defined
up to multiplication by $k_0^{\times}$, and $e^2$ is an element of $k_0^{\times}$. If we define 
the additive character $\psi_0$ of $k$ by
\[  \psi_0(x) = \psi({\rm Tr}(ex)) \]
then $(\psi_0)^{\sigma} = (\psi_0)^{-1}$, and
\[  \epsilon(M) = \epsilon(\chi,\psi_0) = \chi(e)^2 = +1. \]
This establishes the formula when M has dimension 1.

 \vskip 10pt

Since the desired formula is additive in the representation $M$ of $W(k)$, and is true when $\dim(M) =1$, we are reduced to the case of conjugate-orthogonal representations $M$ of even dimension. Then $N = {\rm Ind} (M)$ is an orthogonal representation of determinant $1$. Let $\psi$ be a nontrivial additive character of $k_0$; by the inductivity of local epsilon factors in dimension zero [De1]:

$$ \ep(N, \psi)/\ep(P, \psi)^{\dim(M)} = \ep(M, \psi({\rm Tr}))/\ep(\CC, \psi({\rm Tr}))^{\dim(M)}$$
with $\CC$ the trivial representation and $P = {\rm Ind} (\CC)$ the corresponding
induced representation, which is orthogonal of dimension $2$ and determinant $\omega$.
Since $\ep(\CC, \psi({\rm Tr})) = 1$ and
$\ep(P, \psi)^2 = \omega(-1)$, we obtain the formula
$$\ep(N) = \ep(N, \psi) = \omega(-1)^{\dim(M)/2} \cdot \ep(M, \psi({\rm Tr})).$$
On the other hand, we have
$$\ep(M) = \ep(M, \psi_0) = \det(M)(e) \cdot \ep(M, \psi({\rm Tr})),$$
where $e$ is a nonzero element of $k$ with ${\rm Tr}(e) = 0$ in $k_0$.
Hence, to show that $\ep(M) = +1$, we are reduced to proving the formula
$$\ep(N) = \det(M)(e) \cdot \omega(-1)^{\dim(M)/2}$$
for the root number of the orthogonal induced representation $N$. To do this, we
combine Deligne's formula for the orthogonal root number with the following.
\vskip 5pt

\begin{lemma}  \label{P:w_2}
Let $M$ be a conjugate-orthogonal representation of $W(k)$ of even dimension. Then 
$N = {\rm Ind} (M)$ is an orthogonal representation of $W = W(k_0)$ of determinant $1$. 
The homomorphism $\vr : W  \to \SO(N)$ 
lifts to a homomorphism
$\vr : W \to \Spin(N)$ if and only if $\det M(e)\cdot\omega(-1)^{\dim(M)/2} = +1$.
\end{lemma}

\vskip 5pt
\begin{proof}
Let $T$ be the maximal torus in $\SO(N)$ which consists of the rotations $z_i$ in
$n = \dim(M)$ orthogonal planes.  The restriction of the spin covering $\Spin(N)\rightarrow \SO(N)$
to the torus $T$ is the two-fold covering obtained by pulling back the spin
covering $z \rightarrow z^2$ of $\CC^{\times}$ under the map $F(z_1, \cdots, z_n) = \prod z_i$.

The image of the map $\vr: W \to \SO(N)$ lies in the normalizer of the Levi subgroup $\GL(M)$
which fixes the decomposition $N = M + M^{\vee}$ into maximal isotropic dual subspaces.
There is an involution $j$ of $N$ which switches the subspaces $M$ and $M^{\vee}$.
Since $\det(j) = (-1)^n$ and $n = \dim(M)$ is even, this involution
lies in $\SO(N)$. The normalizer of the Levi is the semi-direct product $\GL(M)\cdot j$.
\vskip 5pt

Since $j$ has $n$ 
eigenvalues which are $+1$, and $n$ eigenvalues which are $-1$, if we view this
involution as a product of rotations $(z_1,\cdots, z_n)$ in orthogonal planes, we get 
$n/2$ values $z_i = -1$ and $n/2$ values $z_i = +1$. Hence the involution $j$ lifts
to an element of order $2$ in $\Spin(N)$ if
and only if $n = \dim(M)$ is divisible by $4$. Since the subgroup $\SL(M)$ is simply-connected,
it always lifts to $\Spin(N)$. Hence the spin cover of $\GL(M)$ is obtained by taking the square root
of the determinant of $M$, via the formula for the covering of $T$ given above.
\vskip 5pt

We now consider the homomorphism
$$\phi: W \rightarrow \GL(M)\cdot j$$ 
whose projection to the quotient $\langle j \rangle$ is the quadratic character
$\omega$ of $Gal(k/k_0)$. The determinant of $M$ corresponds to a character
$$\chi: k^{\times}/k_0^{\times} \rightarrow \CC^{\times}$$
by local class field theory. The character $\chi$ of $k^{\times}/k_0^{\times}$ is a square 
if and only if $\chi(e) = +1$, where $e$ is a nonzero element of $k$ with trace zero
to $k_0$. Indeed, $e$ generates the $2$-torsion subgroup of the one dimensional torus $k^{\times}/k_0^{\times}$.  On the other hand,  the quadratic extension $k$ of $k_0$ can be embedded in a cyclic quartic extension of $k_0$ if and only if the character $\omega$ of $k_0^{\times}$ is a square, or equivalently, if and only if $\omega(-1)=1$. 
Hence the parameter  
$$\phi: W \rightarrow \GL(M)\cdot j \rightarrow \SO(N)$$
lifts to  
$\Spin(N)$ if and only if $\chi(e) \cdot \omega(-1)^{\dim(M)/2} = +1$.
\end{proof}
\vskip 5pt

This - together with the extension to representations of $WD(k)$ - completes the proof of Proposition \ref{P:orthogonal-ep}.
\end{proof}

\vskip 15pt

\section{Characters of component groups}  \label{S:char-comp}

In this section, we will use the results of the previous section on local root numbers, together with Proposition \ref{P:d(M)}, to construct characters of the group $A$ of components of the centralizer $C$ of $(M,B)$.  
\vskip 10pt

First assume $M$ and $N$ are conjugate 
selfdual representations, with signs $b(M)$ and $b(N)$.  Fix $\psi_0$ with $\psi_0^{\sigma} = 
\psi_0^{-1}$,  and for $a$ in $C_M \subset G(M, B)$,  define 
$$\chi_N (a) = \ep (M^a \otimes N, \psi_0).$$

\begin{thm}  \label{T:chiN1}
\begin{enumerate}
\item The value $\chi_N(a)$ depends only on the image of $a$ in $A_M$, 
and defines a character $\chi_N : A_M \to \langle  \pm 1\rangle $ 
\vskip 5pt

\item If $b(M)\cdot b(N) = +1$, then $\chi_N = 1$ on $A_M$.
\vskip 5pt

\item If $b(M) \cdot b(N) = -1$, let $\psi'_0 (x) = \psi_0 (tx)$ with $t$ the nontrivial class in 
$k_0^\times / \NN k^\times$, and define 
\[  \chi'_N(a) = \ep (M^a \otimes N, \psi'_0). \]
Then 
\[  \chi'_N= \chi_N \cdot \eta^{\dim(N)}  \in \Hom(A_M, \pm 1),\] 
where the character $\eta$ of $A_M$ is defined by $\eta(a) = (-1)^{\dim M^a}$.
\end{enumerate}
\end{thm}

\begin{proof}
The first statement is proved using the same argument as [GP1, \S 10]. For an element
$a$ in 
$$C_M \simeq \prod \O(m_i,\CC) \times \prod \Sp(2n_i,\CC) \times \prod 
\GL(p_i,\CC),$$ 
we find that $\chi_N(a)$ depends only on the image in the component group.
\vskip 5pt

When $b(M)\cdot b(N) = +1$, the representations $M^a \otimes N$ are all conjugate-orthogonal by Lemma \ref{L:ind-tensor}, so $\chi_N =1$.
\vskip 5pt

The final statement follows from the formula
$$\chi'_N(a) = \chi_N(a) \cdot \det (M^a\otimes N)(t)$$ 
and the calculation of the sign of the conjugate-dual representation
which is the determinant of the tensor product.
\end{proof}
\bigskip

We use this theorem to define the quadratic character
\[ \chi_N(a_M) \cdot \chi_M(a_N) \]
on elements $(a_M,a_N)$ in the component group $A_M \times A_N$. Here $M$ and
$N$ are two conjugate dual representations, although by part 2 of Theorem \ref{T:chiN1} the
character $\chi_N \times \chi_M$ can only be non-trivial when $b(M) \cdot b(N) = -1$.

\bigskip

The case when $M$ and $N$ are selfdual with signs $b(M)$ and $b(N)$ is more complicated.  
First, the function
$$\chi_N (a) = \ep (M^a \otimes N, \psi)$$
on $C_M$ need not take values in $\pm 1$.  Indeed
$$\chi_N(a)^2 = \det (M^a \otimes N) (-1) = \pm 1.$$
Even when $\det (M^a \otimes N) = 1$ for all $a$ in $C_M$, the value $\chi_N(a) 
=\pm 1$ may not be constant on the cosets of $C_M^0$.  For example, 
when $M = P + P^\vee$ with $P$ irreducible and not isomorphic to $P^\vee$ and
$N$ is the trivial representation $\CC$ of dimension $1$, we 
have $C_M = \GL(1,\CC)$.  But 
$$\chi_N(-1) = \ep (M \otimes N, \psi) = \ep (M, \psi) =  \det P(-1),$$ 
which need not be equal to $\chi_N(1) = 1$.  
\vskip 5pt

When $-1$ is a square in $k^\times$, however,  these problems do not arise and $\chi_N$ defines a character of 
$A_M$ as above.  In the general case, we will only consider selfdual representations $M$ and $N$ of even dimension, and elements $a$ in the subgroup $C_M^+$ of $C_M$, where $ \det (a|M) = +1$. Then $\dim (M^a)$ is also even, and 
$$\det(M^a \otimes N) = \det(M^a)^{\dim(N)} \cdot \det(N)^{\dim(M^a)}$$
is clearly trivial. In particular, $\ep (M^a \otimes N, \psi) =  \ep (M^a \otimes N)$ is independent of the choice of  additive character $\psi$ and satisfies
$ \ep (M^a \otimes N)^2 = +1$.
We correct this sign by another
square root of $\det(M^a \otimes N)(-1)$, and define
$$\chi_N (a) = \ep (M^a \otimes N) \cdot \det(M^a)(-1)^{\dim(N)/2} \cdot \det(N)(-1)^{\dim(M^a)/2}.$$

\begin{thm}  \label{T:chiN2}
\begin{enumerate}
\item The value $\chi_N(a)$ depends only on the image of $a$ in $A_M^+$, 
and defines a character $\chi_N : A_M^+ \to \langle  \pm 1\rangle $ 
\vskip 5pt

\item If $b(M)\cdot b(N) = +1$, then $\chi_N(a) = (\det M^a, \det N)$.

\end{enumerate}
\end{thm}
\begin{proof}

Again, this follows from the method of [GP1, \S 10]. We note that $C^0 \subset C^+ \subset C$,
so the component group $A^+$ has index either $1$ or $2$ in $A$.
\end{proof}

\bigskip

We use this theorem to define the quadratic character
$$ \chi_{N}(a_M) \cdot \chi_{M}(a_N)$$

\noindent on elements $(a_M, a_N)$ in the component group $A_{M}^+ \times A_{N}^+$, for two
selfdual representations $M$ and $N$ of even dimension.  As in the conjugate-dual case, it follows from Theorem \ref{T:chiN2}(2) that the character 
$\chi_{N}  \times \chi_{M}$ is only interesting when $b(M)\cdot b(N) = -1$. When the
representations $M$ and $N$ are both symplectic, $\chi = 1$. When $M$ and
$N$ are both orthogonal, the character $\chi$ is given by a product of
Hilbert symbols
\[  \chi_N(a_M) \cdot \chi_M(a_N) = (\det M^{a_M}, \det N)\cdot  (\det M, \det N^{a_N}). \]
\vskip 15pt

\section{$L$-groups of classical groups} \label{S:L-groups}

Having defined representations, selfdual representations, and 
conjugate-dual representations $M$ of the Weil-Deligne 
group $WD$ of $k$, our next goal is to relate these to the Langlands parameters of classical groups. Before doing 
that, we recall the $L$-group attached to each of the classical groups, 
with particular attention to the $L$-groups of unitary groups.

\vskip 10pt
If $G$ is a connected reductive group over $k_0$, the $L$-group of $G$ is a semi-direct product
\[  {^L}G = \widehat{G} \rtimes \Gal(K/k_0) \]
where $\widehat{G}$ is the complex dual group and $K$ is a splitting field for the quasi-split inner form of $G$, with $\Gal(K/k_0)$ acting on $\widehat{G}$ via pinned automorphisms.  For $G = \GL(V)$, 
we have $\widehat{G} = \GL(M)$ with $\dim M = \dim V$.  If $\{ e_1, \cdots, 
e_n\} $ is the basis of the character group of a maximal torus $T \subset 
G$ given by the weights of $V$, then the weights of the dual torus 
$\widehat{T}$ on $M$ are  the dual basis $\{ e'_1, \cdots, e_n'\} $.  
\vskip 10pt

Now assume that $G \subset \GL(V/k)$ is a connected classical group, defined by a 
$\sigma$-sesquilinear form $\langle,\rangle : V \times V \to k$ of sign $\epsilon$. The group $G$ and its dual group $\widehat{G}$, as well as the splitting field $K$ of its quasi-split inner form,  are given by the following table:

\vskip 15pt
\begin{center}
\begin{tabular}{|c|c|c|c|c|}
\hline 
& & & & \\
$(k, \epsilon)$ & $G$ & $\widehat{G}$ & $K$ & ${^L}G$ \\
\hline
& & & & \\
$k = k_0$ & $\SO(V),$  &  $\Sp_{2n}(\CC)$ & $k_0$ &  $\Sp_{2n}(\CC)$ \\

$\epsilon = 1$ & $\dim V = 2n+1$  & & & \\
\hline
& & & & \\
$k = k_0$&  $\SO(V)$, & $\SO_{2n}(\CC)$ & $k_0( \sqrt{\disc(V)})$ & 
$\O_{2n}(\CC)$ \,  ($\disc(V)\notin k^{\times 2}$)  \\
$\epsilon =1$ & $\dim V = 2n$ & &  &   $\SO_{2n}(\CC)$ \, ($\disc(V) \in k^{\times 2}$)  \\
\hline 
& & & & \\
$k = k_0$ & $\Sp(V)$,  &  $\SO_{2n+1}(\CC)$ & $k_0$ & $\SO_{2n+1}(\CC)$ 
\\ 
$\epsilon = -1$ & $\dim V = 2n$ & & & \\
\hline
& & & & \\
$k \ne k_0$ & $\U(V)$,  & $\GL_n(\CC)$ & $k$ & $\GL_n(\CC) \rtimes \Gal(k/k_0)$ \\
$\epsilon = \pm 1$ &$\dim V = n$  & & & $=$ $L$-group of compact $\U(n)$ \\
\hline
  \end{tabular}
\end{center}

\vskip 20pt
From the above table, we note that the 
$L$-group  $\GL(M)\rtimes \Gal(k/k_0)$ of $\U(V)$ is isomorphic 
to the $L$-group of the anisotropic real group $\U(n)$, associated to a definite hermitian 
space of the same dimension as $V$ over $\CC$.

Before specializing to the case of unitary groups, let us make some general observations on the 
representation theory of the $L$-groups of anisotropic groups over 
$\RR$. In this case, we have 
\[ ^L  G = 
\widehat{G} \rtimes \Gal (\CC/\RR), \]
where the Galois group acts by a pinned 
involution (possibly trivial), which maps to the opposition involution 
in ${\rm Out}(\widehat{G})$ and takes any representation $V$ to its dual $V^\vee$.  
\vskip 10pt

The pinning of $\widehat{G}$ gives a principal $\SL_2$ in $\widehat{G}$, 
which is fixed by the action of $\Gal(\CC/\RR)$, so that we have 
\[ 
\delta : \SL_2 \times \Gal(\CC/\RR) \to \widehat{G} \rtimes \Gal (\CC/\RR). \]
Let $\ep$ in $Z(\widehat{G})^{\Gal (\CC/\RR)}$ be 
the image of $-I$ in $\SL_2$.  Then $\ep^2 =1 $ and $\ep$ acts as a scalar 
on every irreducible representation of $^LG$.  The following result is due
to Deligne.

\vskip 5pt

\begin{prop}Let $G$ be an anisotropic group over $\RR$.
Then every complex representation $N$ of $^LG$ is selfdual.  
If $N$ is irreducible, the $^LG$-invariant pairing $\langle ,\rangle_N : N 
\times N \to \CC$ is unique up to scaling, and is $(\ep|N)$-symmetric.
\end{prop}

\begin{proof} The proof is similar to Bourbaki [B, Ch. VIII, \S 7, Prop. 12 and Ex. 6], which treats the irreducible, 
selfdual representations $M$ of $\widehat{G}$.  There one restricts the 
pairing on $M$ to the highest weight summand for the principal $\SL_2$, 
which occurs with multiplicity one.  In this case, we restrict the pairing 
on $N$ to the subgroup $\SL_2 \times \Gal(\CC/\RR)$, which again has an 
irreducible summand which occurs with multiplicity $1$.  The sign 
$(\ep|N)$ is the sign of $-I$ on this summand, which determines the sign 
of the pairing. 
\end{proof}

If $M$ is an irreducible representation of $\widehat{G}$ which is 
selfdual, then $M$ extends in two ways to $^{L } G$.  If $M$ is 
{\it not} isomorphic to $M^{\vee}$, then the induced representation
$$N = \Ind (M)$$
of $^{L}G$ is irreducible.  If we restrict $N$ to $\widehat{G}$, it 
decomposes as a direct sum 
$$M + \alpha \cdot M \simeq M + M^{\vee}$$
where $\alpha$ generates $\Gal (\CC/\RR)$.  Since $M$ is not isomorphic 
to $M^{\vee}$, the subspaces $M$ and $\alpha \cdot M$ of $N$ are isotropic for the 
pairing $\langle ,\rangle _N$, which gives a 
non-degenerate pairing $\langle ,\rangle_M  : M \times M
\to \CC$ defined by 
\begin{equation} \label{E:real}
\langle m, m'\rangle_M  = \langle m ,  \alpha m'\rangle _N.
\end{equation}
This is a conjugate-duality on $M : $ for $g$ in $\widehat{G}$, we have
$$\begin{cases}
\langle gm, \alpha g \alpha^{-1} m'\rangle_M   =  \langle m, m'\rangle_M  \\
\langle m', m\rangle_M   =  (\ep|M) \cdot \langle m, m'\rangle_M 
\end{cases}$$

\bigskip

We now specialize these arguments to the case of  
unitary groups.  Since the representation of the principal $\SL_2 \to 
\GL(M)$ on $M$ is irreducible and isomorphic to ${\rm Sym}^{n-1}$, we have
$$(\ep|M) = (-1)^{n-1}.$$
Hence the selfduality on $N$ and the conjugate-duality on $M$ are 
$(-1)^{n-1}$-symmetric. In particular, we have
\vskip 10pt

\begin{prop} \label{P:normalizer}
If $G = \U(V)$, with $\dim_k V = n$, then
\[  {^L}G \hookrightarrow \begin{cases}
\Sp(N) &= \Sp_{2n}(\CC), \text{  if $n$ is even;} \\
\O(N) &= \O_{2n}(\CC), \text{  if $n$ is odd.} \end{cases} \]
In each case, ${^L}G$ is identified with the normalizer of a Levi subgroup of a Siegel parabolic subgroup in $\Sp_{2n}(\CC)$ or $\O_{2n}(\CC)$.
\end{prop}

Finally, it is instructive to describe the $L$-groups of the classical groups from the point of view of invariant theory. As one observes from the above table, the $L$-groups of symplectic and orthogonal groups $G(V)$ are themselves classical groups over $\CC$ and have natural realizations as  subgroups of $\GL(M)$ for complex vector spaces $M$ of appropriate dimensions. These subgroups can be described as follows. One has a decomposition 
\[  M \otimes M \cong \Sym^2 M \bigoplus \bigwedge^2 M \]
of $\GL(M)$-modules. The action of $\GL(M)$ on $\Sym^2 M$ or 
$\wedge^2 M$ has a unique open orbit consisting of 
 nondegenerate symmetric or skew-symmetric forms on $M^{\vee}$. Then 
we note:
\vskip 5pt

\begin{enumerate}[(i)]
\item The stabilizer of a nondegenerate vector 
$B$ in $\Sym^2 M$ (resp. $\wedge^2 M$) is the orthogonal group 
$\O(M,B)$  (resp. the symplectic group $\Sp(M,B)$); these groups 
exhaust the $L$-groups of symplectic and orthogonal groups.  
\vskip 5pt

\item The action of this stabilizer on the other representation $\wedge^2 M$ (resp. $\Sym^2 M$) is its adjoint representation. 
\vskip 5pt

\item The two representations $\Sym^2 M$ and $\wedge^2 M$ are also useful for characterizing the selfdual representations of $WD(k)$ introduced in \S \ref{S:selfdual}: a representation $M$ of $WD(k)$ is
orthogonal (resp. symplectic) if and only if $WD(k)$ fixes a nondegenerate vector in $\Sym^2 M$ (resp. $\wedge^2 M$).
\end{enumerate}
\vskip 10pt

These rather obvious remarks have analogs  
for the unitary group $\U(V)$, which we now describe. Suppose that the $L$-group of $\U(V)$ is $\GL(M) \rtimes \Gal(k/k_0)$.  Consider the semi-direct product 
\[  H  = (\GL(M) \times \GL(M)) \rtimes \Z/2\Z \]
where $\Z/2\Z$ acts by permuting the two factors of $\GL(M)$; this is the $L$-group of $\text{Res}_{k/k_0}(\GL(V/k))$ with $\dim_k V = \dim M$. The irreducible representation $M \boxtimes M$ of $H^0 = \GL(M) \times \GL(M)$ is invariant under $\Z/2\Z$ and thus has two extensions to $H$. In one such extension, the group $\Z/2\Z = S_2$ simply acts by permuting the two copies of $M$; the other extension is then given by 
twisting by the nontrivial character of $H/H^0$.
In honor of Asai, we denote these two extensions by ${\rm As}^{+}(M)$ and ${\rm As}^-(M)$ respectively. They can be distinguished by
\[  \text{Trace}(c| {\rm As}^+(M)) = \dim M \quad \text{and}\quad 
 \text{Trace}(c| {\rm As}^-(M)) = -\dim M, \]
where $c$ is the nontrivial element in $\Z/2\Z$. One has
\[  \text{Ind}_{H^0}^H (M \boxtimes M) = {\rm As}^+(M) \bigoplus {\rm As}^-(M). \]
\vskip 10pt

The action of $H^0$ on ${\rm As}^{\pm}(M)$ has an open dense orbit, consisting of isomorphisms $M^{\vee} \to M$ and whose elements we call nondegenerate. Now we have
\vskip 5pt

\begin{prop} \label{P:asai}
If $\dim M = n$, then the stabilizer in $\GL(M) \times \GL(M)$ of a nondegenerate vector in ${\rm As}^{(-1)^{n-1}}(M)$ is isomorphic to the $L$-group of $\U(V)$.  Moreover, the action of this stabilizer on the other representation ${\rm As}^{(-1)^n}(M)$ is the adjoint representation 
of ${^L}U(V)$. 
\end{prop}

\vskip 10pt

The representations ${\rm As}^{\pm}(M)$ are also useful for characterizing conjugate-dual representations of $WD$, which were discussed in \S \ref{S:selfdual}. Indeed, given a representation 
\[  \varphi: WD(k) \to \GL(M), \]
one obtains a map
\[  \tilde{\varphi}: WD(k_0) \to H \]
by setting
\[  \tilde{\varphi}(\tau) = (\varphi(\tau),  \varphi(s \tau s^{-1})) \in \GL(M) \times \GL(M),  \]
for $\tau \in WD(k)$, and
\[  \tilde{\varphi}(s) =(1,\varphi(s^2)) \cdot c \in H \smallsetminus H^0. \] 
The choice of $s$ is unimportant, since the maps $\tilde{\varphi}$'s 
thus obtained for different choices of $s$ are naturally conjugate under 
$H^0$. Through this map, $WD(k_0)$ acts on 
${\rm As}^{\pm}(M)$. In fact, the representation 
${\rm As}^+(M)$ of $WD(k_0)$ is obtained from $M$ by the process of 
{\it multiplicative induction} 
[P2] or {\it twisted tensor product}; 
it is an extension of the representation $M \otimes M^{\sigma}$ of $WD(k)$ to 
$WD(k_0)$, and ${\rm As}^-(M)$ is the twist of ${\rm As}^+(M)$ by the quadratic 
character $\omega_{k/k_0}$ associated to the quadratic extension $k/k_0$.
\vskip 10pt

Now we have:
 \vskip 5pt
 
\begin{prop}  \label{P:cd-invariant}
If $M$ is a representation of $WD(k)$, then 

\begin{enumerate}[(i)]
\item  $M$ is conjugate-orthogonal if and only if $WD(k_0)$ fixes a nondegenerate vector in
${\rm As}^+(M)$. When $M$ is irreducible, this is equivalent to ${\rm As}^+(M)^{WD(k_0)} \ne 0$.
 \vskip 5pt
 
 \item $M$ is conjugate-symplectic if and only if $WD(k_0)$ fixes a nondegenerate vector in ${\rm As}^-(M)$. When $M$ is irreducible, this is equivalent to ${\rm As}^-(M)^{WD(k_0)} \ne 0$.  
 \end{enumerate}
 \end{prop}
\vskip 10pt

\section{Langlands parameters for classical groups}  \label{S:L-parameter}

In this section, we discuss the Langlands parameters of 
classical groups. In particular, we show that these Langlands parameters can be understood in terms of selfdual or conjugate-dual representations $M$ of $WD(k)$. 
\vskip 5pt

If $G$ is a connected, reductive group over $k_0$, a Langlands parameter 
is a homomorphism 
$$\vr : WD(k_0) \to {^L}G = \widehat{G} \rtimes  \Gal (K/k_0).$$
This homomorphism is
required to be continuous on $WD = W(k_0)$ when $k_0 = \RR$ on $\CC$.  In 
the non-archimedean case, it is required to be trivial on an open subgroup of the inertia group, and the image of Frobenius is required to be semi-simple.  In all cases, the 
projection onto $\Gal(K/k_0)$ is the natural map $W(k_0)/W(K)\to \Gal (K/k_0)$.  
Finally, two Langlands parameters are considered equivalent if they are conjugate 
by an element in $\widehat{G}$.  
\vskip 5pt

Associated to any Langlands parameter 
is the reductive group 
$$C_{\vr} \subset \widehat{G}$$
which centralizes the image, and its component group 
$$A_{\vr} = C_{\vr} / C_{\vr}^0.$$
The isomorphism class of both $C_{\vr}$ and $A_{\vr}$ are determined
by the equivalence class of the parameter $\vr$.
\vskip 5pt

For $G = \GL(V/k_0)$, we have $\widehat{G} = \GL(M)$ with $\dim M = \dim V$.  If $\langle e_1, \cdots, 
e_n\rangle $ is the basis of the character group of a maximal torus $T \subset 
G$ given by the weights of $V$, then the weights of the dual torus 
$\widehat{T}$ on $M$ are  the dual basis 
$\langle e'_1, \cdots, e_n'\rangle $.  The Langlands parameters for $G$ are simply equivalence classes of representations of $WD(k_0)$ on $M$.
\vskip 10pt

Now assume that $G \subset \GL(V/k)$ is a connected classical group, defined by a $\sigma$-sesquilinear form $\langle , \rangle : V \times V \to k$ of sign $\epsilon$. 
We will see that for each classical group $G$, a Langlands 
parameter $\vr$ for $G$ corresponds to a natural complex representation  $$WD(k)\to \GL(M)$$ 
with some additional structure, as given in the following theorem. 
\vskip 10pt

\begin{thm}  \label{T:classical-par}
(i) A Langlands parameter $\vr$ of the connected 
classical group $G \subset \GL(V/k)$ determines a selfdual or 
conjugate-dual representation $M$ of $WD(k)$, with the following structure:

\vskip 15pt
\begin{center}
\begin{tabular}{|c|c|c|c|c|}
\hline 
& & & & \\
 $G$ & $\dim(V)$ & $M$ & $\dim M$  &   \\
\hline
& & & & \\
 $\Sp(V)$  &  $2n$ & {\rm Orthogonal} &  $2n+1$ & $\det M =1$ \\
\hline
& & & & \\
 $\SO(V)$ &  $2n+1$ & {\rm Symplectic} & $2n$ &  \\
\hline
& & & & \\
$\SO(V)$ & $2n$ & {\rm Orthogonal} & 2n & $\det M = \disc \,V$  \\ 
\hline 
& & & & \\
 $\U(V)$  &  $2n+1$  & {\rm Conjugate-orthogonal} & $2n+1$ &  
\\ 
\hline
& & & & \\
 $\U(V)$  & $2n$ & {\rm Conjugate-symplectic} & $2n$ & \\
\hline
  \end{tabular}
\end{center}

\vskip 20pt

\noindent (ii) The isomorphism  class of the representation $M$ determines the 
equivalence class of the parameter $\vr$, except in the case when  $M$ is
orthogonal and every irreducible orthogonal summand $M_i$ of $M$ has even dimension.
In the exceptional case, $M$ and $V$ have even dimension and there are two equivalence classes $\{\vr, \vr'\}$ of parameters for $SO(V)$ which give rise to the same orthogonal 
representation $M$. 
\vskip 10pt

\noindent (iii) In the unitary cases, the group $C_{\vr} \subset \widehat{G}$  which centralizes 
the image of $\vr$ is isomorphic to the group $C$ of elements $a$ in 
${\rm Aut}(M, B)$ which centralize the image $WD(k) \to \GL(M)$. In the orthogonal
and symplectic cases, the group $C_{\vr} \subset \widehat{G}$  is isomorphic to the
subgroup $C^+$ of $C$, consisting of
those elements which satisfy $det (a|M) = 1$.

\end{thm}

\begin{proof}
This is well-known if $G$ is an orthogonal or symplectic 
group. Indeed, the $L$-group ${^L}G$ is essentially the automorphism
group of a nondegenerate symmetric or skew-symmetric bilinear form $B$ on a complex vector space $M$ of appropriate dimension. So the theorem amounts to the assertion that if two homomorphisms $WD(k) \to \text{Aut}(M,B) \subset \GL(M)$ are conjugate in $\GL(M)$, then they are conjugate in $\text{Aut}(M,B)$. This
is the content of Lemma \ref{L:BB'}. Moreover, the description of the component group $C_{\varphi}$ follows directly from the results of \S \ref{S:centralizer}. 
\vskip 10pt

Henceforth we shall focus on the unitary case.
For $G = \U(V)$, a parameter is a homomorphism 
$$\vr : WD (k_0) \to \GL(M) \rtimes \Gal (k/k_0)$$ 
with $\dim M = \dim V= n$.  Again, we normalize $M$ so that the weights 
for $\widehat{T}$ on $M$ are dual to the weights for the torus $T = 
\U(1)^n$ on $V$.  The restriction of $\vr$ to $WD(k)$ gives the desired 
representation $M$.  We must next show that $M$ is conjugate-dual with sign $(-1)^{n-1}$.
\vskip 5pt

If $s$ generates the quotient $WD(k_0)/WD(k)$, then 
$$\vr(s) = (A, \alpha) = (A, 1) (1, \alpha) \ {\rm in} \ 
^{L}G$$
with $A$ in $\GL(M)$.  In the previous section (cf. equation (\ref{E:real})), we have seen that
the standard representation $M$ of $\GL(M)$ has a conjugate-duality 
 $\langle-, -\rangle_M$ of sign $(-1)^{n-1}$ with respect to the nontrivial element $\alpha \in \Gal(k/k_0) \subset {^L}G$. 
We define the bilinear form 
$$B(m,m') = B_s (m, m') = \langle m, A^{-1}m'\rangle_M, $$
Then the form $B$ is
non-degenerate on $M$ and satisfies
$$B(\tau m, s \tau s^{-1} m') = B(m, m')$$ 
for all $\tau$ in $WD$, and 
$$B(m', m) = (-1)^{n-1} \cdot B(m, s^2 m').$$ 
Hence $M$ is conjugate-dual with $\sign = (-1)^{n-1}$. 
\vskip 10pt

It is clear that the conjugation of a parameter $\vr$ by an element of 
$\GL(M)$ gives an isomorphism of the associated conjugate-dual 
representations.  Hence we are reduced to showing that every conjugate 
dual  representation $M$ of $\sign = (-1)^{n-1}$ extends to a Langlands 
parameter $\vr$ of $WD(k_0)$, and that the isomorphism class of $M$
determines the equivalence class of $\vr$.
\vskip 10pt

Suppose then that $M$ is a conjugate-dual representation of $WD(k)$ of $\sign 
= (-1)^{n-1}$ with $n = \dim M$. To obtain an extension, consider the induced representation 
$N = \Ind (M)$ of $WD(k_0)$.  By Lemma \ref{L:ind-tensor}(i), $N$ is 
selfdual of $\sign = (-1)^{n-1}$.
Moreover, the proof of Lemma \ref{L:ind-tensor} shows that
the image of $WD(k_0)$ in $\Sp(4d)$ or $\O(4d+2)$ (depending on whether $n = 2d$ or $n 
= 2d+1)$ is contained in the normalizer of a Levi subgroup in a Siegel parabolic 
subgroup.  By Proposition \ref{P:normalizer}, this normalizer is isomorphic to the $L$-group of $\U(V)$: it splits as a semi-direct product $\GL(M) \rtimes  \langle \alpha\rangle $, with $\det (\alpha|N) = (-1)^n$. Thus, we have produced an $L$-parameter for $\U(V)$ whose restriction to $WD(k)$ is the given $M$.
\vskip 10pt

Finally we need to show that the extension obtained above is unique, up to conjugacy by $\widehat{G}$.  If $\vr$ and $\vr'$ are two 
parameters extending $\rho : WD(k) \to \GL(M)$, we must show that the 
elements 
$$\begin{array}{rcl}
\vr(s) & = & (A, \alpha) \\ 
\vr'(s) & = & (A', \alpha) \ {\rm of} \ ^{L}G
\end{array}$$
are conjugate by an element of $\GL(M)$ centralizing the image of $\rho$.  
The bilinear forms
$$\begin{array}{rcl}
B(m,m') & = & \langle  m, A^{-1} m'\rangle_M  \\
B'(m,m') & = & \langle  m, (A')^{-1} m'\rangle_M  
\end{array}$$
give two conjugate-dualities of $M$ which are preserved by $WD(k)$ and non-degenerate with sign 
$(-1)^{n-1}$.  By Lemma \ref{L:BB'2},  there is an element $T$ in $\GL(M)$ centralizing the image 
of $\rho$ with 
$$B' (m, m') = B(Tm, Tm').$$
This gives the identity 
$$\langle m, (A')^{-1}m'\rangle  = \langle  m, (T^{-1})^{\alpha} A^{-1} T m'\rangle $$ 
for all $m$ and $m'$.  Hence
$$A' = T^{-1} A T^{\alpha}$$
and the elements are conjugate by the element $T \in \GL(M) = \widehat{G}$.  
\vskip 5pt

The argument identifying the group $C_{\vr}$ with either the group $C$ or $C^+$ of $(M,B)$ is contained in \S \ref{S:centralizer}. This completes the proof of the theorem.

\end{proof}

\begin{cor} \label{C:unitary}
A representation $M$ of $WD(k)$ gives rise to a Langlands parameter for a quasi-split unitary group $\U(V)$ if and only if $WD(k_0)$ fixes a non-degenerate vector in ${\rm As}^{(-1)^{n-1}}(M)$, with $n = \dim M$. 
\end{cor}

\begin{proof}
This is a consequence of the theorem and Proposition \ref{P:cd-invariant}.
\end{proof}
\vskip 10pt

\noindent{\bf Remark:} 
 When $M$ is orthogonal of even dimension, it is often convenient 
to view it as defining a unique Langlands parameter for the full 
orthogonal group $\O(V)$ (which is not connected), with the equivalence being defined by $\O(M)$-conjugacy; see [P1].

\bigskip

We obtain the following simple invariants of the representation $M$.  For $G = \GL(V)$
we have the character
$$\det M : k^\times \to \CC^\times.$$
For $G = \U(V)$,  we obtain the character 
$$\det M : k^\times / k^\times_0 \to \CC^\times$$
as the sign of $\det M$ is $(-1)^{n(n-1)} = +1$.  Finally, for $G = \Sp(V)$ or $G=\SO(V)$ with
$\dim(V)$ even and $\disc(V) = 1$, the representation $M$ is orthogonal with $ \det (M) = 1$.
Hence we have the root number $\ep(M) = \ep(M,\psi)$ independent of the additive character $\psi$ of $k$,
and 
$$\ep (M) = \pm 1.$$ 
We will relate these invariants to the central characters of certain representations of $G$ after introducing Vogan $L$-packets in the next section.
\vskip 15pt

\section{Vogan $L$-packets - Desiderata}  \label{S:vogan-desiderata}

Let $G$ be a quasi-split, connected, reductive group over a local field 
$k_0$. In this section, we will discuss Langlands parameters $\vr$ as the
conjectural parameters for the isomorphism classes of irreducible smooth admissible
complex representations of the locally compact group $G(k_0)$. 
Before coming to that, we briefly recall the notions of smooth and  admissible representations of $G(k_0)$ when $k_0$ is a local field.
\vskip 10pt

When $k_0$ is local and discretely valued, a smooth representation $\pi$ is simply  a homomorphism 
\[  \pi: G(k_0) \longrightarrow \GL(E) \]
for a complex vector space $E$ (possibly infinite-dimensional) such that
\[  E = \cup_K E^K,\]
where the union is over all open compact subgroups $K$ of $G(k_0)$.
Such a smooth representation is admissible if
$E^K$ is finite dimensional for any open compact subgroup $K$. 
A homomorphism from $(\pi,E)$ to $(\pi', E')$ is simply a linear map $E \longrightarrow E'$ which commutes with the action of $G(k_0)$. 
\vskip 10pt

When $k_0 = \R$ or $\CC$, we will consider the category of smooth Frechet representations $(\pi,E)$ of moderate growth, as introduced 
by Casselman [C] and Wallach [W1]. An admissible representation is such a representation whose subspace of 
$K$-finite vectors (where $K$ is a maximal compact subgroup of $G(k_0)$) is the direct sum of 
irreducible representations of $K$ with finite multiplicities. 
A homomorphism $(\pi,E) \longrightarrow (\pi',E')$ is a continuous linear map $E \longrightarrow E'$ which commutes with the action of $G(k_0)$.
\vskip 10pt

We come now to the local Langlands conjecture.
We shall present this conjecture in a form proposed by Vogan [Vo], which treats representations $\pi$ of
all pure inner forms $G'$ of $G$ simultaneously.
\vskip 10pt

All of the pure inner forms $G'$ of $G$ have the same center $Z$ over $k_0$, and
each irreducible representation $\pi$ of $G'(k_0)$ has a central character
$$\omega_{\pi} : Z(k_0) \to \CC^{\times}.$$
The adjoint group $G'_{ad}(k_0)$ acts on $G'(k_0)$ by conjugation, and
hence acts on the set of its irreducible complex representations. The quotient
group $G'_{ad}(k_0)/{\rm Im~} G'(k_0)$ acts on the set of isomorphism classes of
representations of $G'(k_0)$. This quotient is abelian, and canonically isomorphic to the
cohomology group
$$E = {\rm ker} (H^1(k_0,Z) \to H^1(k_0,G')).$$
\vskip 5pt

Let $B$ be a Borel subgroup of $G$ over $k_0$, with 
unipotent radical $N$.  The quotient torus $T = B/N$ acts on the group 
$\Hom(N, \CC^\times)$.  We call a character $\theta : N(k_0) \to \CC^\times$ {\it 
generic} if its stabilizer in $T(k_0)$ is equal to the center $Z(k_0)$.  If 
$\pi$ is an irreducible representation of $G(k_0)$ and $\theta$ is a generic 
character, then the complex vector space $\Hom_{N(k_0)} (\pi, \theta)$ has 
dimension $\le 1$.  When the dimension is $1$, we say 
$\pi$ is $\theta$-generic.  This depends only on the $T(k_0)$-orbit of 
$\theta$.  
\vskip 5pt

When $Z = 1$, the group $T(k_0)$ acts simply-transitively on the set of 
generic characters.  In general, the set $D$ of $T(k_0)$-orbits on the set of all generic 
characters $\theta$ of $N(k_0)$ forms a principal homogeneous space for the abelian group $E$,
which is also isomorphic to 
$$T_{ad}(k_0)/{\rm Im~} T(k_0) = {\rm ker}(H^1(k_0, Z) \to H^1(k_0, T)).$$
\vskip 5pt

We are now ready to describe the desiderata for Vogan $L$-packets, which
will be assumed in the rest of this paper. These properties are known to hold
for the groups $\GL(V)$ and $\SL(V)$, as well as for some classical groups of
small rank.

\bigskip

\begin{enumerate}

\item Every irreducible representation $\pi$ of $G'$ (up to isomorphism)
determines a Langlands parameter 
$$\vr : WD(k_0) \to \widehat{G} \rtimes \Gal (K/k_0)$$
(up to equivalence).

Each Langlands parameter $\vr$ for $G$ corresponds to a finite set $\Pi_{\vr}$  of
irreducible representations of $G(k_0)$ and its pure inner forms $G'(k_0)$. Moreover,
the cardinality of the finite set $\Pi_{\vr}$ is equal to the number of irreducible
representations $\chi$ of the finite group $A_{\vr} = \pi_0(C_{\vr})$.

\bigskip

\item Each choice of a $T(k_0)$-orbit $\theta$ of generic characters for $G$ gives a bijection
of finite sets 

$$J(\theta): \Pi_{\vr} \to {\rm Irr}(A_{\vr}).$$  

The $L$-packet $\Pi_{\vr}$ contains at most one $\theta'$-generic representation, for each $T$-orbit of generic characters $\theta'$ of $G$. It contains such a representation precisely when the adjoint $L$-function of $\vr$ is regular at the point $s=1$. In this case, we say the $L$-packet $\Pi_{\vr}$ is generic.

Assume that the $L$-packet $\Pi_{\vr}$ is generic. In the bijection $J(\theta)$, the unique $\theta$-generic representation $\pi$ in $\Pi_{\vr}$ corresponds to the trivial representation of $A_{\vr}$. The $\theta'$-generic representations correspond to the one dimensional representations $\eta_g$ described below.

\bigskip

\item The finite set $\Pi_{\vr}$  of irreducible representations is stable under the adjoint action, which permutes the different generic representations for $G$ in an $L$-packet transitively.

In any of the bijections $J(\theta)$, the action of $g \in G'_{ad}(k_0)$ on ${\rm Irr}(A_{\vr})$ is given by tensor product with the one-dimensional representation $\eta_g$ of $A_{\vr}$ alluded to above.

More precisely, Tate local duality gives a perfect pairing 
$$H^1(k_0, Z) \times H^1(K/k_0, \pi_1(\wg))\to \CC^\times.$$
The coboundary map $C_{\vr} \to H^1(K/k_0, \pi_1(\wg))$ factors through 
the quotient $A_{\vr}$, and gives a pairing 
$$H^1(k_0, Z) \times A_{\vr} \to \CC^\times.$$
The adjoint action by the element $g$ in $G'_{ad}(k_0) \to  H^1(k_0, Z)$, viewed as a 
one dimensional  representation $\eta_g$ of $A_{\vr}$, will take $\pi(\vr, 
\chi)$ to the representation $\pi(\vr, \chi \otimes \eta_g)$.

\bigskip

\item In any of the bijections $J(\theta)$, the pure inner form which acts on the representation with parameter $(\vr,\chi)$ is constrained by the restriction of the irreducible representation $\chi$ to the image of the group $\pi_0(Z(\wg))^{\Gal (K/k_0)}$ in $A_{\vr}$. 

More precisely, when $k_0 \not= \RR$, Kottwitz has 
identified the pointed set $H^1(k_0,G)$ with the group of characters of the component group 
of $Z(\wg)^{\Gal (K/k_0)}$.  The inclusion 
$$Z(\wg)^{\Gal(K/k_0)} \to C_{\vr}$$
induces a map on component groups, whose image is central in $A_{\vr}$.  
Hence an irreducible representation $\chi$ of $A_{\vr}$  has a central 
character on $\pi_0(Z(\wg))^{\Gal (K/k_0)}$, and determines a class in 
$H^1(k_0, G)$.  This is the pure inner form $G'$ that acts on the 
representation $\pi(\vr, \chi)$.

\bigskip

\item All of the irreducible representations $\pi$ in $\Pi_{\vr}$ have the same central character $\omega_{\pi}$. This character is determined by $\vr$, using the recipe in [GR, \S 7].

\end{enumerate}

\vskip 15pt

\section{Vogan $L$-packets for the classical groups}  \label{S:vogan-classical}

We now make the desiderata of Vogan $L$-packets completely explicit for the classical 
groups $G \subset \GL(V/k)$. We have already described the Langlands parameters $\vr$ for
$G$ explicitly, as certain representations $M$ of $WD =WD(k)$, in section 6. In all cases, the component group $A_{\vr}$ is an elementary abelian $2$-group, so $\rm Irr(A_{\vr}) = \rm Hom(A_{\vr}, \pm 1)$. We treat each family of groups in turn.
\vskip 15pt

\noindent{\bf \underline{The General Linear Group $G = \GL(V)$}}

\bigskip

\begin{enumerate}

\item A Langlands parameter is a representation $M$ of $WD$, with $\dim(M) = \dim(V)$. The
group $C_{\vr} = C(M)$ is connected, so $A_{\vr} = 1$. Hence $\Pi_{\vr}$ consists of a single element.
In this case, the full Langlands conjecture is known (by [HT] and [He]).

\bigskip

\item There is a unique $T$-orbit on the generic characters, and the regularity of the adjoint $L$-function of $\vr$ at $s=1$ detects generic $L$-packets $\Pi_{\vr}$.

\bigskip

\item The adjoint action is trivial, as the center $Z$ of $G$ has trivial first cohomology.

\bigskip

\item The only pure inner form is $G = \GL(V)$.

\bigskip

\item The center $Z(k) = k^{\times}$, and the central character of $\pi(\vr)$ has parameter $\det(M)$. 

\end{enumerate}

\bigskip

\noindent{\bf \underline{The Symplectic Group $G = \Sp(V)$}}

\bigskip

\begin{enumerate}

\item A Langlands parameter is an orthogonal representation $M$ of $WD$, with 
\[ \text{$\dim(M) = \dim(V) + 1$ and $\det(M) = 1$.} \]
The group $A_{\vr} = A_M^+$ has order $2^{m-1}$, where $m$ is the number of distinct irreducible orthogonal summands $M_i$ in $M$. The full Langlands conjecture is known when $\dim(V)$ = $2$ [LL] or $4$ [GT2].

\bigskip

\item The set $D$ of $T$-orbits on generic characters is a principal homogeneous space for the group $E = H^1(k,Z) = k^{\times}/k^{\times 2}$. We will see in \S \ref{S:nu} that the choice of the symplectic space $V$ identifies the set $D$ with the set of $k^{\times 2}$-orbits on the nontrivial additive characters $\psi$ of $k$.

\bigskip

\item The adjoint action is via elements $c$ in the group $E = k^{\times}/k^{\times 2}$. This acts on the irreducible representations of $A_{\vr}$ via tensor product with the character $\eta_c(a) = \det (M^a)(c)$, and on the set $D$ of orbits of generic characters by mapping $\psi(x)$ to $\psi(cx)$.

\bigskip

\item The only pure inner form is $G = \Sp(V)$.

\bigskip

\item The center $Z(k) = \langle\pm1\rangle$, and the central character of $\pi(\vr)$ maps the element $-1$ in $Z(k)$ to the local root number $\ep(M)$. 

\end{enumerate}

\bigskip

\noindent{\bf \underline{The Odd Special Orthogonal Group $G = \SO(V)$, $\dim(V) = 2n+1$}}

\bigskip

\begin{enumerate}

\item A Langlands parameter is a symplectic representation $M$ of $WD$, with $\dim(M) = \dim(V) - 1$. The group $A_{\vr} = A_M$ has order $2^m$, where $m$ is the number of distinct irreducible symplectic summands $M_i$ in $M$. The full Langlands conjecture is known when $\dim(V)$ = $3$ [Ku] or $5$ [GT1].

\bigskip

\item Since $G$ is an adjoint group, there is a unique $T$-orbit on the set of generic characters, and hence a single natural bijection $J : \Pi_{\vr} \to 
\Hom(A_{\vr}, \pm 1)$.

\bigskip

\item The adjoint action on the $L$-packet is trivial.

\bigskip

\item The pure inner forms of $G$ are the groups $G' = \SO(V')$, where $V'$ is an orthogonal space over $k$ with $\dim(V') = \dim(V)$ and 
$\disc(V') = \disc(V)$. 

If $k$ is non-archimedean and $n \geq 1$, there is a unique non-split pure inner form $G'$, which has $k$-rank $(n-1)$. The representation $\pi(\vr,\chi)$ is a representation of $G$ if $\chi(-1) = +1$ and a representation of $G'$ if $\chi(-1) = -1$. If $k = \RR$ and $G = \SO(p,q)$, then the pure inner forms are the groups $G' = \SO(p',q')$ with $ q' \equiv q \mod 2$, and $\pi(\vr,\chi)$ is a representation of one of the groups $G'$ with $(-1)^{(q-q')/2} = \chi(-1)$.

\bigskip

\item The center $Z$ of $G$ is trivial.

\end{enumerate}

\bigskip

\noindent {\bf \underline{The Even Special Orthogonal Group $G = \SO(V)$, $\dim(V) = 2n$, $\disc(V) = d$}}

\bigskip

\begin{enumerate}

\item A Langlands parameter determines an orthogonal representation $M$ of $WD$, with
\[  \text{$\dim(M) = \dim(V)$ and $\det(M) = \CC(d)$.} \]
 The group $A_{\vr} = A_M^+$ has order $2^m$, where $m$ is either the number of distinct irreducible orthogonal summands $M_i$ in $M$, or the number of distinct orthogonal summands minus 1.
\vskip 5pt

\noindent  The latter case occurs if some irreducible orthogonal summand $M_i$ has odd dimension (in which case the orthogonal representation $M$ determines the parameter $\vr$).  
If every irreducible orthogonal summand $M_i$ of $M$ has even dimension,
then $A_M^+ = A_M$ and there are two parameters $\{\phi, \phi^*\}$ which
determine the same orthogonal representation $M$. The representations
$\pi(\phi,\chi)$ and $\pi(\phi^*,\chi)$ are conjugate under the outer action
of $\O(V)$ on $\SO(V)$.
\vskip 5pt

\noindent The full Langlands conjecture is known when $\dim(V)$ = $2$ or $4$ or $6$.
\bigskip

\item The set $D$ of $T$-orbits on generic characters is a principal homogeneous space for the group $E = \NN K^{\times}/k^{\times 2}$, where $K$ is the splitting field of $G$. We will see in \S \ref{S:nu} that the choice of the orthogonal space $V$ identifies the set $D$ with the set of $G$-orbits on the set of non-isotropic lines $L \subset V$, such that the space $L^{\perp}$ is split.

\bigskip

\item The adjoint action is via elements $c$ in the group 
$E = \NN K^{\times} /k^{\times 2}$. This acts on the irreducible representations of $A_{\vr}$ via tensor product with the character $\eta_c(a) = \det (M^a)(c)$, and on the set $D$ of orbits of generic characters by mapping a line $L = kv$ with $\langle v,v \rangle =\alpha$ in $k^{\times}$ to a line $L' =  kv'$ with $\langle v',v' \rangle 
= c\cdot \alpha$. 

\bigskip

\item The pure inner forms of $G$ are the groups $G' = \SO(V')$, where $V'$ is an orthogonal space over $k$ with $\dim(V') = \dim(V)$ and $\disc(V') = \disc(V)$. 
\vskip 5pt

\noindent If $k$ is non-archimedean and $V$ is not the split orthogonal space of dimension 2, there is a unique pure inner form $G'$, such that the Hasse-Witt invariant of $V'$ is distinct from the Hasse-Witt invariant of $V$. The representation $\pi(\vr,\chi)$ is a representation of $G$ if $\chi(-1) = +1$ and a representation of $G'$ if $\chi(-1) = -1$. If $k = \RR$ and $G = \SO(p,q)$, then the pure inner forms are the groups $G' = \SO(p',q')$ with $ q' \equiv q \mod 2$, and $\pi(\vr,\chi)$ is a representation of one of the groups $G'$ with $(-1)^{(q-q')/2} = \chi(-1)$.

\bigskip

\item If $\dim(V) = 2$, then $Z = G$. If $\dim(V) \geq 4$, then $Z(k) = \langle \pm 1\rangle$ and the central character of $\pi(\vr)$ maps the element $-1$ in $Z(k)$ to $\ep(M,\psi)/ \ep(\det M,\psi)$. 

\end{enumerate}

\bigskip

\noindent{\bf \underline{The Odd Unitary Group $G = \U(V)$, $\dim V = 2n + 1$}}

\bigskip

\begin{enumerate}

\item A Langlands parameter is a conjugate-orthogonal representation $M$ of $WD$, with $\dim(M) = \dim(V)$. The group $A_{\vr} = A_M$ has order $2^m$, where $m$ is the number of distinct irreducible conjugate-orthogonal summands $M_i$ in $M$. The full Langlands conjecture is known when $\dim(V)$ = $1$ or $3$ [Ro].

\bigskip

\item There is a unique $T$-orbit on the set of generic characters, and hence a single natural isomorphism $J : \Pi_{\vr} \to \rm Hom(A_{\vr}, \pm 1)$.

\bigskip

\item The adjoint action on the $L$-packet is trivial.

\bigskip

\item The pure inner forms of $G$ are the groups $G' = \U(V')$, where $V'$ is a hermitian (or skew-hermitian) space over $k$ with $\dim(V') = \dim(V)$. 
\vskip 5pt

\noindent  If $k_0$ is non-archimedean, 
there is a unique pure inner form $G'$ such that the discriminant of $V'$ is distinct from the  discriminant of $V$. The representation $\pi(\vr,\chi)$ is a representation of $G$ if $\chi(-1) = +1$ and a representation of $G'$ if $\chi(-1) = -1$. If $k _0= \RR$ and $G = \U(p,q)$, then the pure inner forms are the groups $G' = \U(p',q')$, and $\pi(\vr,\chi)$ is a representation of one of the groups $G'$ with $(-1)^{q-q'} = \chi(-1)$.

\bigskip

\item The center $Z(k_0) = k^{\times}/k_0^{\times} = \U(1)$, and the central character of $\pi(\vr,\chi)$ has parameter $\det(M)$.

\end{enumerate}

\bigskip

\noindent{\bf \underline{The Even Unitary Group $G = \U(V)$, $\dim V = 2n$}}

\bigskip

\begin{enumerate}

\item A Langlands parameter is a conjugate-symplectic representation $M$ of $WD$, with $\dim(M) = \dim(V)$. The group $A_{\vr} = A_M$ has order $2^m$, where $m$ is the number of distinct irreducible conjugate-symplectic summands $M_i$ in $M$. The full Langlands conjecture is known when $\dim(V)$ = $2$  [Ro].

\bigskip

\item The set $D$ of $T$-orbits on generic characters is a principal homogeneous space for the group $E = H^1(k,Z) = k_0^{\times}/\NN k^{\times }$ of order $2$. We will see in \S \ref{S:nu} that the choice of a hermitian space $V$ identifies the set $D$ with the set of $\NN k^{\times}$-orbits on the nontrivial additive characters $\psi_0$ of $k/k_0$. Similarly, the choice of a skew-hermitian space $V$ identifies the set $D$ with the set of $\NN k^{\times}$-orbits on the nontrivial additive characters $\psi$ of $k_0$.

\bigskip

\item The adjoint action is via elements $c$ in the group $E = k_0^{\times}/\NN k ^{\times}$. The nontrivial class $c$ acts on the irreducible representations of $A_{\vr}$ via tensor product with the character $\eta(a) = (-1)^{\dim(M^a)}$, and on the set $D$ of orbits of generic characters by mapping $\psi_0(x)$ to $\psi_0(cx)$, or $\psi(x)$ to $\psi(cx)$.
\bigskip

\item The pure inner forms of $G$ are the groups $G' = \U(V')$, where $V'$ is a hermitian (or skew-hermitian) space over $k$ with $\dim(V') = \dim(V)$. 
\vskip 5pt

\noindent  If $k_0$ is non-archimedean, there is a unique pure inner form $G'$ such that the discriminant of $V'$ is distinct from the  discriminant of $V$. The representation $\pi(\vr,\chi)$ is a representation of $G$ if $\chi(-1) = +1$ and a representation of $G'$ if $\chi(-1) = -1$. If $k _0= \RR$ and $G = \U(p,q)$, then the pure inner forms are the groups $G' = \U(p',q')$, and $\pi(\vr,\chi)$ is a representation of one of the groups $G'$ with $(-1)^{q-q'} = \chi(-1)$.

\bigskip

\item The center $Z(k_0) = k^{\times}/k_0^{\times} = \U(1)$, and the central character of $\pi(\vr,\chi)$ has parameter $\det(M)$.

\end{enumerate}
\vskip 10pt

The forthcoming book of Arthur [A3] and the papers [Mo1, Mo2] of Moeglin should
establish most of the above expectations. 
\vskip 15pt

\section{Vogan $L$-packets for the metaplectic group}  \label{S:metaplectic}

Let $(W, \langle - ,- \rangle_W)$ be a symplectic space of dimension $2 n \geq 0$ over the 
local field $k$.  We assume, as usual, that char$(k) \not= 2$.  In this 
section, we also assume that $k \not= \CC$.
\vskip 10pt

Let $\widetilde{\Sp}(W)$ denote the nontrivial double cover of the 
symplectic group $\Sp(W)(k)$.  We will use the Howe duality correspondence (also known as the theta correspondence) to describe the (genuine) representation theory of $\widetilde{\Sp}(W)$ in terms of the  representation theory  of  the groups $\SO(V)$ over $k$, with $\dim V 
= 2 n + 1$.  Assuming the Langlands-Vogan parameterization of 
irreducible representations of $\SO(V)$ over $k$, with $\dim V 
= 2 n +1$, we then obtain a notion of Vogan $L$-packets for the  
genuine irreducible representations $\widetilde{\pi}$ of 
$\widetilde{\Sp}(W)$.  More precisely, 
the Langlands parameter of a genuine representation of $\widetilde{\Sp}(W)$ will be a symplectic 
representation
$$\vr : WD \to \Sp(M) \quad \text{with} \quad    \dim M = 2n,$$
and the individual representations $\widetilde{\pi}(\vr, 
\chi)$ in the Vogan packet $\Pi_{\vr}$ will be indexed by quadratic characters 
\[  \chi :  A_{\vr} = A_M \to \langle \pm 1\rangle. \] 
\vskip 10pt

This parameterization of the irreducible genuine
representations of $\ws(W)$ will depend on 
the choice of a nontrivial additive character $\psi$ of $k$, up to 
multiplication by $k^{\times 2}$.  Such an orbit of additive characters 
$\psi$ determines an orbit of generic characters $\theta : N \to \CC^\times$ 
for $\Sp(W)$.  The character $\theta$ also determines a character 
$\widetilde{\theta}$ of the unipotent radical $\widetilde{N}  \simeq N$ of 
$\widetilde{\Sp}(W)$.   Our parameterization is normalized so that for 
generic parameters $\vr$, the unique representation $\widetilde{\pi} \in \Pi_{\vr}$ which is $\widetilde{\theta}$-generic corresponds to the 
trivial character $\chi = 1$ of $A_{\vr}$. Such a dependence of the Langlands parameterization 
on the choice of an additive character $\psi$ is already present in the case of the linear classical groups discussed in the previous section. For the metaplectic groups, the dependence is more serious: even the Langlands parameter $\vr$ associated to $\widetilde{\pi}$ depends on the choice of $\psi$. 

\vskip 10pt

To define the parameters $(\vr, \chi)$ of $\widetilde{\pi}$, we 
let $(V,q)$ be a quadratic space over $k$ with 
\[  \dim V = 2n +1 \quad \text{and} \quad \disc  (V) = (-1)^n \det (V) \equiv 1 \in k^\times/k^{\times 2}. \]
Note that the discriminant above refers to the discriminant of the {\em quadratic} space $(V,q)$. 
The quadratic form $q$ on $V$ gives rise to a symmetric bilinear form
\[  \langle x,y \rangle_V = q(x+y) - q(x) -q(y) \]
so that
\[  \langle x,x \rangle_V = 2 \cdot q(x), \]
and
\[  \disc(V ,\langle-,-\rangle_V) = 2 \cdot \disc(V,q) = 2 \in k^{\times}/k^{\times 2}. \]
\vskip 5pt

The space $W \otimes V$  is a symplectic space over $k$ with the skew-symmetric form 
\[  \langle- , -\rangle_W \otimes \langle-, -\rangle_V, \]
and one  has the  associated Heisenberg group 
\[  H(W \otimes V) = k \oplus (W \otimes V),\]
which has a one dimensional center $k$.
Associated to $\psi$, $H(W \otimes V)$ has a unique irreducible representation $\omega_{\psi}$ with central character $\psi$ (by the Stone-von-Neumann theorem). Now $\Sp(W \otimes V)$ acts as automorphisms of $H(W \otimes V)$ via its natural action on $W \otimes V$ and the trivial action on $k$. Thus $\omega_{\psi}$ gives rise to a projective representation of $\Sp(W \otimes V)$ and it was shown by Weil that this projective representation is a linear representation of $\widetilde{\Sp}(W \otimes V)$. 
We thus have a representation $\omega_{\psi}$ of the semi-direct product 
\[  \widetilde{\Sp}(W \otimes V) \ltimes  H(W \otimes V). \] 
This is the so-called Weil representation (associated to $\psi$). As a representation of $\widetilde{\Sp}(W \otimes V)$, it is the direct sum of two irreducible representations, and its isomorphism class depends only on the $k^{\times 2}$-orbit of $\psi$. 
 \vskip 5pt
  
Via a natural homomorphism 
\[  \widetilde{\Sp}(W) \times \O(V) \longrightarrow \widetilde{\Sp}(W \otimes V), \]  
we regard the Weil representation $\omega_{\psi}$ as a representation $\omega_{W,V, \psi}$ of  
$\widetilde{\Sp}(W) \times \O(V)$. The theory of  Howe duality gives a 
correspondence between irreducible genuine representations $\wpi$ 
of $\ws(W)$ and certain irreducible representations $\sigma$ of $\O(V)$. 
\vskip 10pt

More precisely, given an irreducible representation $\sigma$ of $\O(V)$, the maximal $\sigma$-isotypic quotient of $\omega_{W,V, \psi}$ has the form
\[  \sigma \boxtimes \Theta_{W,V, \psi}(\sigma) \]
for some smooth representation $\Theta_{W,V,\psi}(\sigma)$ (the big theta lift of $\sigma$) of $\ws(W)$. It is known ([K] and [MVW]) that $\Theta_{W,V,\psi}(\sigma)$ is either zero or has finite length. 
Let $\theta_{W,V,\psi}(\sigma)$ (the small theta lift of $\sigma$) denote the maximal semisimple quotient of $\Theta_{W,V,\psi}(\sigma)$. It is known by Howe [Ho] and Waldspurger [Wa3] that when the residue characteristic of $k$ is different from $2$, then $\theta_{W,V,\psi}(\sigma)$ is irreducible or zero; this is the so-called Howe's conjecture. In the following, we will assume that the same holds when the residue characteristic of $k$ is $2$. 
\vskip 10pt

Analogously, if $\wpi$ is an irreducible representation of $\ws(W)$, we have the representations $\Theta_{W,V, \psi}(\wpi)$ and $\theta_{W,V,\psi}(\wpi)$ of $\O(V)$. 
\vskip 10pt

Now we have the following theorem, which is due to Adams-Barbasch [AB] when $k  = \R$ and essentially due to Kudla-Rallis [KR] when $k$ is non-archimedean.
\vskip 5pt

\begin{thm} \label{T:KR}  
Assume that the local field $k$ is either real or non-archimedean with odd residual characteristic.
Then corresponding to the choice of an additive character $\psi$ of $k$, there is a natural bijection given by the theta correspondence between
\[\xymatrix{ \left\{ \text{irreducible genuine  representations  $\wpi$  of  $\ws(W)$} \right \}  \ar@{<->}[d] \\
\coprod \left\{\text{irreducible  representations  $ \sigma' $ of  $\SO(V')$} \right\}
}\]
where the union is disjoint, and taken over all the isomorphism  classes of 
orthogonal spaces $V'$ over $k$ with $\dim V' = 2n+1$ and $\disc(V') =1$.  
\vskip 5pt 

More precisely,  given an irreducible  representation $\wpi$ of $\ws(W)$, 
there is a unique $V'$ as above such that 
$\theta_{W,V', \psi}(\wpi)$ is nonzero, in which case the image of 
$\wpi$ under the above bijection is the restriction of 
$\theta_{W,V', \psi}(\wpi)$ to $\SO(V')$. 
\end{thm}
 
 \vskip 5pt

\vskip 10pt
 
\begin{proof}
We give a sketch of the proof of Theorem \ref{T:KR} when $k$ is non-archimedean.
The reader will not find Theorem \ref{T:KR} in the reference [KR], so let us explain how it follows from the results there (plus a little extra work).  There are 2 isomorphism classes of quadratic space of dimension $2n+1$ and trivial discriminant; we denote these by $V$ and $V'$, and assume that $V$ is split. 
To simplify 
notation, we shall write $\Theta$ in place of $\Theta_{W,V, \psi}$ and $\Theta'$ in place of $\Theta_{W,V',\psi}$. 
\vskip 5pt

We now divide the proof into two steps:
\vskip 10pt

\noindent (i) Given an irreducible representation $\wpi$ of $\widetilde{\Sp}(W)$, exactly one of 
$\Theta(\wpi)$ or $\Theta'(\wpi)$ is nonzero.
\vskip 5pt

This assertion was also shown in a recent preprint of C. Zorn [Z]. 
In any case, [KR, Thm. 3.8] shows that any irreducible representation $\wpi$ of $\widetilde{\Sp}(W)$ participates in theta correspondence with at most one of $\O(V)$ or $\O(V')$.
We claim however that $\wpi$ does have nonzero theta lift to $\O(V)$ or 
$\O(V')$. 
To see this, note that [KR, Prop. 4.1] shows that $\wpi$ has nonzero theta lift to $\O(V)$  
if and only if
\[  \Hom_{\widetilde{\Sp}(W) \times \widetilde{\Sp}(W)} (R(V), \wpi \boxtimes \wpi^{\vee}) \ne 0, \]
where $R(V)$ is the big theta lift of the trivial representation of $\O(V)$ to 
$\widetilde{\Sp}(W+ W^-)$ (where $W^-$ is the symplectic space obtained from $W$ by scaling its form by $-1$). Similarly, one has the analogous statement for $V'$. On the other hand,
if $I_P(s)$ denotes the degenerate principal series representation of $\widetilde{\Sp}(W+ W^-)$ unitarily induced from the character $\chi_{\psi} \cdot |\det|^s$ of a Siegel parabolic subgroup, then one can show that
\[  I_P(0)  = R(V) \oplus R(V'), \]
so that $R(V)$ and $R(V')$ are unitarizable and thus irreducible (since they have a unique irreducible quotient). In particular, we conclude that $\wpi$ has nonzero theta lift to one of $\O(V)$ or $\O(V')$ 
if and only if
\[  \Hom_{\widetilde{\Sp}(W) \times \widetilde{\Sp}(W)} (I_P(0), \wpi \boxtimes \wpi^{\vee}) \ne 0. \] 
We thus need to show that this Hom space is nonzero. This can be achieved by  the doubling method of Piatetski-Shapiro and Rallis [GPSR], which  provides a zeta integral 
\[    Z(s):  I_P(s) \otimes \wpi \otimes \wpi^{\vee} \longrightarrow \CC. \]
The precise definition of $Z(s)$ need not concern us here; it 
suffices to note that for 
a flat section $\Phi(s) \in I_P(s)$ and $f \otimes f^{\vee} \in \wpi \otimes \wpi^{\vee}$, 
$Z(s, \Phi(s), f \otimes f^{\vee})$ is a meromorphic function in $s$. Moreover, at any $s = s_0$, the leading term of the Laurent expansion of $Z(s)$ gives a nonzero element 
\[  Z^*(s_0) \in \Hom_{\widetilde{\Sp}(W) \times \widetilde{\Sp}(W)} 
(I_P(s_0), \wpi \boxtimes \wpi^{\vee}). \]
This proves our contention that $\wpi$ participates in the theta correspondence with exactly one of 
$\O(V)$ or $\O(V')$.  

\vskip 10pt
 By (i), one obtains a map
\[\xymatrix{
\left\{ \text{irreducible genuine  representations  $\wpi$  of  $\ws(W)$} \right\} \ar[d]\\
\coprod \left \{\text{irreducible  representations  $ \sigma' $ of  $\O(V')$} \right \}
}\]

Moreover, this map is injective by the theorem of Waldspurger [Wa3] proving Howe's conjecture.
 \vskip 10pt

\noindent (ii) An irreducible representation $\pi_0$ of $\SO(V)$ has two extensions to 
$\O(V) = \SO(V) \times \langle \pm 1\rangle$, and 
exactly one of these extensions participates in the theta correspondence with $\widetilde{\Sp}(W)$. 
The same assertion holds for representations of $\SO(V')$.
\vskip 5pt

Suppose on the contrary that $\pi$ is an irreducible representation of $\O(V)$ such that both $\pi$ and $\pi \otimes \det$ participate in theta correspondence with $\widetilde{\Sp}(W)$, say
\[  \wpi = \theta(\pi) \quad\text{and} \quad \wpi'  = \theta(\pi \otimes \det).  \]
Now consider the seesaw diagram:

\[\xymatrix{ \widetilde{\Sp}(W+W^-)   \ar@{-}[dd]  & \O(V)\times \O(V) \ar@{-}[dd]\\
 & \\
 \widetilde{\Sp}(W) \times \widetilde{\Sp}(W)    & \O(V).
}\]
The seesaw identity implies that
\[   \Hom_{\widetilde{\Sp}(W) \times \widetilde{\Sp}(W)}(\Theta_{W+W^-, V,\psi}(\det), \wpi' \boxtimes \wpi^{\vee})
  \supset 
\Hom_{\O(V)}((\pi\otimes \det)   \otimes \pi ^{\vee}, \det) \ne 0. \]
This implies that 
\[ \Theta_{W+W^-, V, \psi}(\det) \ne 0.\]
However, a classical result of Rallis [R, Appendix] says that the determinant character of $\O(V)$ does not participate in the theta correspondence with $\widetilde{\Sp}(4r)$ for $r \leq n$. This gives the desired contradiction.
\vskip 5pt

We have thus shown that at most one of $\pi$ or $\pi \otimes \det$ could have nonzero theta lift to $\widetilde{\Sp}(W)$.  On the other hand, the analog of the zeta integral argument in (i) shows that
one of $\pi$ or $\pi \otimes \det$ does lift to $\widetilde{\Sp}(2n)$. This proves (ii).

\vskip 10pt

Putting (i) and (ii) together,  we have established the theorem.
\end{proof}
\vskip 10pt

The only reason for the assumption of odd residue characteristic in the theorem is that Howe's conjecture for local theta correspondence is only known under this assumption.   \vskip 10pt

Since $\SO(V)$ is an adjoint 
group, there is a unique orbit of generic characters, and the Vogan 
parameterization of irreducible representations $\sigma'$ of the groups
$\SO(V')$ requires no further choices. 
So we label $\wpi = \wpi (M, \chi)$ using 
the Vogan parameters $(M,\chi)$ of the representation $\sigma' = \Theta_{W,V',\psi}(\wpi)$.   The theorem thus gives the following corollary.
\vskip 5pt

\begin{cor}  \label{C:KR}
Assume that the residue characteristic of $k$ is odd.
Suppose that the local Langlands-Vogan  parameterization holds for $\SO(V')$.
Then one has a parameterization (depending on $\psi$)
of
$$\{ \text{irreducible genuine  representations  $\wpi$  of  $\ws(W)$} \}$$
by  the set of isomorphism classes of pairs $(\varphi,\chi)$ such that 
\[  \varphi: WD \longrightarrow \Sp(M) \]
is a symplectic representation of $WD$ and $\chi$ is an irreducible character of the component group $A_{\varphi}$. 
\end{cor}
\vskip 5pt

It follows that 
the various desiderata for the Vogan packets of $\ws(W)$ can be obtained from those of $\SO(V')$ if one understands the properties of the theta correspondence. For example, 
in the theta correspondence, generic representations of 
the split group $\SO(V)$ lift to $\widetilde{\theta}$-generic representations of $\ws(W)$.  Hence 
the $\widetilde{\theta}$-generic element in the $L$-packet of $M$ corresponds to the trivial
character of the component group $A_M$.  Also, when $k$ is non-archimedean, $\wpi(M,\chi)$ 
is lifted from the split group $\SO(V)$ precisely when $\chi(-1)=1$.

 \bigskip 
One difference between metaplectic  and linear groups is in the description of the action of the adjoint group by outer automorphisms on the set of irreducible representations. 
The adjoint action of the symplectic similitude group $\text{GSp}(W)$ on the set of genuine irreducible  representations of $\ws(W)$ factors through the quotient
$$k^\times/k^{\times 2} = 
{\rm PSp}(W) (k) /{\rm Image~} 
\Sp(W)(k).$$  
In the metaplectic case, this outer action does not permute the representations $\widetilde{\pi}$ in an individual Vogan $L$-packet, and we predict a more complicated recipe, as follows.
\vskip 5pt

\vskip 5pt

\begin{con} \label{conj-meta}
 If $\widetilde{\pi}$ has $\psi$-parameter $(M,\chi)$ and $c$ is a class in $k^{\times}/k^{\times 2}$, the
conjugated representation $\widetilde{\pi^c}$ has $\psi$-parameter $(M(c), \chi \cdot \eta[c])$.
Here $M(c)$  is the twist of 
$M$ by the one-dimensional orthogonal representation $\CC(c)$ so that its component group $A_{M(C)}$ is canonically isomorphic to $A_M$. The character $\eta[c]$ is defined by 
\[  \eta[c]= \chi_N : A_M \to \langle\pm1\rangle, \]
where $N$ is the two dimensional orthogonal representation $N = \CC + \CC(c)$, so that
$$\eta[c] (a)   =  \ep(M^a)\cdot  \ep(M(c)^a)\cdot  (c, -1)^{\frac{1}{2}\dim M^a}.$$ 
\end{con}
\bigskip

This conjecture is known when $\dim W = 2$, where it is a result of Waldspurger ([W1] and [W2]); our recipe above is suggested by his results.
\vskip 10pt

The above conjecture has the  following consequence. If one replaces the character $\psi$ by the character 
\[  \psi_c : x \mapsto \psi(cx) \]
of $k$, then   the new Vogan parameter (relative to $\psi_c$) of $\widetilde{\pi}$ will be $(M(c), \chi \cdot \eta[c])$.
\vskip 10pt

A consequence of this is the following.
Suppose that $\wpi$ is such that 
\[  \theta_{W, V, \psi}(\wpi) \ne 0 \quad \text{and} \quad  \theta_{W,V,\psi_c}(\wpi) \ne 0 \]
as representations of $\SO(V)$.
Then, when $\dim W = 2$, a basic result of Waldspurger says that
\[  \theta_{W,V,\psi}(\wpi) \cong \theta_{W,V,\psi_c}(\wpi) \otimes \chi_c, \]
where $\chi_c$ is the character
\[  
\begin{CD}
\SO(V)(k) @>{\text{spinor norm}}>> k^{\times} /k^{\times 2} @>(c,-)>> \langle \pm 1\rangle. \end{CD} \]
However, according to the conjecture above, if the Vogan parameter of  $\theta_{W,V\psi}(\wpi)$ is 
$(M ,\chi)$, then that of $\theta_{W,V,\psi_c}(\wpi) \otimes \chi_c$ is $(M, \chi \cdot \eta[c])$. So the two representations are equal if and only if the character $\eta[c]$ is trivial. The assumption that $\wpi$ has nonzero theta lift to $\SO(V)$ with respect to both $\psi$ and $\psi_c$ implies that 
\[  \eta[c](-1) = 1. \]
When $\dim W = 2$, this is equivalent to saying that $\eta[c]$ is trivial. But when $\dim W  >2$, this is no longer the case and one can construct such counterexamples already when $\dim W = 4$. 
\vskip 10pt

Another way of saying the above is that
if the Langlands-Vogan parameter of a representation $\sigma$ of $\SO(V)(k)$
is $(M,\chi)$, then $(M(c),\chi)$ is the parameter of $\sigma \otimes \chi_c$ and 
 under the $\psi$-theta correspondence, $\sigma \otimes \chi_c$ should be paired with an element in the Vogan packet of $\widetilde{\pi^c}$, but not necessarily to that specific representation. 
\vskip 10pt

 \bigskip

\section{The representation $\nu$ of $H$ and generic data}  \label{S:nu}

In this section, we shall describe the remaining ingredient in the restriction problem to be studied. Suppose as before that $k$ is a local field with an involution $\sigma$ (possibly trivial) and $k_0$ is the fixed field of $\sigma$. Let $V$ be a $k$-vector space endowed with a non-degenerate sesquilinear form $\langle-, -\rangle$ with sign $\epsilon$. Moreover, suppose that $W \subset V$ is a non-degenerate subspace satisfying:
\vskip 5pt
\begin{enumerate}
\item $\epsilon \cdot (-1)^{\dim W^{\perp}} = -1$ 
\vskip 5pt

\item $W^{\perp}$ is a split space. 
\end{enumerate}
\vskip 5pt

\noindent So we have
\[  \dim W^{\perp} = \begin{cases}
\text{odd, if $\epsilon = 1$, i.e. $V$ is orthogonal or hermitian;} \\
\text{even, if $\epsilon = -1$, i.e. $V$ is symplectic or skew-hermitian.} \end{cases} \]
\vskip 5pt

Let $G(V)$ be the identity component of the automorphism group of $V$ and $G(W) \subset G(V)$ the subgroup which acts as identity on $W^{\perp}$. Set
\[  G = G(V) \times G(W). \]
As explained in Section \ref{S:classical}, $G$ contains a subgroup $H$ defined as follows. 
Since $W^{\perp}$ is split, we may write
\[  W^{\perp} = X + X^{\vee} \quad \text{or} \quad W^{\perp} = X+ X^{\vee} + E \]
depending on whether $\dim W^{\perp}$ is even or odd, where in the latter case, $E$ is a non-isotropic line.  Let $P$ be a parabolic subgroup which stabilizes a complete flag of (isotropic) subspaces in $X$. Then $G(W)$ is a subgroup of a Levi subgroup of $P$ and thus acts by conjugation on the unipotent radical $N$ of $P$.  We set
\[  H = N \rtimes G(W). \]
Note that there is a natural embedding $H \hookrightarrow G$  which is the natural inclusion 
\[  H \subset P \subset G(V) \]
in the first factor and is given by the projection
\[  H \longrightarrow H/N = G(W) \]
in the second factor.  When $G' = G(V') \times G(W')$ is a relevant pure inner form of $G$, a similar construction gives a distinguished subgroup $H'$. 
\vskip 10pt

The goal of this section is to describe a distinguished representation $\nu$ of $H$ (and similarly $H'$).
 It will turn out that 
$\dim \nu = 1$ if $\dim W^{\perp}$ is odd (orthogonal and hermitian cases), whereas $\nu$ has Gelfand-Kirillov dimension $1/2 \cdot \dim(W/k_0)$ when $\dim W^{\perp}$ is even (symplectic and skew-hermitian cases). Because of this, we will treat the cases when $\dim W^{\perp}$ is even or odd separately.
\vskip 10pt

\noindent{\bf \underline{Orthogonal and Hermitian Cases (Bessel Models)}}
\vskip 10pt

Assume that $\dim W^{\perp} = 2n+1$ and write 
\[ W^{\perp} = X + X^{\vee}+E \quad \text{with} \quad E =  \langle e \rangle,  \]
where $X$ and $X^{\vee}$ are maximal isotropic subspaces which are in duality using the form $\langle-,-\rangle$ of $V$, and $e$ is a non-isotropic vector.  Let $P(X)$ be the parabolic subgroup in $G(V)$ stabilizing the subspace $X$, and let $M(X)$ be the Levi subgroup of $P(X)$
which stabilizes both $X$ and $X^\vee$, so that  
\[  M(X) \cong \GL(X) \times G(W \oplus E). \]
We have 
\[  P(X)= M(X) \ltimes N(X) \]
where $N(X)$ is the unipotent radical of $P(X)$. The group $N(X)$
sits in an exact sequence of $M(X)$-modules,
\[  \begin{CD}
0 @>>> Z(X) @>>> N(X) @>>> N(X)/Z(X)  @>>> 0, 
\end{CD} \]
and using the form on $V$, one has natural isomorphisms
\[  Z(X) \cong \{ \text{skew-hermitian forms on $X^{\vee}$}\}, \]
  and
\[  N(X)/Z(X) \cong \Hom(W+E, X) \cong (W +E) \otimes X.  \]
Here, when $k = k_0$, skew-hermitian forms on $X^{\vee}$  simply mean symplectic forms.
In particular, $Z(X)$ is the center of $N(X)$ unless $k = k_0$ and $\dim X =1$, in which case $Z(X)$ is trivial and $N(X)$ is abelian.
  \vskip 5pt

Now let 
\[  \ell_X: X \rightarrow k \] 
be a nonzero $k$-linear homomorphism, and let 
\[  \ell_W: W \oplus E  \longrightarrow k \]
be a nonzero $k$-linear  homomorphism which is zero on the hyperplane $W$. Together, these give a map 
\[  \ell_X \otimes \ell_W: X \otimes (W + E ) \longrightarrow k, \]
and one can  consider the composite map
\[  \begin{CD}
 \ell_{N(X)} : N(X) @>>> N(X)/Z(X) \cong X \otimes (W + E)
@>\ell_X \otimes \ell_W>> k.
\end{CD} \]
Let $U_X$ be any maximal unipotent subgroup  of $\GL(X)$ which stabilizes $\ell_X$. Then
the subgroup
\[  U_X \times G(W) \subset M(X) \]
fixes the homomorphism $\ell_{N(X)}$.
 
 \vskip 5pt
 
 Now the subgroup 
 \[  H \subset G = G(V) \times G(W) \] 
 is given by 
 \[  H = (N(X) \rtimes (U_X \cdot G(W)) = N \rtimes G(W). \]
We may extend the map $\ell_{N(X)}$ of $N(X)$ to $H$, by making it trivial on $U_X \times G(W)$.  If $\psi$ is a non-trivial additive character of $k$, and
\[  \lambda_X: U_X \longrightarrow \SS^1 \]
is a generic character of $U_X$, then
 the representation $\nu$ of $H$ is defined by
 \[   \nu =  (\psi \circ \ell_{N(X)} )\boxtimes \lambda_X. \]
 The pair $(H,\nu)$ is uniquely determined up to conjugacy
in the group $G = G(V) \times G(W)$ by the pair $W \subset V$. 

\vskip 10pt

One can give a more explicit description of $(H ,\nu)$, by explicating the choices of $\ell_X$, $U_X$ and $\lambda_X$ above. To do this, 
choose a  basis $\{ v_1, \cdots, v_n\} $ of $X$, 
with dual basis $\{ v_i'\} $ of $X^{\vee}$.  
Let $P \subset G(V)$ be the parabolic subgroup which stabilizes the flag
$$0 \subset 
  \langle v_1\rangle  \subset  \langle v_1, v_2\rangle  \subset 
\cdots  \subset \langle v_1, \cdots,   v_n\rangle  = X,$$ 
and let 
\[  L = (k^\times)^n \times  G(W + E) \]
be the Levi subgroup of $P$ which stabilizes
the lines $\langle v_i\rangle $ as well as the subspace $W +E$.  
The torus $T = (k^\times)^n$ scales these lines: 
\[  t(v_i) =  t_i v_i, \]
and $G(W + E)$ acts trivially. 
\vskip 10pt

Let $N$ be the unipotent radical of $P$, so that 
\[  N = U_X \ltimes N(X) \]
where $U_X$ is the unipotent radical of the Borel subgroup in $\GL(X)$ stabilizing the chosen flag above.  Now define a homomorphism $f : 
N \to k^n$ by 
$$
\begin{array}{rcl}
f(u) & = & (x_1, \cdots, x_{n-1}, z),  \\
x_i & = & \langle u v_{i+1}, v'_i \rangle , \ \ \ \  i = 1,2,\cdots, n-1 \\
z & = & \langle ue, v'_n \rangle . 
\end{array}
$$
\noindent 
The subgroup of $L$ which fixes $f$ is $G(W)$, the subgroup of $G(W + 
E)$ fixing the vector $e$.  The torus acts on $f$ by 
$$f(tut^{-1}) = (t_1^{\sigma} / t_2 x_1, t^{\sigma}_2/t_3 x_2, \cdots, 
t^{\sigma}_n z).$$  
\vskip 5pt

Consider the subgroup  $H = N.G(W)$ of  $G = G(V) \times G(W)$. 
Then, for a non-trivial additive character $\psi$ of $k$,  the representation $\nu$ is given by:
$$\begin{array}{lcl} 
\nu & : & H \to \CC^\times \\
&& (u,g) \mapsto \psi (\sum x_i + z).
\end{array}
$$
It is as regular as possible on $N$, among the characters fixed by $G(W)$.  
As noted above, up to $G$-conjugacy, the pair $(H, \nu)$ depends only on the initial data 
$W \subset V$, and not on the choices of $\psi$, $\{v_i\}$, or $e$ used 
to define it.  
\vskip 5pt

 \vskip 10pt

A special case of $(H , \nu)$ is worth noting.
If $V$ is orthogonal of even dimension and $W$ 
has dimension 1, then $\SO(W) = 1$ and $H = N$ is the unipotent radical 
of a Borel subgroup $P \subset G = \SO(V)$.  In this case, $\nu$ is simply a 
generic character $\theta_W$ of $N$. By choosing different non-isotropic lines $L$ in the 2-dimensional orthogonal space $W + E$, so that $L^{\perp} = X + X^{\vee} + L'$, and using $L$ in place of $W$ in the above construction, the map $L \mapsto \theta_L$ gives a bijection
\[ \xymatrix { \left \{ \text{$T$-orbits of generic characters on $N$} 
\right \} \ar@{<->}[d] \\ 
  \{ \text{$\SO(V)$-orbits of non-isotropic lines 
$L$ with $L^{\perp}$ split} \} }\]
in the even orthogonal case.
 
\vskip 15pt

\noindent{\bf \underline{Symplectic and Skew-Hermitian Cases (Fourier-Jacobi Models)} }
\vskip 10pt

We now treat the symplectic and skew-hermitian cases, so that $W^{\perp}$ is split  of 
even dimension $2n$ and we may write
\[  W^{\perp} = X + X^{\vee}, \]
where $X$ and $X^{\vee}$ are maximal isotropic subspaces which are in duality using the form on $V$. In this case, $G(W)$ is a subgroup of $\Sp(W/k_0)$, preserving the 
form ${\rm Tr}_{k/k_0} \circ \langle -, - \rangle$.  It will turn out that the representation $\nu$ of $H$ depends on some other auxiliary data besides the spaces $W \subset V$. 
As in the case of Bessel models, we include the case $k = k_0 \times k_0$  in our discussion below.  
\vskip 10pt

We assume first that $\dim X > 0$.
 Let $P(X) = M(X) \cdot N(X)$ be the parabolic subgroup
in $G(W)$ stabilizing the subspace $X$,  with Levi subgroup
\[  M(X) \cong \GL(X) \times G(W) \]
stabilizing both $X$ and $X^\vee$.
Let $Z(X)$ be the center of the unipotent radical $N(X)$ of $P(X)$, so that one has the exact sequence of $M(X)$-modules: 
\[  \begin{CD}
0 @>>> Z(X) @>>> N(X) @>>> N(X)/Z(X) @>>> 0,
\end{CD} \]
Using the form on $V$, one has natural isomorphisms
\[  Z(X) \cong  \{\text{hermitian forms on $X^{\vee}$}\} \]
 and
\[  N(X)/Z(X) \cong \Hom(W, X) \cong W \otimes X.\] 
 Here, if $k = k_0$, then hermitian forms on $X^{\vee}$ simply mean symmetric bilinear forms.
 
 \vskip 5pt
\noindent The commutator map 
\[  [ -, - ]: N(X) \times N(X) \rightarrow N(X) \]
factors through $N(X)/Z(X)$ and takes value in $Z(X)$. It thus
gives rise to a skew-symmetric $k_0$-bilinear map 
\[  \Lambda_{k_0}^2(X \otimes W) \longrightarrow Z(X) =\{ \text{hermitian forms on $X^{\vee}$}\}, \] 
or  equivalently by duality, a map
\[ \{ \text{hermitian forms on $X$} \} \longrightarrow \Lambda_{k_0}^2(X^\vee \otimes W) = \{ \text{symplectic forms on $\text{Res}_{k/k_0}(X \otimes W)$}\}. \] 
Indeed, this last map is a reflection of the fact that, using the skew-hermitian  structure on $W$, the space of hermitian  forms on $X$ can be naturally embedded in the space of 
skew-hermitian forms on $X \otimes W$, and then by composition with the trace map if necessary, 
in the space of symplectic forms on $\text{Res}_{k/k_0}(X \otimes W)$.
\vskip 10pt

Now let 
\[  \ell_X: X \rightarrow k \]
be a nonzero homomorphism, and let $U_X \subset \GL(X)$ be a maximal unipotent subgroup which fixes $\ell_X$.  Then the group $H$ is defined by 
\[  H = N(X) \rtimes (U_X \times G(W)) =  N \rtimes G(W) \]
with $N = N(X) \rtimes U_X$. If   
\[  \lambda_X: U_X \longrightarrow \SS^1  \]
is a generic character of $U_X$, then by composing with  the projection from $H$ to $U_X$, we may regard $\lambda_X$ as a character of $H$.
 \vskip 10pt

On the other hand, by pullinag back, the homomorphism $\ell_X$ gives rise to a linear map 
\[   k_0 = \{ \text{hermitian forms on $k$} \} \longrightarrow \{ \text{hermitian forms on $X$} \}, \]
and hence by duality 
\[  \ell_{Z(X)}:  Z(X) = \{ \text{hermitian forms on $X^{\vee}$} \} \longrightarrow k_0. \] 
Moreover, $\ell_X$ gives a $k$-linear map 
\[  \ell_W: X \otimes W \rightarrow W \]
 making the following 
diagram commute:

\vskip 5pt
 $$
\begin{CD}
  \Lambda_{k_0}^2(X \otimes W) @>[-,-]>> Z(X) \\ 
 @V\ell_WVV   @VV\ell_{Z(X)}V    \\
\Lambda_{k_0}^2(W) @>>\frac{2}{[k:k_0]} \cdot Tr_{k/k_0}(\langle -, - \rangle)> k_0.
\end{CD}
$$
\vspace{4mm}

 For example, when $k = k_0$ and $\dim X = 1$, then the commutator map $[-,-]$ is given by the skew-symmetric form $2 \cdot \langle - , -\rangle_W$ on $N(X)/Z(X) = W$. On the other hand, when $k \ne k_0$ and $\dim X = 1$, it is given by the skew-symmetric form $Tr_{k/k_0}(\langle - , -\rangle_W)$ on $W/k_0$. 
In any case, let us  set 
 \[  V_1 =   \begin{cases}
 \text{the quadratic space with discriminant $1$,   if $k = k_0$;} \\
 \text{the hermitian space of discriminant $1$ if $k \ne k_0$,}
 \end{cases} \]
 and let $H(V_1 \otimes W)$ be the Heisenberg group associated to the symplectic vector space $V_1 \otimes W$ over $k_0$ with form 
 \[  Tr_{k/k_0}( \langle - , -\rangle_{V_1} \otimes \langle - , - \rangle _W). \]
Here, given a quadratic space $(V,q)$ over $k_0$, the associated symmetric bilinear form is
 \[  \langle v_1, v_2 \rangle= q(v_1+v_2) - q(v_1) - q(v_2). \]
 Thus, when $k = k_0$, the form on $V_1$ is such that $\langle v,  v\rangle_V =2$, so that 
 $H(V_1 \otimes W)$ is the Heisenberg group associated to the symplectic vector space $(W, 2 \cdot \langle - , - \rangle_W)$.
   \vskip 5pt

  Now one has the following  commutative diagram of algebraic groups over $k_0$:
\vskip 5pt
 $$
\begin{CD}
0 @>>>  Z(X) @>>> N(X) @>>> X \otimes W @>>> 0\\ 
& & @V\ell_{0,X}VV @VVV @V\ell_W VV \\
0 @>>> k_0 @>>> H(V_1 \otimes W) @>>> W @>>> 0.
\end{CD}
$$
\vspace{4mm}

Given a nontrivial character $\psi:k_0 \rightarrow \SS^1$, one may consider the
unique irreducible  unitarizable representation $\omega_{W, \psi}$ of $H(V_1 \otimes W)$ of GK-dimension $\frac{1}{2} \cdot \dim_{k_0} W$, on which the center of $H(V_1 \otimes W)$ acts by $\psi$. 
Pulling back by 
the above diagram, one obtains an irreducible representation 
$\omega_{\psi}$ of $N(X)$ with central character $\psi \circ \ell_{Z(X)}$.
Up to conjugation by $M(X)$, the representation $\omega_{\psi}$ depends only on $\psi$ up to multiplication by $(k^{\times})^{1+\sigma}$.
The representation $\omega_{\psi}$ can be extended trivially to $U_X$.
Moreover, the group $G(W)$ acts as outer automorphisms of $H(V_1 \otimes W)$, so the theory of Weil representations furnishes us with a projective representation of $G(W)$ on $\omega_{\psi}$.
Thus, one has a projective representation $\omega_{\psi}$ of $H$.

\vskip 10pt

 As in the orthogonal and hermitian cases, we can make the above discussion completely explicit by making specific choices of $\ell_X$, $U_X$ and $\lambda_X$. 
 Assuming that $\dim X = n > 0$, 
choose a basis $\{ v_i\} $ for $X$ and let $\{ v'_i\} $ be the dual basis of $X^{\vee}$.  Let $P \subset G(V)$ be the 
subgroup stabilizing the flag 
$$0 \subset  \langle v_1\rangle  \subset  \cdots  \subset \langle v_1, \cdots, v_n\rangle  
\subset  X,$$ 
and let 
\[  L = G(W) \times (k^\times)^n \]
be the Levi subgroup of $P$  stabilizing the lines $\langle v_i\rangle $ as well as the subspace $W$.  
\vskip 5pt

Let $N$ be the unipotent radical of $P$ and define a homomorphism to 
a vector group
$$f : N \to k^{n-1} \oplus W$$
given by:
$$\begin{array}{rcl}
f(u) & = & (x_1, \cdots, x_{n-1}, y) \\
x_i & = & \langle uv_{i+1}, v'_i\rangle  \\
y & = & \sum_j \langle uw_j, v'_n \rangle . w'_j.
\end{array}$$
\vskip 5pt

\noindent  Here $\{ w_1, \cdots, w_n\} $ is a basis for $W$ over $k$ and $\langle w_i, 
w'_j \rangle  = \delta_{ij}$.  Thus $y$ is the unique vector in $W$ with 
$$\langle w, y \rangle  = \langle uw, v'_n \rangle $$ 
for all $w$ in $W$. 
\vskip 5pt

The torus $T = (k^\times)^n$ acts on $f$ by 
$$f(t u t^{-1}) = (t^{\sigma}_1/t_2 \cdot x_1, \cdots, t^{\sigma}_{n-1}/t_n \cdot  
x_n, t^{\sigma}_n \cdot y)$$ 
and an element $g \in G(W)$ acts by 
$$f(g u g^{-1}) = (x_1, \cdots, x_n, g(y)).$$
\vskip 10pt

Now  the maps $(x_1, \cdots, x_{n-1})$ give a functional 
$$\begin{array}{rcl}
\ell & : & N \to k_0 \\
&& n \mapsto {\rm Tr} (\sum x_i) 
\end{array}
$$
which is fixed by $G(W)$, and is as regular as possible subject to this 
condition. 
Choose a nontrivial additive character $\psi$ of $k_0$.  Then the 
character 
$$\begin{array}{rcl}
\lambda & : & N \to  \SS^1 \\
&& u \mapsto \psi (\ell(u)) = \psi({\rm Tr}(\sum x_i))
\end{array}
$$
is regular, and up to conjugacy by the torus, independent of the choice 
of $\psi$. Since it is fixed by $G(W)$, we may extend it trivially to $G(W)$ and obtain a character $\lambda$ of $H$.

\vskip 10pt

On the other hand, one may define a homomorphism of $N$ to a Heisenberg group as follows.
Let $N_0 \lhd  N$ be the kernel of the map 
$$\begin{array}{cl}
N& \to W \\
u & \mapsto y
\end{array}$$
and define a homomorphism 
$$\begin{array}{rcl}
f_0 & : & N_0 \to k \\
&& u_0 \mapsto z = \langle u_0 v'_n, v'_n\rangle. 
\end{array}
$$
Note that the element $z$ lies in the subfield $k_0$ of $k$, since 
\[   \text{$u_0 v'_n$ is isotropic} \Longrightarrow  z-z^{\sigma} = 0. \]
Hence, we have
\[  f_0 : N_0 \longrightarrow  k_0. \]
The torus act by 
$$f_0 (tut^{-1}) = t_n^{1+\sigma} \cdot z$$ 
and $G(W)$ acts trivially. 
 The above two maps combine to give a 
homomorphism from $N$ to the Heisenberg group $H(W/k_0)$: 
$$
\begin{CD}
0 @>>> N_0 @>>> N @>>> W @>>>0 \\
  @.  @Vf_0VV   @VVV   @VVidV    @. \\
 0 @>>> k_0 @>>>H(V_1 \otimes W) @>>>W @>>> 0 
 \end{CD}
 $$
which is equivariant for the action of $G(W)$ on $N$ and $G(W) \subset 
\Sp(V_1 \otimes W)$ on $H(V_1 \otimes W)$.
The non-trivial additive character $\psi$ then gives rise to the projective representation $\omega_{\psi}$ of $H$ as above.  
\vskip 10pt

 \vskip 10pt

It is now more convenient to consider the symplectic and skew-hermitian cases separately. 
\vskip 15pt
\begin{enumerate}
\item[(i)]  (symplectic case) When $k = k_0$, we have $G(W) = \Sp(W)$.  
In this case, it is known that the projective representation 
$\omega_{\psi}$ of $G(W)$ lifts to a linear representation of the double 
cover $\widetilde{G}(W) = \ws(W)$, the metaplectic group. 
Recalling that 
\[  H = N \rtimes G(W), \]
we thus obtain a unitary representation 
$$\nu_{\psi} = \omega_{\psi} \otimes \lambda$$
of 
\[  \widetilde{H} = N \rtimes \widetilde{G(W)} \] 
in the case when $\dim W^{\perp} >  0$. 
\vskip 5pt

When $W^{\perp} = 0$, so that $W=V$, we have 
\[  N = \{1\} \quad \text{and} \quad H= G(W) = G(V). \] 
In this case, we simply set
\[  \nu_{\psi} =  \omega_{\psi}, \]
which is a representation of $\widetilde{H}$. 
\vskip 5pt

In each case, the representation $\nu_{\psi}$
has GK-dimension $1/2 \cdot \dim(W/k_0)$. Up to conjugation by the 
normalizer of $H$ in $G$,  $\nu_{\psi}$ depends  only on $\psi$ up to the 
action of $(k^\times)^2$.

\vskip 10pt

A particular case of this is worth noting. When $W = 0$, so that $G(W)$ is trivial, the group $H$ is simply the unipotent radical $N$ of the Borel subgroup $P$ and $\nu_{\psi}$ is simply a generic character 
of $N$. This gives a bijection
\[ \xymatrix{ \left \{ \text{$T$-orbits of generic characters of $N$} \right \} \ar@{<->}[d] \\ 
 \left \{ \text{$k^{\times 2}$-orbits of 
nontrivial characters $\psi$ of $k$} \right \}} \] 
in the symplectic case.
\vskip 15pt

\item[(ii)] (skew-hermitian case) When $k \ne k_0$, $G(W) = \U(W)$. 
In this case, the projective 
representation $w_{\psi}$ of $G(W) = \U(W)$ lifts to a linear 
representation of $G(W)$, but when $\dim W > 0$, the lifting is not 
unique:  it requires the choice of a character 
\[  \mu : k^\times \to \CC^\times \] 
whose restriction to $k_0^{\times}$ is the quadratic character $\omega_{k/k_0}$ associated to $k/k_0$ [HKS].  Equivalently, when $k$ is a field, it  requires the choice of a $1$-dimensional, conjugate-dual 
representation  of $WD_k$ with sign $c = -1$. 
Given such a $\mu$, we let $\omega_{\psi, \mu}$ be the corresponding representation of  $G(W)$ and set
$$\nu_{\psi,\mu} = \omega_{\psi, \mu} \otimes \lambda.$$
Hence, we have defined an irreducible unitary representation $\nu_{\psi,\mu}$ of $H = N. \U(W)$ 
when $\dim W^{\perp} > 0$.
\vskip 5pt

When $W^{\perp} = 0$, so that $W=V$, we have
\[  N = \{1 \} \quad \text{and} \quad H = G(W) = G(V). \]
In this case, we simply set
\[  \nu_{\psi,\mu} = \omega_{\psi, \mu}. \]
\vskip 10pt

In each case, the representation $\nu_{\psi,\mu}$  has
GK-dimension $1/2 \cdot \dim(W/k_0)$.  It  depends, up to conjugation by the normalizer of $H$ in $G = \U(V) \times \U(W)$,  on $\psi$ up to the action of $\NN k^\times$ (as well as the choice of $\mu$).
\vskip 10pt

A particular case of this is noteworthy. When $W = 0$ (and $V$ is even dimensional), so that $G(W)$ is trivial, the group $H$ is simply the unipotent radical $N$ of the Borel subgroup $P$ and there is no need to choose $\mu$. 
Hence, $\nu_{\psi,\mu} = \nu_{\psi}$ is simply a generic 
character of $N$. This gives a bijection

\[ \xymatrix{ \left \{ \text{$T$-orbits of generic characters of $N$} \right \} \ar@{<->}[d] \\ 
 \left \{ \text{$\NN k^{\times }$-orbits of 
nontrivial characters $\psi$ of $k_0$} \right \}} \] 

in the even skew-hermitian case.
 
\vskip 15pt

\noindent{\bf Remarks:} Note that when $W = 0$, there is no need to invoke the Weil representation at all. Hence, the above description of generic characters could be carried out for hermitian spaces (of even dimension $2n$) as well. One  would consider the situation
\[  W = 0 \subset V \quad \text{of hermitian spaces.} \]
and note that the homomorphisms 
\[  f: N \longrightarrow k^{n-1} \]
and 
\[ f_0 : N=N_0 \longrightarrow k \]
can still be defined by the same formulas. But now the image of $f_0$ lies in the subspace of trace zero elements of $k$ (as opposed to the subfield $k_0$ when $V$ is skew-hermitian). 
The torus actions on $f$ and $f_0$ are again given by the same formulas. Thus, giving a $T$-orbit of generic characters of $N$ in the even  hermitian case amounts to giving a nontrivial character of $k$ trivial on $k_0$, up to the action of $\NN k^{\times}$. 
\end{enumerate}
\vskip 15pt

This completes our definition of the  representation $\nu$ of $H$.  
\vskip 10pt

As we noted in the course of the discussion above, 
special cases of the pair $(H,\nu)$  give the determination of $T$-orbits of generic characters. 
Recall that if $G = G(V)$ is quasi-split, with Borel subgroup $B = T \cdot N$, then the set $D$ of $T(k_0)$-orbits of generic characters of $N$  is a principal homogeneous space for the abelian group 
$$E = T^{ad}(k)/{\rm Im~} T(k) = {\rm ker}(H^1(k, Z) \to H^1(k, T)).$$
When $G = \GL(V)$ or $\U(V)$ with $\dim V$ odd or $\SO(V)$ with $\dim V$ odd or $\dim V = 2$, the group $E$ is trivial. In the remaining cases, $E$ is a finite elementary abelian 2-group.
In Section \ref{S:vogan-classical}, we have described the $E$-torsor $D$ explicitly for the various classical groups $G(V)$, but did not say how this was done.  Our discussion of $(H, \nu)$ above  
has thus filled this gap, and we record the result in the following proposition for ease of reference.

\vskip 10pt
\begin{prop}  \label{P:generic}
\begin{enumerate}
\item If $V$ is symplectic, $E = k^\times / k^{\times 2}$,  and we have a bijection 
of $E$-spaces 
$$D \longleftrightarrow k^{\times 2}-{\rm orbits \ on \ nontrivial} \ 
\psi:k \to \CC^\times.$$ 

\item If $V$ is hermitian of even dimension, $E = k_0^\times /\NN k^\times$, and we 
have a bijection of $E$-spaces 
$$ D \longleftrightarrow \NN k^\times -{\rm orbits \ on \ nontrivial} \ \psi: 
k/k_0\to \CC$$ 
\item If $V$ is skew-hermitian of even dimension, $E = k_0^\times/\NN k^\times$,
 and we 
have a bijection of $E$-spaces 
$$D \longleftrightarrow \NN k^\times-{\rm orbits \ on \ nontrivial} \ \psi: k_0 
\to \CC^\times$$ 
\item If $V$ is orthogonal of even dimension and split by the quadratic 
algebra $K$, then $E = \NN K^\times /k^{\times 2}$, and we have a bijection of $E$-spaces
$$D \longleftrightarrow \SO(V)-{\rm orbits \ on \ non-isotropic \ lines} \ 
L\subset V, \ {\rm with} \ L^\perp \ {\rm split}.$$ 
\end{enumerate}
\end{prop}

\vskip 15pt

 \section{Bessel and Fourier-Jacobi models for $\GL(n)$}

The construction of the pair $(H,\nu)$ given in the previous section includes the case 
when $k = k_0 \times k_0$ is the split quadratic algebra. In this case, the groups $G(V)$
and $G(W)$ are general linear groups, and it is useful to give a direct construction of $(H ,\nu)$ 
in the context of general linear groups, rather than regarding them as unitary groups of hermitian or skew-hermitian spaces over $k$. We describe this direct construction in this section.
\vskip 10pt

We first give a  brief explanation of how one translates from the context of unitary groups to that of general linear groups. The hermitian or skew-hermitian space $V$ has the form
\[  V = V_0 \times V_0^{\vee} \]
for a vector space $V_0$ over $k_0$. Moreover, up to isomorphism, the hermitian form on $V$ can be taken to be
\[  \langle (x,x^{\vee}), (y, y^{\vee}) \rangle = (\langle x, y^{\vee} \rangle,  \langle y, x^{\vee}\rangle) \in k , \]
whereas the skew-hermitian form on $V$ can be taken to be 
\[  \langle (x,x^{\vee}), (y, y^{\vee}) \rangle =( \langle x, y^{\vee} \rangle, - \langle y, x^{\vee}\rangle) \in k. \]
Then, by restriction to  to $V_0$,  one has an isomorphism
\[  G(V) \cong \GL(V_0). \] 
\vskip 10pt

If $W \subset V$ is a nondegenerate subspace, then $W=W_0 \times W_0^{\vee}$ gives rise to $W_0 \subset V_0$. 
If, further, $W^{\perp}$ is split, and $X \subset W^{\perp}$ is a maximal isotropic subspace, then $X$ has the form
\[  X = X_0 \times Y_0^{\vee} \subset V_0 \times V_0^{\vee} \]
with the natural pairing of $X_0$ and $Y_0^{\vee}$ equal to zero, so that $X_0$ is contained in the kernel of $Y_0^{\vee}$. Writing the kernel of $Y_0^{\vee}$ as $X_0 + W_0$, we see that the isotropic space $X$ determines a decomposition
\[ V_0 =  X_0  + W_0 + Y_0, \] 
with a natural perfect pairing between $Y_0$ and $Y_0^{\vee}$. Then the parabolic subgroup $P(X)$ stabilizing $X$ in $G(V)$ is isomorphic to the parabolic subgroup of $\GL(V_0)$ stabilizing the flag
\[  X_0 \subset  X_0 + W_0 \subset V_0. \]
It is now easy to translate the construction of $(H,\nu)$ given in the previous section to the setting of $W_0 \subset V_0$, and we simply describe the answer below.
\vskip 10pt

\noindent{\bf \underline{Bessel Models for $\GL(n)$}}
\vskip 10pt

In this case, we start with a vector space $V_0$ over $k_0$ with a decomposition
\[V_0 = X_0 + W_0 + E_0 + X_0^\vee ,\]
where $E_0 = \langle e \rangle$ is a line.
Consider the (non-maximal) parabolic subgroup $Q$ stabilizing the flag
\[  X_0 \subset X_0 +W_0+E_0 \subset V_0. \]
It has Levi subgroup
\[  L = \GL(X_0) \times \GL(W_0 + E_0) \times \GL(X_0^{\vee}) \]
and  unipotent radical $U$ sitting in the exact sequence:
\[ \begin{CD}
0 @>>> \Hom(X_0^\vee,X_0) @>>> U @>>> \Hom(X_0^\vee,W_0+E_0) 
+ \Hom(W_0+E_0 ,X_0) @>>> 0.\end{CD} \]
We may write the above exact sequence as:
\[ \begin{CD}
0 @>>> X_0 \otimes X_0 @>>> U @>>> X_0 \otimes 
(W_0+ E_0 ) 
+ (W_0^\vee + E_0^{\vee} )\otimes X_0 @>>> 0,\end{CD}  \]
where $E_0^{\vee} = \langle f \rangle $ is the dual of $E_0$.
\vskip 10pt

Let
\[  \ell_{X_0} : X_0 \longrightarrow k_0 \]
be any non-trivial homomorphism, and let  $U_{X_0} \times U_{X_0^{\vee}}$ be a maximal
unipotent subgroup of $\GL(X_0)\times \GL(X_0^{\vee})$ which fixes $\ell_{X_0}$.  
 On the other hand, let 
 \[  \ell_{W_0} : (W_0 +E_0) +(W_0^{\vee} + E_0^{\vee}) \longrightarrow k_0 \]
 be a linear form which is trivial on $W_0 + W_0^{\vee}$ but non-trivial on $E_0$ and $E_0^{\vee}$. 
 Together, the homomorphisms $\ell_{X_0}$ and $\ell_{W_0}$ give a map
 \[ \ell =   \ell_{X_0} \otimes \ell_{W_0}:  U  \longrightarrow X_0 \otimes (W_0+ E_0 ) 
+ (W_0^\vee + E_0^{\vee} )\otimes X_0 \longrightarrow k_0. \]
Since $\ell$ is fixed by $U_{X_0} \times U_{X_0^{\vee}} \times \GL(W_0)$, we may extend $\ell$ trivially to this group. Thus, we may regard $\ell$ as a map on
 \[  H = U \rtimes ((U_{X_0} \times U_{X_0^{\vee}} ) \times \GL(W_0)). \]
Choose any non-trivial additive character $\psi$ of $k_0$ and any generic character 
\[  \lambda: U_{X_0} \times U_{X_0^{\vee}} \longrightarrow \SS^1, \]
which we may regard as a character of $H$. 
Then the representation $\nu$ of $H$ is defined by
\[  \nu = (\psi \circ \ell) \otimes \lambda. \]
The pair $(H ,\nu)$ depends only on the spaces $W_0 \subset V_0$, up to conjugacy by $\GL(V_0)$. This completes the construction of $(H , \nu)$ in the case when ${\rm codim} W_0$ is odd.

\vskip 20pt

\noindent{\bf \underline{Fourier-Jacobi models for $\GL(n)$}}
\vskip 5pt

 In this case, we consider  a vector space $V_0$ over $k_0$, together with 
a decomposition 
\[ V_0 = X_0 + W_0 + X_0^\vee.  \]
As before, let $Q$ be the parabolic subgroup  stabilizing the flag 
\[  X_0 \subset X_0 + W_0 \subset V_0. \]
Thus $Q$ has Levi subgroup

 
\[  L = \GL(X_0) \times \GL(W_0) \times \GL(X_0^{\vee}), \]
and its unipotent radical $U$ sits in the exact sequence,
\[ 0 \rightarrow \Hom(X_0^\vee,X_0) \rightarrow U \rightarrow \Hom(X_0^\vee,W_0) 
+ \Hom(W_0,X_0) \rightarrow 0, \]
in which $\Hom(X_0^\vee,X_0)$ is central. The group $U$ is completely described
by the natural bilinear map 
\[\Hom(X_0^\vee,W_0) \times \Hom(W_0,X_0) \rightarrow \Hom(X_0^\vee,X_0). \]
Indeed, given a bilinear map 
\[  \langle-, -\rangle : B \times C \rightarrow A, \]
of vector groups,  there is a 
natural central extension of $B \times C$ by $A$ defined by a group structure
on $A \times B \times C$ given by 
\[  (a_1,b_1,c_1)(a_2,b_2,c_2)=(a_1+a_2+ \langle b_1,c_2 \rangle,
b_1+b_2,c_1+c_2). \]
\vskip 5pt

Given a linear map 
\[  \ell_{X_0}: X_0 \rightarrow k_0, \]
let 
\[  U_{X_0} \times U_{X_0^{\vee}} \subset \GL(X_0) \times\GL(X_0^{\vee}), \]
be a maximal unipotent subgroup fixing $\ell_{X_0}$.
 Let 
\[  \lambda: U_{X_0} \times U_{X_0^{\vee}}  \longrightarrow \SS^1 \]
be a generic character, which we may regard as a character of
\[  H = U \rtimes (U_{X_0} \times U_{X_0^{\vee}}  \times GL(W_0) ) \]
via projection onto $U_{X_0} \times U_{X_0^{\vee}} $. 
\vskip 10pt

On the other hand, the homomorphism $\ell_{X_0}$ allows one to define  
a homomorphism from $U$  to the Heisenberg group  $H(W_0+ W_0^\vee)$:
 
$$
\begin{CD}
0 @>>> X_0 \otimes X_0  @>>>  U  @>>> W_0^\vee \otimes X_0 + X_0 \otimes W_0 @>>>0 \\
  @.  @V   VV   @VVV   @VV   V    @. \\
 0 @>>> k_0 @>>>H(W_0^\vee+ W_0) @>>>W_0^\vee + W_0 @>>> 0,
 \end{CD}
 $$
which is clearly equivariant under the action of $U_{X_0}\times U_{X_0^{\vee}} \times  \GL(W_0)$.
  \vskip 10pt

Thus, given any non-trivial additive character $\psi$ of $k_0$, we may consider the
unique irreducible representation of $H(W_0^{\vee} +W_0)$ with central character $\psi$, and regard it as a representation of $U$ using the above diagram. This representation can be extended trivially to $U_{X_0} \times U_{X_0^{\vee}}$, and is realized naturally on the space $\mathcal{S}(W_0)$ of Sxhwarz-Bruhat functions on $W_0$.
For any character $\mu: k_0^{\times} \longrightarrow \CC^{\times}$, one then obtains a Weil representation $\omega_{\psi, \mu}$ of
\[  H =  (\GL(W_0) \times U_{X_0} \times U_{X_0^{\vee}})\ltimes U. \] 
on ${\mathcal S}(W_0)$, where the action of $\GL(W_0)$ is given by
\[  (g \cdot f)(w) = \mu(\det(g)) \cdot f(g^{-1} \cdot w). \]

\vskip 10pt

When $\dim X_0 > 0$, the representation $\nu_{\psi, \mu}$ of $H$ is then given by 
\[  \nu_{\psi,\mu} = \omega_{\psi,\mu} \otimes \lambda. \] 
When $\dim X_0 = 0$, we have $W_0 = V_0$ and we take the representation $\nu_{\psi,\mu}$ of $H = \GL(W_0)$ to be the representation $\omega_{\psi, \mu}$ of $\GL(W_0)$ on $\mathcal{S}(W_0)$ defined above.
In either case, the isomorphism class of $\omega_{\psi,\mu}$ is independent of $\psi$ and the pair $(H, \nu_{\psi,\mu})$ is independent of $\psi$,  up to conjugacy in $\GL(V_0)$. 
This completes the definition of $(H, \nu)$ when ${\rm codim} W_0$ is even.

 \vspace{4mm}

This concludes our direct construction of the pair $(H ,\mu)$ for general linear groups.
  \vskip 15pt

 \section{Restriction Problems and Multiplicity One Theorems}  \label{S:multiplicity-one}

We are now ready to formulate the local restriction problems
 studied in this paper. 

\vskip 10pt

Let $W \subset V$ be as in \S \ref{S:nu}, so that $G = G(V) \times G(W)$ contains the subgroup $H = N \rtimes G(W)$. We have defined a representation $\nu$ of $H$ (or its double cover), which may depend on some auxiliary data such as $\psi$ or $\mu$. Let $\pi = \pi_V \boxtimes \pi_W$ be an irreducible representation of $G$ (or an appropriate double cover). Then the restriction problem of interest is to determine 
\[  \dim_{\CC} \Hom_H(\pi \otimes \overline{\nu}, \CC). \] 
More precisely, we have:
\vskip 5pt

\begin{enumerate}
\item In the orthogonal or hermitian cases, the representation $\nu$ of $H$ depends only on $W\subset V$ and so we set
\[  d(\pi) =  \dim_{\CC} \Hom_H(\pi\otimes \overline{\nu}, \CC) = \dim_{\CC} \Hom_H(\pi,\nu). \] 

\vskip 5pt

In the literature, this restriction problem is usually referred to as a problem about the existence of Bessel models. Indeed, for an irreducible representation $\pi= \pi_V \boxtimes \pi_W$ of $G$, the space $\Hom_H(\pi,\nu)$ is usually called the space of $\pi_W^{\vee}$-Bessel models of $\pi_V$.

\vskip 5pt

\item In the symplectic case, the representation $\nu_{\psi}$ is a 
representation of the double cover $\widetilde{H} = N \rtimes 
\ws(W)$ and depends on a nontrivial additive character $\psi$ of 
$k = k_0$ up to the action of $k^{\times 2}$. In this case, for the above Hom space to be nonzero, the representation $\pi = \pi_V \boxtimes \pi_W$ must be a genuine representation when restricted to $\widetilde{H}$. Hence, we have to take  an irreducible representation 
\[  \widetilde{\pi} = \pi_V \boxtimes \widetilde{\pi}_W \quad \text{of $\Sp(V) \times \ws(W)$} \]
 or 
\[ \widetilde{\pi} = \widetilde{\pi}_V \boxtimes \pi_W \quad \text{of $\ws(V) \times \Sp(W)$}. \]
In this case, we set
\[  d(\tilde{\pi}, \psi) =   \dim_{\CC} 
\Hom_{\widetilde{H}}(\widetilde{\pi} \otimes \overline{\nu}_{\psi}, \CC). \] 
In the literature, this restriction problem is usually referred to as one about  Fourier-Jacobi models.
Indeed, the space 
$\Hom_{\widetilde{H}}(\widetilde{\pi} \otimes \overline{\nu}_{\psi}, \CC)$ is usually called the $(\pi_W^{\vee}, \psi)$-Fourier-Jacobi models of $\pi_V$.
 \vskip 5pt
 
 \item In the skew-hermitian case, the representation $\nu_{\psi,\mu}$ of $H$ depends on a nontrivial additive character $\psi$ of $k_0$, up to the action of $\NN k^{\times}$, and also on the choice of a character $\mu$ of $k^{\times}$ whose restriction to $k_0^{\times}$ is the quadratic character associated to $k/k_0$. In this case, we set
 \[  d(\pi, \mu, \psi) =   \dim_{\CC} \Hom_H(\pi \otimes 
\overline{\nu}_{\psi,\mu}, \CC). \] 
 In the literature, this restriction problem is usually referred to as a problem about the existence of Fourier-Jacobi models in the context of unitary groups.
Indeed, the space 
$\Hom_{H}(\widetilde{\pi} \otimes \overline{\nu}_{\psi,\mu}, \CC)$ is usually called the $(\pi_W^{\vee}, \psi, \mu)$-Fourier-Jacobi models of $\pi_V$.
 \end{enumerate}
 \vskip 10pt

 We remark that in the orthogonal and hermitian cases, since $\nu$ is 1-dimensional and unitary, one has:
 \[   \Hom_H(\pi\otimes \overline{\nu}, \CC) \cong \Hom_H(\pi \otimes \nu^{\vee},\CC) \cong  \Hom_H(\pi,\nu). \]
In the symplectic and skew-hermitian cases over non-archimedean fields, the same assertion holds, even though $\nu$ is infinite-dimensional. However, over archimedean fields, it is only clear to us that:
\[ \Hom_H(\pi, \nu) \subseteq \Hom_H( \pi \otimes \nu^{\vee}, \CC) \cong  \Hom_H(\pi\otimes \overline{\nu}, \CC). \]
The difficulty arises in the subtlety of duality in the theory of topological vector spaces. In any case,
 we work with $\Hom_H(\pi \otimes \overline{\nu}, \CC)$ since this is the space which naturally arises in the global setting. We should also mention that, over archimedean fields, the tensor product $\pi \otimes \overline{\nu}$ refers to the natural completed tensor product of the two spaces (which are nuclear Fr\'{e}chet spaces) and $\Hom(-, \CC)$ refers to continuous linear functionals. For a discussion of these archimedean issues, see [AG, Appendix A].
 \vskip 10pt
 
 A basic conjecture in the subject is the assertion that 
 \[  d(\pi) \leq 1 \]
in the various cases.  Recently, there has been much progress in the most basic 
cases where $\dim W^{\perp} = 0$ or $1$. We describe these in the following theorem.
 \vskip 10pt
 
 \begin{thm} \label{T:mult1}
 Assume that $\dim W^{\perp} = 0$ or $1$. 
 \vskip 5pt
 
 \noindent  (i) In the orthogonal case,  with $G = \O(V) \times \O(W)$, we have
 \[  d(\pi) \leq 1. \]
 \vskip 5pt
 
\noindent  (ii) In the hermitian case (including the case when $k = k_0 \times k_0$), we have
\[  d(\pi) \leq 1. \]

\vskip 5pt

\noindent (iii) In the symplectic case, suppose that $k$ is non-archimedean. Then we have
\[  d(\pi,\psi) \leq 1. \]
\vskip 5pt

\noindent (iv) In the skew-hermitian case (including the case $k = k_0 \times k_0$), suppose that $k$ is non-archimedean and that $\pi = \pi_V \boxtimes \pi_W$ with at least one of $\pi_V$ or $\pi_W$ supercuspidal. Then we have
\[  d(\pi, \mu,\psi) \leq 1. \]
 \end{thm}
 
 \vskip 5pt
 
 \begin{proof}
 The cases (i) and (ii) are due to Aizenbud-Gourevitch-Rallis-Schiffmann [AGRS] in the p-adic case  and to Sun-Zhu [SZ]  and Aizenbud-Gourevitch [AG] in the archimedean case. The case (iii) is a result of Sun [S] in the non-archimedean case (the archimedean case seems to be still open). The methods of [S] could likely be adapted to give (iv) completely in the non-archimedean case.
 However,  we shall show how (iv) can also be deduced from (ii) using theta correspondence.
 \vskip 5pt

Thus, suppose that $W$ is a skew-hermitian space and $\pi_1$ and $\pi_2$ are irreducible representations of $\U(W)$ with $\pi_2$ supercuspidal.
For a hermitian space $V$ of the same dimension as $W$, one may consider the theta lifting between $\U(W)$ and $\U(V)$, with respect 
to the fixed additive character $\psi$ and the fixed character $\mu$.
By [HKS], there is a (unique) $V$ such that 
the theta lift of $\pi_1$ to $\U(V)$ is nonzero; 
in other words, there is an irreducible representation $\tau$ of $\U(V)$ such that $\Theta_{\psi,\mu}(\tau)$ has $\pi_1$ as a quotient. Now consider the seesaw diagram

\[\xymatrix{
\U(W) \times \U(W) \ar@{-}[dd]  &  \U(V\oplus \langle -1 \rangle) \ar@{-}[dd]\\
 & \\
\U(W)  & \U(V) \times \U(\langle -1 \rangle).
}\]
The resulting seesaw identity gives
\begin{align}
 d(\pi_1 \boxtimes \pi_2, \mu, \psi) 
 &=  \dim \Hom_{\U(W)}(\pi_1 \otimes \overline{\omega_{\psi,\mu}},\pi_2^{\vee})   \notag \\
 &\leq \dim \Hom_{\U(W)}(\Theta_{\psi,\mu}(\tau) \otimes \overline{\omega_{\psi,\mu}},\pi_2^{\vee}) \notag \\
 &= \dim \Hom_{\U(V)}(\Theta_{\psi,\mu}(\pi_2^{\vee}),   \tau) \notag \\
 &=\dim \Hom_{\U(V)}(\theta_{\psi,\mu}(\pi_2^{\vee}),   \tau) \notag \\
 &\leq 1,\notag
 \end{align}
 as desired. Here, we have used the assumption that $\pi_2$ is supercuspidal to deduce that $\Theta_{\psi,\mu}(\pi_2^{\vee}) = \theta_{\psi,\mu}(\pi_2^{\vee})$ is irreducible.
 
 \vskip 10pt
 The above seesaw argument works in the case when $k = k_0 \times k_0$ as well, thus completing the proof of (iv). 
 Indeed, this argument shows more generally that
 the multiplicity one results in (ii) and (iv) are equivalent, modulo the issue of whether $\Theta(\tau) = \theta(\tau)$. Similarly, 
the analogous argument for symplectic-orthogonal dual pairs shows that the results of (i) and (iii) are equivalent, with the same caveat on $\Theta$ versus $\theta$.
\end{proof}
 
 \vskip 10pt
 
 \noindent{\bf Remarks:} In the orthogonal case, it is still not known whether the multiplicity one result of Theorem \ref{T:mult1} holds for special orthogonal groups, i.e. for $G = \SO(V) \times \SO(W)$. It is expected but  does not seem to follow from the result for orthogonal groups.
 
\vskip 15pt

 \section{Uniqueness of Bessel Models} \label{S:bessel}
 
 In this section, we show that if $k$ is non-archimedean, the multiplicity one theorem for the general Bessel models can be deduced from Theorem \ref{T:mult1}(i) and (ii) in the orthogonal and hermitian cases. We remind the reader that the case $k = k_0 \times k_0$ is included in our discussion. In particular, the results we state below are valid in this case as well, though we frequently write our proofs only for $k$ a field, and leave the adaptation to the case $k = k_0 \times k_0$ to the reader.  Moreover, in the orthogonal case, we shall work with orthogonal groups, rather than special orthogonal groups, since Theorem \ref{T:mult1} only applies to the former. 
 \vskip 10pt

Thus, we consider the case when $W \subset V$ are orthogonal or hermitian spaces of odd codimension.  Then we have
 \[  W^{\perp} = X +X^{\vee}+ E \]
 where $E = k \cdot e$ is a non-isotropic line and 
 \[  X = \langle v_1, v_2, \cdots ,v_n \rangle \]
is an isotropic subspace with $\dim X= n   > 0$ and dual basis $\{ v_i'\}$ of 
$X^{\vee}$.
With $G = G(V) \times G(W)$ and $H = N \cdot G(W)$, we would like to show that
\[  \dim \Hom_H(\pi_V \otimes \pi_W, \nu) \leq 1 \]
for any irreducible representation $\pi_V \boxtimes \pi_W$ of $G$. 
\vskip 10pt

Let 
 \[ E^- = k \cdot f  \]
denote the rank 1 space equipped with a form which is the negative 
of that on $E$, so that $E+ E^-$ is a split rank 2 space. 
The two isotropic lines in $E + E^-$ are spanned by
 \[    v_{n+1} =  e + f \quad \text{and} \quad v_{n+1}' = \frac{1}{2 \cdot \langle e,e \rangle} \cdot (e -f). \] 
 Now consider the space
 \[  W' = V  \oplus   E^- \]
 which contains $V$ with codimension $1$ and isotropic subspaces 
 \[  Y = X + k \cdot  v_{n+1} = \langle v_1, \cdots,v_{n+1} \rangle \]
 and
 \[  Y^{\vee} = X^{\vee} + k \cdot v'_{n+1} = \langle v_1', \cdots,v_{n+1}' \rangle. \]
 Hence we have
 \[  W' =  Y + Y^{\vee} + W. \]
Let $P = P(Y)$ be the parabolic subgroup of $G(W')$ stabilizing $Y$ and let $M$ be its  
Levi subgroup stabilizing $Y$ and $Y^{\vee}$, so that
\[   M \cong \GL(Y) \times G(W). \]
Let $\tau$ be a supercuspidal representation of $\GL(Y)$ and $\pi_W$ an irreducible representation of $G(W)$ and let
\[  I(\tau, \pi_W) = {\rm Ind}_P^{G(W')} (\tau \boxtimes \pi_W) \]
be the (unnormalized) induced representation of $G(W')$ from the representation $\tau \boxtimes \pi_W$ of $P$.  

\vskip 10pt

Our goal is is to prove the following theorem:
\vskip 5pt

\begin{thm} \label{T:SOVPS}
Assume that $k$ is non-archimedean. 
With the notations as above, we have
\[  \Hom_{G(V)}(I(\tau, \pi_W) \otimes \pi_V, \CC) = \Hom_H(\pi_V \otimes \pi_W, \nu) \] 
 as long as $\pi_V^{\vee}$ does not belong to the Bernstein component of $G(V)$ associated to 
$(\GL(Y') \times G(V'), \tau \otimes \mu)$, where $Y' \subset V$ is isotropic of dimension equal to 
$\dim Y$ with $V=  Y' + {Y'}^{\vee} + V'$ and  $\mu$ is any irreducible representation of 
$G(V')$.
 \end{thm}
\vskip 5pt

\begin{proof}
 We assume that $k$ is a field in the proof and calculate the restriction of $\Pi := I(\tau,\pi_W)$ to $G(V)$ by Mackey's orbit method. 
For this, we begin by observing that $G(V)$ has at most two orbits on the flag
variety $G(W')/P(Y)$ consisting of:
\vskip 5pt

\begin{enumerate} 
\item $n+1$-dimensional isotropic subspaces of $W'$ which are contained
in $V$; these exist if and only if $W$ is isotropic, in which case if $Y'$ is a representative of this closed orbit, then its stabilizer in $G(V)$ is the parabolic subgroup $P_V(Y') = P(Y') \cap G(V)$;
\vskip 5pt

\item $n+1$-dimensional isotropic subspaces of $W'$ which are not contained
in $V$; a representative of this open orbit is the space $Y$ and its stabilizer in $G(V)$ is the subgroup 
$Q = P(Y) \cap G(V)$. 
\end{enumerate}
\vskip 5pt
\noindent By Mackey theory, this gives a filtration on the restriction of $\Pi$ 
to $G(V)$ as follows:
\vskip 5pt
\[  \begin{CD}
0 @>>>  {\rm ind}_{Q}^{G(V)} (\tau \otimes\pi_W)|_Q 
@>>>  \Pi|_{G(V)} @>>>  
{\rm Ind}_{P_V(Y')}^{G(V)} \tau \otimes \pi_W|_{G(V')}  @>>> 0, \end{CD} \] \vskip 5pt

\noindent where the induction functors here are unnormalized.

 \vskip 5pt

By our assumption, $\pi_V^{\vee}$ does not appear as a quotient of the 3rd term of the above short exact sequence and we have
\[  \Hom_{G(V)}({\rm ind}_{Q}^{G(V)}(\tau \otimes\pi_W)|_{Q}, \pi_V^{\vee}) = 
\Hom_{G(V)}(\Pi, \pi_V^{\vee}). \]
It thus suffices to analyze the representations of $G(V)$ which appear on the open 
orbit. For this, we need to determine the group $Q = P(Y) \cap G(V)$ as a subgroup of $G(V)$ and 
$P(Y)$.
\vskip 5pt

 \vskip 5pt

Recall that $W' = Y \oplus W \oplus Y^\vee $, and $V$
is the codimension 1 subspace 
$X  \oplus W \oplus X^\vee \oplus  E$
which is the orthogonal complement of $f$.
It is not difficult to see that as a subgroup of $G(V)$,
\[  Q = G(V) \cap P(Y) \subset P_V(X). \]
Indeed,  if $g \in Q$, then $g$ fixes $f$ and stabilizes $Y$, and we need to show that it stabilizes $X$. 
If $x \in X$, it suffices to show that  $\langle g \cdot x\,  ,  e-f \rangle  = 0$.
But 
\vskip 5pt
\[  \langle g \cdot x\,  , e-f \rangle = \langle x \, , g^{-1} \cdot (e-f) \rangle = \langle x \, ,  g^{-1} \cdot v_{n+1} -2f \rangle = 0, \]
\vskip 5pt
\noindent as desired. 
\vskip 10pt

Now we claim that as a subgroup of $P_V(X)$,
\[  Q =  (\GL(X) \times G(W)) \ltimes N_V(X), \]
where $N_V(X)$ is the unipotent radical of $P_V(X)$. 
To see this, given an element $h \in P_V(X)$, note that $h \in Q$ if and only if  
\[ h \cdot v_{n+1}  \in Y,  \quad \text{or equivalently} \quad  h \cdot e - e \in Y. \] 
We may write
\[  h \cdot e = \lambda \cdot e + w + x, \quad \text{with $w \in W$ and $x  \in X$.} \]
\vskip 5pt
\noindent Then we see that  $h \cdot e - e \in Y$ if and only if $\lambda = 1$ and $w = 0$, so that 
$h$ fixes $e$ modulo $X$ and hence stabilizes $W$ modulo $X$, in which case $h \in (\GL(X) \times G(W)) \ltimes N_V(X)$, as desired.
  \vskip 10pt

Since we are restricting the representation $\tau \boxtimes \pi_W$ of $P(Y)$ to the subgroup 
$Q$, we also need to know how $Q$ sits in $P(Y)$. For this, we have:
\vskip 5pt

\[  \begin{CD}
0 @>>> N(Y) @>>> P(Y) @>>> \GL(Y) \times G(W) @>>> 0 \\
&  &  @AAA     @AAA  @AAA  \\
0@>>> N(Y) \cap Q @>>> Q @>>>  R \times G(W) @>>> 0 
\end{CD} \]
where
\[  R  \subset  \GL(Y) \]
is the mirabolic subgroup which stabilizes the subspace $X \subset Y$ and fixes $v_{n+1}$ modulo $X$.
Note also that  $N(Y) \cap Q \subset N_V(X)$ and
 \[  N_V(X) / (N(Y) \cap Q) \cong \Hom(E,  X). \]
As a consequence, one has:  
\[  (\tau \boxtimes \pi_W)|_Q = \tau|_{R} \boxtimes \pi_W. \]
\vskip 10pt

By a well-known result of Gelfand-Kazhdan, one knows that  
\[  \tau|_R \cong {\rm ind}_{U}^{R} \chi \]
where $U$ is the unipotent radical of the Borel subgroup of $\GL(Y)$ stabilizing the flag
\[  \langle v_1 \rangle \subset \langle v_1, v_2 \rangle \subset \cdots \subset \langle v_1,\cdots ,v_{n+1}\rangle = Y, \] 
and $\chi$ is any generic character of $U$.
Now observe that the pre-image of $U \times G(W)$ in $Q$ is precisely the subgroup $H$ of $G(V)$ and the representation $\chi \boxtimes \pi_W$ of $U \times G(W)$ pulls back to the representation 
$\nu^{\vee} \otimes \pi_W$ of $H$.    
Hence, by induction in stages,
 we conclude that 
\[  {\rm ind}_{Q}^{G(V)}(\tau \otimes\pi_0)|_Q  \cong {\rm ind}_{H}^{G(V)} 
\pi_W \otimes \nu^{\vee}. \]
Thus, by dualizing and 
Frobenius reciprocity, one has
\[  \Hom_{G(V)}(I(\tau, \pi_W), \pi_V^{\vee}) \cong \Hom_H(\pi_V \otimes \pi_W,  \nu). \]
This completes the proof of the theorem.
\end{proof} 

\vskip 15pt

\begin{cor}  \label{C:orthogonal-hermitian}
In the orthogonal or hermitian cases over a non-archimedean $k$, with $W\subset V$ of odd codimension, we have
\[  \dim_{\CC} \Hom_H(\pi ,\nu) \leq 1 \]
for any irreducible representation $\pi$ of $G = G(V) \times G(W)$.
 \end{cor} 
\vskip 5pt

\begin{proof}
We apply the theorem with a choice of the supercuspidal representation $\tau$ satisfying
the conditions of the theorem and such that the induced representation $I(\tau,\pi_W)$ is irreducible. 
\end{proof}

\vskip 15pt

\noindent{\bf Remarks:} In the archimedean case, a recent preprint of Jiang-Sun-Zhu [JSZ] uses an adaptation of the proof of Theorem \ref{T:SOVPS}  gives the containment
 \[  \Hom_H(\pi_V \otimes \pi_W, \nu)  \subset  \Hom_{G(V)}(I(\tau, \pi_W) \otimes \pi_V, \CC). \] 
Namely, to each element on the LHS, [JSZ] constructs an associated element on the RHS, using an explicit integral. This is enough to deduce  the multiplicity one result of Corollary \ref{C:orthogonal-hermitian}from the results of [SZ] and [AG].


 

\vskip 15pt

\section{Uniqueness of Fourier-Jacobi Models} \label{S:FJ}

In this section, we continue with the assumption that $k$ is non-archimedean and our goal is to establish the analog of Theorem \ref{T:SOVPS} in the symplectic and skew-hermitian cases, which 
will imply that 
\[  d(\pi,\psi)  \leq 1. \]
Before coming to the analogous result, which is given in Theorem \ref{T:FJPS}, we need to recall certain structural results about parabolic induction for the metaplectic groups. 
\vskip 5pt

Recall that  if $W$ is a symplectic space, then a 
parabolic subgroup $\widetilde{P}$ in $\widetilde{\Sp}(W)$ is nothing but the inverse
image of a parabolic $P$ in $\Sp(W)$. It is known that the
metaplectic covering splits (uniquely) 
over unipotent subgroups, so for a 
Levi decomposition $P= M \cdot N$, it makes sense to speak of the 
corresponding Levi decomposition 
\[  \widetilde{P}= \widetilde{M} \cdot N \quad \text{ in} \quad 
\widetilde{\Sp}(W). \]
 Furthermore, we note that for a maximal parabolic subgroup $P(X)$ of $\Sp(W)$
 with Levi subgroup
of the form $M= \GL(X)\times \Sp(W_0)$ in $\Sp(W)$, 
\[  \widetilde{M}= \left( \widetilde{\GL}(X)\times \widetilde{\Sp}(W_0)\right) / \Delta\mu_2  \]
where $\widetilde{\GL}(X)$ is a certain two-fold cover of $\GL(X)$ defined as follows. As a set, we write
\[  \widetilde{\GL}(X) = \GL(X) \times \{ \pm 1\}, \]
and the multiplication is given by
\[  (g_1, \epsilon_1) \cdot (g_2,\epsilon_2) = (g_1 g_2, \epsilon_1\epsilon_2 \cdot (\det g_1 , \det g_2)), \]
where $(-,-)$ denotes the Hilbert symbol on $k^{\times}$ with values in $\{\pm1 \}$. 
\vskip 10pt

The two-fold cover $\widetilde{\GL}(X)$ has a natural genuine 1-dimensional character
\[  \chi_{\psi}:  \widetilde{\GL}(X)  \longrightarrow \CC^{\times} \]
 defined as follows. The determinant map gives rise to a natural group homomorphism
 \[  \det: \widetilde{\GL}(X) \longrightarrow \widetilde{\GL}(\wedge^{top}X) = \widetilde{\GL}(1). \]
 On the other hand, one has a genuine character on $\widetilde{\GL}(1)$ defined by
 \[  (a, \epsilon) \mapsto \epsilon \cdot \gamma(a,\psi)^{-1}, \]
  where 
  \[  \gamma(a,\psi)  = \gamma(\psi_a)/\gamma(\psi) \]
  and $\gamma(\psi)$ is an 8-th root of unity associated to $\psi$ by Weil. Composing this character with $\det$ gives the desired genuine character $\chi_{\psi}$ on $\widetilde{\GL}(X)$, which satisfies:
  \[  \chi_{\psi}^2 (g,\epsilon) = (\det(g), -1). \]
 Thus, there is a bijection between the set of irreducible representations of $\GL(X)$ and the set of genuine representations of $\widetilde{\GL}(X)$, given simply by
 \[   \tau \mapsto \tilde{\tau}_{\psi} = \tau \otimes \chi_{\psi}. \]
 Note that this bijection depends on the additive character $\psi$ of $k$.
 Now associated to a representation $\tau$ of $\GL(X)$ and $\pi_0$ of $\widetilde{\Sp}(W_0)$,
one has the representation
\[   \tilde{\tau}_{\psi} \boxtimes \pi_0 \quad \text{ of} \quad  \widetilde{M}. \]
Then one can consider the (unnormalized) induced representation 
\[ I_{\psi}(\tau, \pi_0) = {\rm Ind}_{\widetilde{P}}^{\widetilde{\Sp}(W)}
(\tilde{\tau}_{\psi} \boxtimes \pi_0). \] 
 
 \vskip 10pt

Here is the analog of Theorem \ref{T:SOVPS} in the symplectic case.
\vskip 10pt

\begin{thm} \label{T:FJPS}
Consider  $W= X \oplus W_0  \oplus X^\vee$  with $X \ne 0$ and fix the additive character $\psi$ of the non-archimedean local field $k$. Let 
\begin{itemize}
\item $\tau$ be a supercuspidal representation of $\GL(X)$;
\vskip 5pt

\item $\pi_0$ be a genuine representation 
of $\widetilde{\Sp}(W_0)$;
\vskip 5pt

\item $\pi$ be an irreducible  representation of ${\Sp}(W)$,
\end{itemize}
\vskip 5pt
\noindent and consider the (unnormalized) induced representation $I_{\psi}(\tau,\pi_0)$ of $\widetilde{\Sp}(W)$.
Assume that $\pi^{\vee}$ does not belong to the Bernstein component associated to $(\GL(X) \times \Sp(W_0),  \tau \boxtimes \mu)$ for any representation $\mu$ of $\Sp(W_0)$. 
Then 
\[  \Hom_{\widetilde{\Sp(W)}}(I_{\psi}(\tau,  \pi_0) \otimes \pi,\omega_{W, \psi})
\cong \Hom_H(\pi \otimes \pi_0, \nu_{W, W_0, \psi}). \]
\end{thm}
\vskip 10pt

\begin{proof} 
We shall compute
\[  \Hom_{\widetilde{\Sp(W)}}(I_{\psi}(\tau, \pi_0) \otimes \omega_{W, \psi}^{\vee}, \pi^{\vee}). \]
Let $P(X) = M(X)\cdot N(X)$ be the parabolic subgroup
in $\Sp(W)$ stabilizing the subspace $X$, so that
\[  \widetilde{M}(X) \cong \left( \widetilde{\GL}(X) \times \widetilde{\Sp}(W_0)\right)/\Delta \mu_2. \]
 The Weil representation $\omega_{W, \overline{\psi}} = \omega_{W,\psi}^{\vee}$ has a convenient description as a $\widetilde{P}(X)$-module; this is the so-called mixed model of the Weil representation. This model of $\omega_{W,\psi}^{\vee}$
is realized on the space
\[  \SS(X^{\vee}) \otimes \omega_{W_0, \psi}^{\vee} \]
of Schwartz-Bruhat functions on $X^{\vee}$ valued in $\omega_{W_0,\psi}^{\vee}$.
In particular, evaluation at $0$ gives a $\widetilde{P}(X)$-equivariant map
\[ ev:  \omega_{W,\psi}^{\vee} \longrightarrow \chi_{\overline{\psi}} |{\det}_X|^{1/2} \boxtimes \omega_{W_0,\psi}^{\vee}, \]
where $N(X)$ acts trivially on the target space. In fact, this map is the projection of $\omega_{W,\psi}^{\vee}$ onto its space of $N(X)$-coinvariants.  

\vskip 10pt

On the other hand, to determine the kernel of the map $ev$, 
note that $\GL(X)$ acts transitively on the nonzero elements of 
$X^{\vee}$. Recall that we have fixed a basis 
$\{v_1,\cdots, v_n \}$ of $X$ 
in the definition of the data $(H , \nu_{\psi})$, with dual basis 
$\{ v_1',\cdots, v_n' \}$ of $X^{\vee}$.  
Let $R$ be the stabilizer of $v_n'$ in $\GL(X)$, so that $R$ is the 
mirabolic subgroup of $\GL(X)$ and set
\[  Q = \left( R \times \Sp(W_0) \right) \cdot N(X) \subset P(X)  \]
so that its inverse image in $\widetilde{P}(X)$ is 
\[  \widetilde{Q} = \left( (\widetilde{R} \times \widetilde{\Sp}(W_0))/\Delta \mu_2 \right)   \cdot N(X) \subset \widetilde{P}(X). \]
Then  one deduces the following short 
exact sequence of ${\widetilde{P}(X)}$-modules: 
\vskip 5pt
\[   \begin{CD}
0 
@>>>
 {\rm ind}^{\widetilde{P}(X)}_{\widetilde{Q}} 
 \chi_{\overline{\psi}}|{\det}_X|^{1/2}  \boxtimes \omega_{W_0,\psi}^{\vee} 
@>>>
 \omega_{W,\psi}^{\vee} 
 @>ev>>
\chi_{\overline{\psi}} |{\det}_X|^{1/2}  \boxtimes \omega_{W_0,\psi}^{\vee}
 @>>>
 0, \end{CD} \]
\vskip 5pt

\noindent where the compact induction functor ${\rm ind}$ is unnormalized and
the action of $\widetilde{Q}$ on $\omega_{W_0,\psi}^{\vee}$ is via the Weil representation of $\widetilde{\Sp}(W_0) \cdot N(X)$ with respect to the character $\overline{\psi}$. 
\vskip 10pt

Tensoring the above short exact sequence by $\tilde{\tau}_{\psi} \boxtimes \pi_0$ and then inducing to 
$\widetilde{\Sp}(W)$, one gets a short exact sequence of ${\Sp}(W)$-modules:
\vskip 5pt
\[  \begin{CD}
0\\
@VVV \\
 {\rm ind}^{{\Sp}(W)}_{Q} 
|{\det}_X|^{1/2} \cdot \tau|_{R} \otimes (\pi_0 \otimes \omega_{W_0,\psi}^{\vee}) = A \\
@VVV \\
  I_{\psi}(\tau,\pi_0) \otimes \omega_{W ,\psi}^{\vee}  =B\\ 
@VVV \\
  {\rm Ind}_{P(X)}^{{\Sp}(W)}(\tau \cdot  |{\det}_X|^{1/2} \otimes 
(\pi_0 \otimes \omega_{W_0,\psi}^{\vee})) = C\\
@VVV \\
 0.
\end{CD} \]
\vskip 5pt

By our assumption on $\pi$, 
\[  \Hom_{\Sp(W)}(C, \pi^{\vee}) = 0 \quad \text{and}  \quad \Hom_{\Sp(W)}(B, \pi^{\vee}) = 
\Hom_{\Sp(W)}(A, \pi^{\vee}). \]
Moreover, by a well-known result of Gelfand-Kazhdan, one has
\[   \tau|_R \cong {\rm ind}_{U}^{R} \chi, \]
where $U$ is the unipotent radical of the Borel subgroup of $\GL(X)$ stabilizing the flag
\[  \langle v_1 \rangle \subset \langle v_1, v_2\rangle \subset \cdots\subset \langle v_1,\cdots,v_n \rangle = X \]
and $\chi$ is any generic character of $U$.
Observing that 
\[  H = \left( U \times \Sp(W_0) \right) \cdot N(X), \]
we conclude that
\[  A = {\rm ind}_{H}^{\Sp(W)} (\pi_0 \otimes \nu_{\psi}^{\vee}). \]
Therefore, the desired result follows by Frobenius reciprocity.
 \end{proof}
\vskip 10pt

\begin{cor}  \label{C:symplectic}
In the symplectic case over a non-archimedean local field, we have
\[  \dim_{\CC} \Hom_H (\pi \otimes  \overline{\nu_{\psi}}, \CC)  \leq 1 \]
for any irreducible representation $\pi$ of $G = G(V) \times G(W)$.
 \end{cor}
 \vskip 10pt

 One can prove an analog of Theorem \ref{T:FJPS} in the skew-hermitian case, including the case when $k = k_0 \times k_0$, and deduce the following corollary; we omit the details.
 \vskip 5pt
 
 \begin{cor}  \label{C:skew-hermitian}
 In the skew-hermitian case over a non-archimedean $k$, with $W \subset V$ of even codimension,  we have
\[  \dim_{\CC} \Hom_H (\pi \otimes \overline{\nu_{\psi, \mu}}, \CC)  \leq 1 \]
for any irreducible representation $\pi = \pi_V \boxtimes \pi_W$ of $G = G(V) \times G(W)$ with $\pi_V$ supercuspidal.  
\end{cor}

 \vskip 15pt
 
 \section{Local Conjectures} \label{S:local-conj}
 
 In this section, we propose a conjecture for the restriction problem formulated in Section 
 \ref{S:multiplicity-one}.  
 Recall that we are considering the restriction of irreducible representations $\pi = \pi_V \boxtimes \pi_W$ of $G = G(V) \times G(W)$ to the subgroup $H  = N \cdot G(W) \subset G$. Recall also that, with auxiliary data if necessary, we have defined a unitary representation $\nu$ of $H$ (or sometimes its double cover $\widetilde{H}$), which has  dimension $1$ when $W \subset V$ are orthogonal or hermitian, and has Gelfand-Kirillov dimensional $1/2 \cdot \dim (W/k_0)$ when $W\subset V$ are symplectic or skew-hermitian.
Then we are interested in
 \[  d(\pi) =  \dim_{\CC} \Hom_H(\pi \otimes \overline{\nu}, \CC), \]
 which is known to be $\leq 1$ in almost all cases.
 In this section, we shall give precise criterion for this Hom space to be nonzero, in terms of 
 the Langlands-Vogan parameterization of irreducible representations of $G$. We first note the following
 conjecture, which has been called {\it multiplicity one in $L$-packets}.
 \vskip 10pt

 \begin{con} \label{conj-mult1} There is a unique representation $\pi$ of a relevant group $G' = G(V') \times G(W')$ in each generic Vogan $L$-packet $\Pi_{\varphi}$ of $G$ which satisfies
 \[  \Hom_H(\pi \otimes\overline{\nu}, \CC) \ne 0. \]
 \end{con}
 \vskip 10pt

In two recent preprints [Wa1] and [Wa2], Waldspurger has made substantial progress towards this conjecture in the orthogonal case. Namely, assuming certain natural and expected properties of the characters of representations in a Vogan $L$-packet,  he has shown that the above conjecture holds for tempered packets in the orthogonal case.
There is no doubt that his methods will give the same result in the hermitian case.   
\vskip 10pt

We also note  that when $k = k_0 \times k_0$, 
we have $G \cong \GL(V_0) \times \GL(W_0)$ so that the Vogan 
packets are all singletons. In this case, the above conjecture simply asserts that $\Hom_H(\pi \otimes \overline{\nu}, \CC) \ne 0$ for any irreducible generic representation $\pi$ of $G$. In this case, we have:
\vskip 10pt

\begin{thm} \label{T:GLn}
(i) If $k = k_0 \times k_0$,  then Conjecture \ref{conj-mult1} holds when $\dim W^{\perp} = 0$ or $1$.
\vskip 5pt

\noindent (ii) If $k = k_0 \times k_0$ is non-archimedean, then Conjecture \ref{conj-mult1} holds in general.
 \end{thm}

 \vskip 5pt
 
 \begin{proof}
When $\dim W^{\perp} =0$ or $1$, Conjecture \ref{conj-mult1} is an immediate consequence of the local Rankin-Selberg theory of Jacquet, Piatetski-Shapiro and Shalika ([JPSS] and [JS]). 
Indeed, the local Rankin-Selberg integral gives a nonzero element of $\Hom_H(\pi, \nu)$. When $k$ is non-archimedean, the general case then follows from Theorems \ref{T:SOVPS} and Theorem \ref{T:FJPS}.
\end{proof} 
\vskip 10pt

In each of the remaining cases, we will make the above conjecture more precise by specifying a canonical character $\chi$ of the component group $A_{\varphi}$. The character $\chi$ depends on the choice of a generic character $\theta$ of $G$ (used to make the Langlands-Vogan parameterization canonical) and on the additional data  needed to define the representation $\nu$ when $\epsilon = -1$. 
We then conjecture that the representation $\pi$ in Conjecture \ref{conj-mult1} has parameter 
$\pi = \pi(\varphi, \chi)$ in the Vogan correspondence $J(\theta)$. 
 \vskip10pt
 
 We treat the various cases separately.
 \vskip 10pt
 
 \noindent{\bf \underline{$G = \SO(V) \times \SO(W)$, $\dim W^{\perp}$ odd}} 
 \vskip 5pt
 
 Here the character $\theta$  is determined by the pair of orthogonal spaces $W \subset V$. In view of Proposition \ref{P:generic}, specifying $\theta$ amounts to giving a non-isotropic line $L$ in the even orthogonal space (with $L^{\perp}$ split), and we simply take the line $L$ to have discriminant  equal to the discriminant of the odd space. The representation $\nu$ is also canonical. 
 \vskip 5pt
 
 The $L$-packet $\Pi_{\varphi}$ is determined by a parameter 
 \[  \varphi: WD \longrightarrow \Sp(M) \times \O(N) \]
 with $\dim N$ even. We define 
 \[  \chi = \chi_N \times \chi_M: A_M \times A_N^+ \longrightarrow  \langle \pm 1\rangle, \]
where the characters $\chi_N$ and $\chi_M$ were defined in 
\S \ref{S:char-comp}.
 \vskip 10pt
 
 \noindent{\bf \underline{$G = \U(V) \times \U(W)$, $\dim W^{\perp}$ odd}}
 \vskip 10pt
 
 Here, in view of Proposition \ref{P:generic},  we need to choose a nontrivial character
 \[  \psi_0: k/k_0 \longrightarrow \SS^1 \]
 up to the action of $\NN k^{\times}$ in order to define a 
generic character $\theta_0$ of the even unitary group.  If $\delta$ is the discriminant of the odd hermitian space, then we define 
\[  \theta(x)= \theta( - 2 \cdot \delta \cdot x) \] 
 and use $\theta$ to fix the Vogan parametrization for the even unitary group.
 The representation $\nu$ is canonical.
 \vskip 10pt
 
 The $L$-packet is determined by a parameter 
 \[  \varphi: WD \longrightarrow \GL(M) \times \GL(N) \]
 with $M$ conjugate-symplectic of even dimension and $N$ conjugate-orthogonal of odd dimension. 
 We define:
 \[  \chi = \chi_N \times \chi_M : A_M \times A_N \longrightarrow \langle \pm 1 \rangle, \]
 using the character $\psi_0$ to calculate the local epsilon factors which intervene in the definition of $\chi$. 
 \vskip 15pt
 
 \noindent{\bf \underline{$G = \ws(V) \times \Sp(W)$ or $\Sp(V) \times \ws(W)$, $\dim W^{\perp}$ even}}
 
 \vskip 10pt
 
 Here we need to choose a nontrivial additive character $\psi: k \to \SS^1$ to define a generic character $\theta$ of the symplectic group, 
 the notion of Vogan parameters for the metaplectic group and the representation $\nu_{\psi}$ of $H$.
 \vskip 5pt
 
 The $L$-packet is determined by a parameter
 \[ \varphi :WD \longrightarrow \Sp(M) \times \SO(N) \]
 with $\dim N$ odd. Let 
 \[  N_1 = N \oplus \CC \]
 be the corresponding orthogonal representation of even dimension and define
 \[  \chi = \chi_{N_1} \times \chi_M: A_M \times A^+_{N_1} \longrightarrow \langle \pm 1\rangle. \]
 The group $A_{\varphi}$ is a subgroup of index $1$ or $2$ in $A_M \times A^+_{N_1}$ and we take the restriction of $\chi$ to this subgroup.
 \vskip  15pt
 
 \noindent {\bf \underline{$G = \U(V) \times \U(W)$, $W\subset V$ skew-hermitian and $\dim W \equiv \dim V \equiv 1 \mod 2$}}
 \vskip 10pt
 Here there is a unique orbit of generic character $\theta$. On the other hand, we need to choose 
 \[  \psi : k_0 \to \SS^1 \]
  up to $\NN k^{\times}$ and 
 \[  \mu:  k^{\times}/\NN k^{\times} \longrightarrow \CC^{\times} \]
 to define the representation $\nu_{\psi,\mu}$ of $H$. 
 \vskip10pt
 
 Let $e$ be the discriminant of $V$ and $W$ which is a nonzero element of trace $0$ in $k$, well-defined up to $\NN k^{\times}$. Let
 \[  \psi_0(x) = \psi({\rm Tr}(ex)) \]
 which is a nontrivial  character of $k/k_0$, 
well-defined    up to $\NN k^{\times}$.\vskip 10pt

The $L$-packet has parameter
\[  \varphi: WD \longrightarrow \GL(M) \times \GL(N) \]
with $M$ and $N$ conjugate-orthogonal representations of odd dimension. We define:
\[  \chi = \chi_N \times \chi_{M(\mu^{-1})} = \chi_{N(\mu^{-1})} \times \chi_M : A_M \times A_N \longrightarrow \langle \pm 1\rangle, \]
using $\psi_0$ to calculate the local epsilon actors which intervene in the definition of $\chi$.
\vskip 15pt

\noindent{\bf \underline{$G = \U(V) \times \U(W)$, $W\subset V$ skew-hermitian, $\dim W \equiv \dim V \equiv 0 \mod 2$}}

\vskip10pt
In this case, we must choose $\psi: k_0 \to \SS^1$ to define $\theta$ for both groups, and $\mu: k^{\times}/\NN k^{\times} \to \CC^{\times}$ to define $\nu = \nu_{\psi,\mu}$. 
\vskip 10pt

The parameter of an $L$-packet is
\[ \varphi: WD \longrightarrow \GL(M) \times \GL(N) \]
with $M$ and $N$ conjugate-symplectic representations of even dimension. We define
\[  \chi = \chi_N \times \chi_{M(\mu^{-1})} = \chi_{N(\mu^{-1})} \times \chi_M: A_M \times A_N \longrightarrow \langle \pm 1\rangle. \]
Since the representations both have even dimension, the values of $\chi$ are independent  of the choice of $\psi_0$ used to define the epsilon factors.
\vskip 10pt

\begin{con} \label{conj-character}
Having fixed the Langlands-Vogan parameterization for the group $G$ and its pure inner forms in the various cases above, the unique representation $\pi$ in a generic Vogan packet $\Pi_{\varphi}$ which satisfies $\Hom_H(\pi \otimes \overline{\nu}, \CC) \ne 0$ has parameters 
\[  \pi = \pi(\varphi, \chi) \]
where $\chi$ is as defined above.
\end{con}
\vskip 15pt
It seems likely that the approach used by Waldspurger in [Wa1] and [Wa2] will lead to a proof of this precise local conjecture in the orthogonal and hermitian cases.
\vskip 15pt

\section{Compatibilities of local conjectures}  \label{S:compatible}
In this section, we  verify that the precise conjecture \ref{conj-character} is independent of:
\vskip 5pt

\begin{enumerate}
\item the bijection 
\[  J(\theta) : \Pi_{\varphi} \leftrightarrow \Hom(A_{\varphi}, \pm 1) \]
given by the choice of a generic character $\theta$ of $G$;
\vskip 5pt

\item the scaling of the form $\langle-,-\rangle$ on $V$ and hence $W$, which does not change the groups $G$ and $H$;
\vskip 5pt
\item the data needed to define the representation $\nu$ of $H$. 
\end{enumerate}
\vskip 5pt

\noindent This serves as a check on the internal consistency of the conjecture. Again, we consider the various cases separately. 
\vskip 10pt

In the orthogonal case, the generic character $\theta$ and the representation $\nu$ of $H$ are determined by the pair of spaces $W \subset V$ and are unchanged if the bilinear form on $V$ is scaled by $k^{\times}$. The character $\chi = \chi_N \times \chi_M$ depends only on the Langlands parameter $\varphi: WD \longrightarrow \Sp(M) \times \O(N)$. So our conjecture is internally consistent in this case, as there is nothing to check.
\vskip 10pt

In the hermitian case, both the generic character $\theta$ and  the character 
$\chi= \chi_N \times \chi_M$ depend on a choice of nontrivial $\psi_0 : k/k_0 \to \SS^1$, up to multiplication by $\NN k^{\times}$, while the representation $\nu$ is determined by the spaces $W \subset V$.  
On the other hand, if we scale the hermitian form on $W \subset V$ by an element of $k_0^{\times}$, 
the generic character $\theta$, the representation $\nu$ of $H$ and the character $\chi$ of the component group  are unchanged. 

\vskip 5pt

To see the dependence of our conjecture on the choice of $\psi_0$, 
suppose that  $t$ represents the nontrivial coset of $k_0/\NN k^{\times}$ and let $\theta^t$ be the generic character associated to $\psi_0^t(x) = \psi_0(tx)$.  For $(a,b) \in A_M \times A_N$, we have
\begin{align} 
\chi^t(a,b) &= \epsilon(M^a \otimes N, \psi_0^t) \cdot \epsilon(M \otimes N^b, \psi_0^t) \notag \\
&= \det M^a (t) \cdot \epsilon(M^a \otimes N, \psi_0) \cdot \epsilon(M \otimes N^b, \psi_0) \notag \\
&= (-1)^{\dim M^a} \cdot \chi(a,b) \notag \\
&= \eta(a) \cdot \chi(a,b)  \notag 
\end{align}
Here we have used the facts that $M$ is conjugate-symplectic of even dimension and 
$N$ is conjugate-orthogonal of odd dimension. 
\vskip 5pt

Now if the parameter of $\pi$ under $J(\theta)$ is $(\varphi, \chi)$, then its parameter under $J(\theta^t)$ is $(\varphi,\chi \cdot \eta) = (\varphi, \chi^t)$, according to the desiderata in \S \ref{S:vogan-classical}. Hence our conjecture is independent of the choice of  $\psi_0$ in the hermitian case.
\vskip 10pt

\vskip 10pt
In the symplectic case, we will discuss representations of $G = \ws(W) \times \Sp(V)$; the case of representations of $\Sp(W) \times \ws(V)$ is similar. In this case, we used the choice of an additive character $\psi: k \to \SS^1$, up to multiplication by $k^{\times 2}$, to 
\vskip 5pt

\begin{enumerate}[(i)]
\item  define the notion of  $L$-parameters $M$ for representations of $\ws(W)$;
\vskip 5pt

\item define a generic character $\theta$ for $\Sp(V)$ and a generic character $\widetilde{\theta}$ for $\ws(W)$;
\vskip 5pt

\item define the representation $\nu = \nu_{\psi}$ for $H$.
\end{enumerate}
Note, however, that the character $\chi$ of the component group $A_M$ is independent of the choice of $\psi$.

\vskip 5pt
Suppose that under the $\psi$-parameterization, the parameter $(M,N, \chi)$ corresponds to the representation 
$\widetilde{\pi}$ of $G$, so that our conjecture predicts that
\[  \Hom_H(\widetilde{\pi},  \nu_{\psi}) \ne 0. \]
Now replace the character $\psi$ by $\psi_c$ for $c \in k^{\times}/k^{\times 2}$ and let $\widetilde{\pi}'$  
be the representation of $G$ corresponding to  $(M ,N, \chi)$ under the $\psi_c$-parameterization.
By our construction of the Vogan parameterization for metaplectic groups, it is easy to see that $\widetilde{\pi}'$ is isomorphic to the conjugated representation $\widetilde{\pi}^c$. Thus, our conjecture for the character $\psi_c$ predicts that
\[  \Hom_H(\widetilde{\pi}^c, \nu_{\psi_c}) \ne 0 \quad \text{and hence} \quad  \Hom_H(\widetilde{\pi}, \nu_{\psi_c}^c) \ne 0. \]
Since $\nu_{\psi}^c = \nu_{\psi_c}$, our conjecture is internally consistent with respect to changing 
$\psi$.
\vskip 10pt

Note that the use of $\psi$ in (i) and (ii) above concerns the Vogan parameterization, whereas  its use in (iii) concerns the restriction problem in representation theory.  Hence there is no reason why one needs to use the same character $\psi$ for these two different purposes. 
\vskip 10pt

Suppose that one continues to use $\psi$ for the Vogan parameterization
 in (i) and (ii), but uses the character $\psi_c(x) = \psi(cx)$  to define the representation $\nu_{\psi_c}$ of 
$H$.  Then for a given Vogan packet of $G$ with $\psi$-parameter $\varphi$, one can ask which representation $\pi \in \Pi_{\varphi}$ satisfies $\Hom_H(\pi, \nu_{\psi_c}) \ne 0$. This can be answered using Conjecture \ref{conj-character}, together with Conjecture \ref{conj-meta} from 
\S \ref{S:metaplectic}. We have:
\vskip 10pt

\begin{prop}
Assume the conjectures \ref{conj-meta} and \ref{conj-character}. 
 Let 
\[  \varphi : WD \longrightarrow \Sp(M) \times \SO(N) \]
 be a generic Langlands parameter for $\ws(W) \times \Sp(V)$ relative to  the nontrivial additive character $\psi$ of $k$. Let 
 \[  N_c = N(c) \oplus  \CC  \quad \text{for $c \in k^{\times}/k^{\times 2}$} \]
 Then the unique representation $\pi$ in $\Pi_{\varphi}$  with $\Hom_H(\pi, \nu_{\psi_c}) \ne 0$ 
 corresponds under the bijection $J(\widetilde{\theta} \times \theta)$  to the restriction of the character
 \[  \chi_{N_c} \times \chi_M: A_M \times A_{N_c}^+ \longrightarrow \langle \pm 1\rangle \]
 to the subgroup  $A_{\varphi} = A_M \times A_N^+$, multiplied by the character 
 \[  \eta_c(a) = \det N^a(c) \quad \text{  of $A_N^+$}. \]
 \end{prop}
 \vskip 10pt
 
 \begin{proof}
 Let $\pi$ be the representation whose $\psi$-parameter is
$(M,N, \chi)$, where $\chi$ is as given in the proposition:
\[  \chi = \chi_{N_c} \times \chi_M \cdot \eta_c : A_M \times A_N^+ \longrightarrow \langle \pm 1\rangle. \]
We want to show that 
\[ \Hom_H(\pi, \nu_{\psi_c}) \ne 0. \]
By Conjecture \ref{conj-meta}, the $\psi_c$-parameters of $\pi$ are 
$(M(c), N, \chi')$ with
\[  \chi' = \chi_{N_c} \cdot \eta[c] \times \chi_M \cdot \eta_c^2. \]
Hence, it suffices to show that $\chi'$ is equal to the character predicted by Conjecture \ref{conj-character} relative to $\psi_c$. More precisely, we need to show that
\[  \chi' = \chi_{N_1} \times \chi_{M(c)}. \]
We now calculate this character on an element  
\[  (a',a) \in A_M \times A_N^+ = A_{M(c)} \times A_N^+, \]
using the fact that for $a$ in $C_N^+ \to C_{N_c}^+$, 
\[  N_c^a = N(c)^a. \]
Since this space has even dimension,
 \[  \det N_c^a =\det N^a = \det N_1^a. \]
 Hence we have
 \begin{align}
 \chi'(a',a) &= (\chi_{N_c} \cdot \eta[c])(a') \cdot \chi_M(a) \notag \\
 &= \epsilon(M^{a'} \otimes N_c) \cdot \epsilon(M^{a'} \otimes (\CC \oplus \CC(c))) \cdot \epsilon(M \otimes N_c^a) \cdot \det(N_c^a)(-1)^{\frac{1}{2} \cdot \dim M} \notag \\
 &= \epsilon(M^{a'} \otimes (N(c) + \CC(c))  \cdot \epsilon(M \otimes N^a(c)) \cdot \det(N^a(c))(-1)^{\frac{1}{2} \cdot  \dim M} \notag \\
 &= \epsilon(M(c)^{a'} \otimes N_1)  
 \cdot \epsilon(M(c) \otimes N_1^a) \cdot \det(N_1^a)(-1)^{\frac{1}{2} \cdot \dim M} \notag \\
 &= (\chi_{N_1} \times \chi_{M(c)})(a',a). \notag 
 \end{align}
 \vskip 5pt
 
 \noindent This proves the proposition.
 \end{proof}
 \vskip 10pt

Finally, we consider the skew-hermitian case with $W \subset V$ of even codimension. We consider the two cases: 
\begin{enumerate}
\item[(i)] $\dim W \equiv \dim V \equiv 1 \mod 2$;
\vskip 5pt

\item[(ii)] $\dim W \equiv \dim V \equiv 0 \mod 2$ 
\end{enumerate}
in turn.
\vskip 5pt

Assume first that $\dim W \equiv \dim V \equiv 1 \mod 2$ and the discriminant of $W$ and $V$ is a trace zero element $e \in k$. In this case, the Vogan parameterization is 
completely canonical, given the spaces $W \subset V$. However, the representation $\nu$ of $H$ depends not only on the spaces $W \subset V$ but also on
the choice of an additive character $\psi: k_0 \to \SS^1$ and on the choice of a multiplicative character 
$\mu: k^{\times} \to \CC^{\times}$ which is trivial on $\mathbb{N}(k^{\times})$ but nontrivial on $k_0^{\times}$. Thus, to be completely precise, we denote the group $H$ by $H_{W,V}$ and the representation $\nu$ by $\nu_{W,V, \psi,\mu}$.
Finally, the character $\chi$ of the component group is defined using $\mu$ and the additive character
\[  \psi_0(x) = \psi({\rm Tr}(ex))  \]
of $k$ and depends on $\psi$ up to multiplication by $\mathbb{N}k^{\times}$.
 
 \vskip 10pt

Suppose without loss of generality that the representation $\pi$ with parameter $(\varphi,\chi)$ is one for 
the group $G(W) \times G(V)$, so that our conjecture (for the character $\psi$) predicts that
\[  \Hom_{H_{W,V}}(\pi,  \nu_{W,V, \psi,\mu}) \ne 0.  \]
If $t$ represents the nontrivial coset of $k_0^{\times}/\mathbb{N}k^{\times}$, let $\chi_t$ be the character of $A_{\varphi}$ defined using the character
$\psi_t(x) = \psi(tx)$.  Then we have
\begin{align}
 \chi_t(a,b) &= \epsilon(M^a \otimes N(\mu^{-1}), \psi_0^t) \cdot \epsilon(M \otimes N(\mu^{-1})^b, \psi_0^t) \notag \\
&= (-1)^{\dim M^a}  \cdot (-1)^{\dim N^b} \cdot \chi(a,b) \notag \\
&= \eta(a,b)  \cdot \chi(a,b). \notag
\end{align}
Thus, the representation $\pi_t$ indexed by the character $\chi_t$ is one for the pure inner form $G(W') \times G(V')$. Moreover, the spaces
 $W' \subset V'$ are simply the spaces $tW \subset tV$ obtained from $W \subset V$ by scaling the skew-hermitian forms by $t$. 
Thus, our conjecture (for the character $\psi_t$) predicts that
 \[  \Hom_{H_{tW,tV}}(\pi_t,  \nu_{tW,tV, \psi_t, \nu}) \ne 0. \] 
To see that this is equivalent to the prediction of our conjecture for the character $\psi$, note that $G(W') \times G(V')$ is canonically identified with $G(W) \times G(V)$ as a subgroup of $\GL(W) \times \GL(V)$ and under this identification, one has
$\pi_t = \pi$. Moreover,  we also have $H_{W',V'} = H_{W,V}$ as subgroups of $G(W) \times G(V)$ and 
\[   \nu_{tW, tV, \psi_t, \mu} = \nu_{W,V, \psi, \mu}. \]
This proves that our conjecture is internally consistent with changing $\psi$.
\vskip 5pt

On the other hand, if we replace $\mu$ by $\mu'$, then $\mu' = \mu  \cdot \mu_0$ for some character 
$\mu_0: k^{\times}/ k_0^{\times} \to \CC^{\times}$. Moreover, we have [HKS]
\[  \nu_{\mu', \psi} \cong \mu_0  \cdot \nu_{\mu, \psi}. \]
Hence 
\[ \Hom_H(\pi \otimes \overline{\nu}_{\mu',\psi}, \CC) \cong \Hom_H((\pi \cdot \mu_0^{-1}) \otimes \overline{\nu}_{\mu, \psi}, \CC). \]
Now our conjecture for $\mu'$ says that the left hand side
 of the above is nonzero if and only if $\pi$ has Vogan parameter $(M,N, \chi_{M,N, \mu'})$ with
\begin{align}
  \chi_{M,N, \mu'}(a,b) &= \epsilon(M^a \otimes N({\mu'}^{-1}),\psi)\cdot \epsilon(M \otimes N({\mu'}^{-1})^b, \psi)   \notag \\
   &= \epsilon(M^a \otimes (N \cdot \mu_0^{-1})(\mu^{-1}), \psi) \cdot \epsilon(M \otimes (N \cdot \mu_0^{-1})(\mu^{-1})^b, \psi)\notag \\
 &= \chi_{M, N(\mu_0^{-1}), \mu}(a,b). \notag  
\end{align}
On the other hand, our conjecture for $\mu$ says that the right hand side 
 is nonzero if and only if $\pi \cdot \mu_0^{-1}$ has Vogan parameter 
$(M, N(\mu_0)^{-1}, \chi_{M,N(\mu_0^{-1}), \mu})$. 
Thus, our conjecture for $\mu'$ is equivalent to that for $\mu$. 
 \vskip 10pt
 
 Finally, consider the case when $\dim W \equiv \dim V \equiv 0 \mod 2$. In this case, we need the additive character $\psi: k_0 \to \SS^1$ to specify the Vogan parameterization, and both $\psi$ and $\mu$ to define the representation $\nu_{W,V, \psi,\mu}$. The character $\chi$, on the other hand, is independent of $\psi$ but depends on $\mu$. 
 \vskip 5pt
 
 Suppose that under the $\psi$-Vogan parameterization, the representation $\pi$ corresponding to the character $\chi$ is one for the group $G(V) \times G(W)$, so that our conjecture for $\psi$ predicts that
 \[  \Hom_{H_{W,V}}(\pi \otimes \overline{\nu}_{W,V, \psi,\mu}, \CC) \ne 0. \] 
 If we replace the additive character $\psi$ by $\psi_t$ with $t \in k_0^{\times}$ but 
 $t \notin \mathbb{N} k^{\times}$, then under the $\psi_t$-Vogan parameterization, the character $\chi$
 corresponds to the conjugated representation $\pi^t$ (using an element in the similitude group with  similitude $t$). So our conjecture for $\psi_t$ predicts that
 \[  \Hom_{H_{W,V}}(\pi^t \otimes \overline{ \nu}_{W,V,\psi_t, \mu}, \CC) \ne 0. \]
 But one may check that
 \[  \nu_{W,V, \psi_t, \mu} \cong \nu_{W,V,\psi,\mu}^t, \]
 so that the two predictions (for $\psi$ and $\psi_t$) are consistent with each other. 
 The consistency with changing $\mu$ is similar to the analogous situation  treated above.
 \vskip 5pt

 
 \section{Reduction to basic cases} \label{S:reduction}
 
 In this section, we shall show:
\vskip 5pt

\begin{thm} \label{T:reduction}
Assume that $k$ is a non-archimedean local field. Then
Conjectures \ref{conj-mult1} and  \ref{conj-character} follow from  the basic cases where $\dim W^{\perp} = 0$ or $1$. 
\end{thm}

\vskip 5pt

\begin{proof}
As we shall explain, this is a simple consequence of Theorems \ref{T:SOVPS} and \ref{T:FJPS}. We treat the two cases separately.
 
 \vskip 10pt
 
We first consider the orthogonal and hermitian cases. 
Suppose that $W \subset V$ with 
\[  W^{\perp} = X + X^{\vee} + E \]
where $X = \langle v_1,\cdots ,v_n \rangle$ is nonzero isotropic  and $E$ is a non-isotropic line. 
 Let $M$ and $N$ be $L$-parameters for $G(V)$ and $G(W)$ respectively. 
 We would like to verify Conjectures \ref{conj-mult1} and \ref{conj-character} for 
 the associated Vogan packet $\Pi_M \times \Pi_N$ of $G = G(V) \times G(W)$. We shall exploit 
 Theorem \ref{T:SOVPS} for this purpose.
 \vskip 10pt
 
 Recall the setting of Theorem \ref{T:SOVPS}, where we have set
 \[  W' = V \oplus (-E)  = V \oplus k \cdot f \]
and 
\[  Y = \langle v_1, \cdots, v_n, v_{n+1} \rangle  \]
with $v_{n+1} = e+f$. Then we have
\[  W' = Y + Y^{\vee} + W. \]
 Let $\tau$ be an irreducible supercuspidal representation of $\GL(Y)$ with $L$-parameter $N_{\tau}$. 
 We may assume that for any $\pi_W \in \Pi_N$, the induced representation 
 \[  I(\tau, \pi_W)  = {\rm Ind}_{P(Y)}^{G(W')} (\tau \boxtimes \pi_W) \]
 is irreducible. Then the set 
 \[  \{ I(\tau,\pi_W) : \pi_W \in \Pi_N \} \]
 is simply the Vogan packet associated to the parameter 
 \[  N' =N_{\tau} + N + {^\sigma}N_{\tau}^{\vee}. \]
 Moreover, there is a canonical isomorphism
 \[  A_N \cong A_{N'} \]
and the representations $\pi_W$ and $I(\tau,\pi_W)$ are indexed by the same character of $A_N^+ \cong A_{N'}^+$.
\vskip 10pt

We may further assume that $\tau$ is chosen so that the conditions of Theorem \ref{T:SOVPS} are met.
Then by Theorem \ref{T:SOVPS}, we see that for any $\pi_V \in  \Pi_M$, 
\[ \Hom_{G(V)}(I(\tau, \pi_W) \otimes \pi_V, \CC) = \Hom_H(\pi_V \otimes \pi_W, \CC). \]
Thus, Conjecture \ref{conj-mult1} holds for $\Pi_M \times \Pi_N$ if it holds for $\Pi_{N'} \times \Pi_M$. To see that the same implication holds for Conjecture \ref{conj-character}, it suffices to check that
the character 
\[  \chi_N \times \chi_M \quad \text{of $A^+_M \times A^+_N$} \]
agrees with the character
\[  \chi_{N'} \times \chi_M \quad \text{ of $A_M^+ \times A^+_{N'}$.} \]
For $a \in A^+_M$, we see from definition that
\[   \chi_{N'}(a) = \chi_N(a) \cdot \chi_{N_{\tau} + {^{\sigma}}N_{\tau}^{\vee}}(a), \]
 and it follows by Proposition \ref{P:epsilon-WD}
 \[  \chi_{N_{\tau} + {^{\sigma}}N_{\tau}^{\vee}}(a) = 1 \quad \text{for any $a \in A_M^+$.} \]
 This establishes Theorem \ref{T:reduction} in the orthogonal and hermitian cases.
 \vskip 15pt
 
The symplectic and skew-hermitian cases are handled in a similar way, using  Theorem \ref{T:FJPS}; we omit the details.
\end{proof}

 


 \section{Variant of the local conjecture}  \label{S:variant}

 In this section, we give a variant of the local conjecture \ref{conj-character}.  This variant does not require  the precise parameterization of the members of a Vogan $L$-packet by the characters of the component group, which can be a very delicate issue.  This conjecture is typically what is checked in practice. 
 
 \vskip 15pt
 
 Suppose that $W \subset V$ and we are given an $L$-parameter $M$ of $G(V)$, so that $M$ is a selfdual or conjugate-dual representation of the Weil-Deligne group $WD$ of $k$ with a given sign. 
 As described in \S \ref{S:centralizer}, the component group $A_M$ is an elementary abelian 2-group 
 with a canonical basis $\{a_i\}$, indexed by the distinct isomorphism classes of irreducible summands $M_i$ of $M$ which are of the same type
 as $M$. Hence, we have a canonical isomorphism
 \[  A_M \cong \Z/2\Z \cdot a_1 \times \cdots \times \Z/2\Z \cdot a_k. \]
Here, the elements $a_i$ are such that 
 \[  M^{a_i} \cong M_i. \]
\vskip 10pt

Now consider a representation $\pi$ of $G(V)$ in the Vogan packet $\Pi_M$. Let $\mathcal{L}_W(\pi)$ denote
 the set of generic $L$-parameters $N$ for $G(W)$ such that
 \[  d(\pi, N): = \sum_{\pi' \in \Pi_N} \dim \Hom_H(\pi \boxtimes \pi', \nu) \ne 0. \]
According to our Conjecture \ref{conj-mult1}, one has a partition
\[  \{ \text{generic $L$-parameters of $G(W)$} \} = \bigcup_{\pi \in  \Pi_M} \mathcal{L}_W(\pi). \]
In this context, we have the following variant of Conjecture \ref{conj-character}.
\vskip 10pt

\begin{con} \label{C:refined2}

\noindent (1) Fix $\pi \in \Pi_M$.
For any two $L$-parameters $N$ and $N'$ in $\mathcal{L}_W(\pi)$, we have:
\vskip 5pt
\begin{enumerate}[(i)]
\item in the orthogonal case,
\[  \epsilon (M_i \otimes N, \psi)= \epsilon (M_i \otimes N',\psi) \in \{ \pm 1 \} \quad \text{for any $i$}, \]
where $\psi$ is any nontrivial additive character of $k$;
\vskip 5pt

\item in the unitary case, 
\[  \epsilon(M_i \otimes N, \psi) = \epsilon (M_i \otimes N',\psi) \in \{ \pm 1\} \quad \text{for any $i$}, \]
for any nontrivial additive character $\psi$ of $k/k_0$.

\vskip 5pt

\item in the symplectic case,  with $\nu  = \nu_{\psi}$,
\[  \epsilon(M_i \otimes N, \psi) = \epsilon(M_i \otimes N' ,\psi) \in \{ \pm 1\} \quad \text{for any $i$}. \]
 \vskip 5pt
 
 \item in the skew-hermitian case, with $\nu = \nu_{\psi, \mu}$, 
 \[  \epsilon(M_i \otimes N(\mu^{-1}), \psi) = \epsilon(M_i \otimes N'(\mu^{-1}),\psi) \in \{ \pm 1\} \quad \text{for any $i$}, \] 
where $\psi$ is any nontrivial additive character of $k/k_0$.
 \end{enumerate}
 
 \vskip 10pt
 
 In particular, $\pi$ determines a character $\chi_{\pi}$ on $A_M$, defined by
 \[  \chi_{\pi}(a_i) = \epsilon (M_i \otimes N, \psi) \]
 for any $N \in \mathcal{L}_W(\pi)$ and a fixed $\psi$ appropriate for each of the cases above. 
 
\vskip 10pt

\noindent (2) The map $\pi \mapsto \chi_{\pi}$
gives a bijection (depending on the choice of a character $\psi$) 
 \[  \Pi_M \longleftrightarrow \text{Irr}(A_M). \]
\vskip 10pt

\noindent (3) One has the analogs of (1) and (2) above with the roles of $V$ and $W$ exchanged.
 
 \end{con}

 \vskip 10pt

The above conjecture says effectively that one can 
exploit the restriction problem for $W \subset V$ and use the collection of epsilon factors  described above to serve as parameters for elements of $\Pi_M$. 
In the hermitian and skew-hermitian cases, the character $\chi_{\pi}$ associated to a given $\pi \in  \Pi_M$ is equal to the character $\chi_N$ of $A_M$ defined in \S \ref{S:char-comp}, 
for any $N \in \mathcal{L}_W(\pi)$, provided the additive character $\psi$ is appropriately chosen. In the orthogonal and symplectic cases, these two characters may differ.

\vskip 15pt
 

\section{Unramified parameters}  \label{S:unramified}

In this section, we assume that $k$ is a non-archimedean local field with ring of integers $A$, uniformizing element $\varpi$ and finite residue filed $A/\varpi A$. We will also assume that $A/\varpi A$ 
has characteristic $p>2$, so that the group $A^{\times}/A^{\times 2}$ has order 2.  
\vskip 5pt

In the case when $k$ has a nontrivial involution $\sigma$, we will assume that the action of $\sigma$ on $A/\varpi A$ is also nontrivial. Then $k$ is unramified over $k_0$ and every unit in the subring $A_0$ of $A$ fixed by $\sigma$ is the norm of a unit of $A$. 
\vskip 5pt

In addition, we will only consider additive characters $\psi$ of $k$ which are trivial on $A$ but not on $\varpi^{-1}A$. Then $\psi$ is determined up to translation by a unit in $A$. If we insist that 
$\psi^{\sigma} = \psi^{\pm 1}$, then $\psi$ is determined up to translation by a unit in $A_0$.
We call such additive characters of $k$ unramified.

\vskip 5pt

Let $WD = W(k) \times \SL_2(\CC)$ be the Weil-Deligne group of $k$. A representation 
\[  \varphi : WD \longrightarrow \GL(M) \]
is {\em unramified} if $\varphi$ is trivial on $\SL_2$ and on the inertia subgroup $I$ of $W(k)$.
An unramified representation is determined by the semisimple conjugacy class $\varphi(F)$ in $\GL(M)$. Let $\CC(s)$ denote the one dimensional unramified representation with $\varphi(F) = s \in \CC^{\times}$.
Then any unramified representation $M$ is isomorphic to a direct sum of the form
\[  M = \bigoplus_{i=1}^n \CC(s_i), \quad \text{with $n = \dim M$.} \]
We now determine which unramified representations of $WD$ are selfdual or conjugate-dual (with respect to the unramified involution $\sigma$ of $k$).
\vskip 10pt

\begin{prop}
Assume that $M$ is an unramified representation of $WD$ and is either selfdual or conjugate-dual. 
Then $M$ is isomorphic to a direct sum of the form:
\[  M \cong \oplus_i (\CC(s_i) +\CC(s_i^{-1}))  \oplus  m\cdot \CC(-1) \oplus  n \cdot \CC(1), \]
with $s_i \ne s_i^{-1}$ in $\CC^{\times}$   and $m,n \geq 0$ in $\Z$. 
\vskip 5pt

If $M$ has this form, then we have:
\vskip 5pt

\begin{enumerate}[(i)]

\item $M$ is orthogonal and its centralizer in $\SO(M)$ has component group 
\[  A_M^+ = \begin{cases} 
\Z/2\Z, \text{  if both $m,n > 0$,} \\
1, \text{  otherwise.} \end{cases} \]
\vskip 5pt

\item $M$ is symplectic if and only if  $m \equiv n \equiv 0 \mod 2$, in which case its centralizer in $\Sp(M)$ has component group
\[  A_M = 1. \]
\vskip 5pt

\item $M$ is conjugate-orthogonal if and only if  $m \equiv 0 \mod 2$, in which case its centralizer in $\Aut(M,B)$ has component group
\[ A_M = \begin{cases} 
\Z/2\Z, \text{  if $n > 0$,} \\
1, \text{  otherwise.} \end{cases} \]

\vskip 5pt

\item $M$ is conjugate-symplectic if and only if  $n \equiv 0 \mod 2$, in which case its centralizer in $\Aut(M,B)$ has component group
\[  A_M = \begin{cases} 
\Z/2\Z, \text{  if $m > 0$,} \\
1, \text{  otherwise.} \end{cases} \] 
\end{enumerate}

\end{prop}

\vskip 5pt

\begin{proof}
Since $\CC(s)^{\vee} \cong \CC(s^{-1})$ and $\CC(s)^{\sigma} \cong \CC(s)$, the one dimensional representation $\CC(s)$ is selfdual or conjugate-dual if and only if  $s^2 = 1$. In the selfdual case, 
both $\CC(-1)$ and $\CC(1)$ are orthogonal.
In the conjugate-dual case, $\CC(1)$ is conjugate-orthogonal 
and $\CC(-1)$ is conjugate-symplectic.
Indeed, the character 
\[  \mu: k^{\times}/ \mathbb{N} k^{\times} \longrightarrow \langle \pm 1\rangle \]
defined by 
\[  \mu(\alpha) = (-1)^{\ord_{\varpi}(\alpha)} \]
is nontrivial on $k_0^{\times}$. The proposition follows easily. 
\end{proof}

\vskip 10pt

\begin{prop}
(i) If $M$ and $N$ are two selfdual unramified representations of $WD$ of even dimension, with signs $c_M$ and $c_N$ respectively, then the character 
\[  \chi_N: A_M^+ \longrightarrow \langle \pm 1\rangle \]
defined by 
\[  \chi_N(a) = \epsilon(M^a \otimes N,\psi) \cdot \det M^a(-1)^{\frac{1}{2}\dim N} \cdot \det N(-1)^{\frac{1}{2} \dim M^a} \]
is trivial.
\vskip 10pt

\noindent (ii) If $M$ and $N$ are two conjugate-dual unramified representations with signs $c_M$ and $c_N$ respectively and $\psi^{\sigma} = \psi^{-1}$, then the character 
\[  \chi_N: A_M^+ \longrightarrow \langle \pm 1\rangle \]
defined by
\[  \chi_N(a) = \epsilon(M^a \otimes N, \psi) \]
is trivial.
\end{prop}

\vskip 5pt

\begin{proof}
If $M$ is any unramified representation of $WD$ and $\psi$ is an  unramified additive character, then we have the formulae:
\[  \epsilon(M,\psi) = 1 \quad \text{and} \quad \det M(-1) = 1. \] 
The proposition follows easily from these facts.
\end{proof}
\vskip 10pt

We now turn to the restriction conjectures for unramified generic parameters $\varphi$. Since 
$\chi_N \times \chi_M$ is the trivial character of $A_{\varphi}$, the unique representation in the associated Vogan packet which supports a nonzero Hom space should be the one indexed by the trivial character. In that case, for the purpose of global applications, 
we can make our conjectures more refined. 
\vskip 10pt

Recall that $W \subset V$ is a pair of nondegenerate spaces for the sesquilinear form 
$\langle - , -\rangle$.  and that $W^{\perp}$ is split. We say that an $A$-lattice $L \subset V$ is nondegenerate if
\begin{enumerate}[(1)]
\item $\langle - , - \rangle: L \times L \rightarrow A$;
\vskip 5pt

\item the map $L \rightarrow \Hom(L,A)$ defined by mapping $w$ to 
$f_w(v) = \langle v,w \rangle$ is an isomorphism of $A_0$-modules.
\end{enumerate}

\vskip 5pt

We assume henceforth that there is a nondegenerate $A$-lattice $L \subset V$with the additional property that $L_W = L \cap W$ is a nondegenerate $A$-lattice in $W$. Then the orthogonal complement $L_W^{\perp}$ of $L_W$ in $L$ is a nondegenerate lattice  in $W^{\perp}$, so that
$L_W^{\perp}$ has the form 
\[ Y+Y^{\vee} \quad \text{or} \quad   Y+Y^{\vee} + Ae \] 
with $Y$ isotropic and $Y^{\vee} \cong \Hom(Y,A)$. Moreover, $L = L_W + L_W^{\perp}$.
\vskip 5pt
Under this assumption, the group $G=G(V) \times G(W)$ is quasi-split and split by an unramified extension of $k_0$. Indeed, the subgroup $J=\Aut(L) \times \Aut(L_W)$ is a hyperspecial maximal compact subgroup of $G$. We now construct the subgroup $H$ of $G$ and the unitary representation 
$\nu$ of $H$ using this unramified data.
\vskip 5pt

Write $L = L_W  + L_W^{\perp}$, and define the parabolic subgroup $P_A$ and its unipotent radical $N_A$  using a complete $A$-flag in the isotropic subspace $Y \subset L_W^{\perp}$.  Then 
\[  H_A =N_A \cdot \Aut(L_W)  \]
 gives a model of $H$ over $A_0$ and 
\[ H_A = J \cap H_A(k_0) = J \cap H. \]
\vskip 5pt

In the orthogonal and hermitian cases, the one dimensional representation $\nu$ of $H$ associated to the decomposition 
\[  L_W^{\perp}= Y+Y^{\perp}+Ae  \]  
 and a suitable unramified additive 
character $\psi$ has trivial restriction to the subgroup $J \cap H$. In the metaplectic case, we can define $\nu = \nu_{\psi}$ using an unramified additive character $\psi$ (there are two choices, up to translation by $A^{\times 2}$). In the skew-hermitian case, we define $\nu = \nu_{\psi,\mu}$ 
using an unramified character $\psi$ with $\psi^{\sigma} = \psi$ (which is unique up to translation by $\mathbb{N}A^{\times}$) and the unramified symplectic character $\mu$ associated to the representation $\CC(-1)$.  Then in all cases, the representation $\nu$ of $H$ is $J \cap H$-spherical; it has a unique line fixed by the compact open subgroup $J \cap H$.
\vskip 10pt

Since the group $G$ is quasi-split over $k_0$, we can also define unramified generic characters 
$\theta$ of the unipotent radical $U$ of a Borel subgroup, using the pair $L_W\subset L$ of 
nondegenerate lattices and a suitable unramified additive character $\psi$. Again, the $T$-orbit of 
$\theta$ is unique except in the metaplectic case when there are two unramified orbits. In all cases, the restriction of $\theta$ to the compact open subgroup $J \cap U$ is trivial.
\vskip 10pt

To summarize, if we use unramified data to define the representations $\theta$ of $U$ and $\nu$ of $H$, then the complex vector spaces 
\[  \Hom_{J \cap U} (\CC , \theta) \quad \text{and} \quad \Hom_{J \cap H}(\CC, \nu) \]
both have dimension equal to $1$.
\vskip 10pt

Now let $\varphi$ be an unramified generic parameter and let $\pi$ be the unique $\theta$-generic 
element in the Vogan packet $\pi_{\varphi}$. Then the formula of Casselman and Shalika [CS] shows that
\vskip 5pt

\begin{enumerate}[(i)]
\item  $\Hom_J(\CC, \pi)$ has dimension $1$;
\vskip 5pt

\item the pairing of one-dimensional complex vector spaces 
\[  \Hom_J(\CC, \pi) \otimes \Hom_U(\pi,\theta) \longrightarrow \Hom_{J \cap U}(\CC, \theta) = \CC \]
is nondegenerate.
\end{enumerate}
\vskip 5pt

We conjecture that the same is true for the representation $\nu$ of $H$.
\vskip 10pt

\begin{con}  \label{conj-unramified}
Let $\pi$ be the unique $J$-spherical representation in the Vogan packet $\Pi_{\varphi}$. Then
\vskip 5pt

\begin{enumerate}[(i)]
\item $\Hom_H(\pi,\nu)$ has dimension $1$;

\vskip 5pt

\item the pairing of one-dimensional complex vector spaces
\[  \Hom_J(\CC, \pi) \otimes \Hom_H(\pi,\nu)  \longrightarrow \Hom_{J \cap H}(\CC, \nu) \]
is nondegenerate.
\end{enumerate}
\end{con}
 
\vskip 15pt

Besides the cases treated by Casselman-Shalika [CS], this conjecture has been verified in a large number of cases, which we summarize below.
\vskip 15pt

\begin{thm} \label{T:unramified}
Conjecture \ref{conj-unramified} is known in the following cases:
\vskip 10pt

\begin{enumerate}[(i)]
\item the special orthogonal and hermitian cases;
\vskip 5pt

\item the general linear case, with $\dim W^{\perp} = 1$;
\vskip 5pt
 
\item the symplectic case, with $\dim W^{\perp} = 2$;
\end{enumerate}
\end{thm}

\vskip 5pt

\begin{proof}
The orthogonal case is due to Kato-Murase-Sugano [KMS]. Their proof is extended to the unitary case by Khouri in his Ohio-State PhD thesis [Kh]. 
Parts (ii)  and (iii) are both due to Murase-Sugano [MS1, MS2]. 
 \end{proof}

 \vskip 15pt

\section{Automorphic forms and $L$-functions}  \label{S:automorphic}

The remainder of this paper is devoted to formulating global analogs of our local conjectures. 

\vskip 10pt

Let  $F$ be a global field with ring of ad\`{e}les $\A$ and let
$G$ be a reductive algebraic group over $F$. Then $G(F)$ is a discrete subgroup of the locally compact group $G(\A)$. For simplicity, we shall further assume that the identity component of the center of 
$G$ is anisotropic, so that the quotient space $G(F) \backslash G(\A)$ has finite measure. 
\vskip 10pt

We shall consider the space $\mathcal{A}(G)$ of automorphic forms on $G$, which consists of smooth functions
\[  f : G(F) \backslash G(\A) \longrightarrow \CC \]
satisfying the usual finiteness conditions, except that we do not impose the condition of $K_{\infty}$-finiteness at the archimedean places.
For each open compact $K_f \subset  G(\A_f)$, the space $\mathcal{A}(G)^{K_f}$ has a natural topology, giving it the structure of an LF-space(see [W2]) with respect to which the action of $G(F \otimes \R)$ is smooth.
\vskip 10pt

Let $\mathcal{A}_0(G) \subset \mathcal{A}(G)$ denote the subspace of cusp forms.  An irreducible admissible representation $\pi =\pi_{\infty} \otimes \pi_f$ of $G(\A)$ is cuspidal if it admits a continuous embedding
\[   \pi \hookrightarrow \mathcal{A}_0(G). \]
The multiplicity of $\pi$ in $\mathcal{A}_0(G)$ is the dimension of the space $\Hom_{G(\A)}(\pi, \mathcal{A}_0(G))$, which is necessarily finite.
\vskip 10pt

Suppose now that $G$ is quasi-split, with a Borel subgroup $B = T \cdot U$ defined over $F$. 
A homomorphism $\lambda :  U \longrightarrow \G_a$
is generic if its centralizer in $T$ is equal to the center of $G$.  Composing $\lambda_{\A}$ with a nontrivial additive character $\psi$ of $\A/F$ gives an automorphic generic character $\theta = \psi \circ \lambda$.
Now one may consider the map
\[  F(\theta): \mathcal{A}(G) \longrightarrow \CC(\theta) \]
defined by
\[  f \mapsto \int_{U(F) \backslash U(\A)} f(u) \cdot \overline{\theta(u)} \, du. \]
The map $F(\theta)$ is a nonzero continuous homomorphism of $U(\A)$-modules, 
which is known as the $\theta$-{\it Fourier coefficient}.
If $F(\theta)$ is nonzero when restricted to $\pi \subset \mathcal{A}(G)$, we say that $\pi$ is globally generic with respect to $\theta$.

\vskip 10pt   
  
The notion of automorphic forms can also be defined for nonlinear finite covers $\widetilde{G}(\A)$ of $G(\A)$,
which are split over the discrete subgroup $G(F)$; see [MW]. For the purpose of this paper, we only need to consider this in the context of the metaplectic double cover of $\Sp(W)(\A)$ and so we give a brief description in this case.
\vskip 10pt

Assume that the characteristic of $F$ is not two.
For each place $v$ of $F$, we have a unique nonlinear double cover 
$\widetilde{\Sp}(W)(F_v)$ of $\Sp(W)(F_v)$. If the residual characteristic of $F_v$ is odd,
then this cover splits uniquely over a hyperspecial maximal compact subgroup $F_v$ of $\Sp(W)(F_v)$.
Hence, one may form the restricted direct product
\[  \prod_{K_v} \widetilde{\Sp}(W)(F_v),  \]
which  contains a central subgroup $Z = \oplus_v \mu_{2,v}$. If $Z^+$ denotes the index two subgroup of $Z$ consisting of elements with an even number of components equal to $-1$, then the group  
\[  \widetilde{\Sp}(W)(\A) := \left( \prod_{K_v} \widetilde{\Sp}(W)(F_v) \right) / Z^+ \]
is a nonlinear double cover of $\Sp(W)(\A)$.
It is a result of Weil that this double cover splits uniquely over the subgroup $\Sp(W)(F)$, so that we may speak of automorphic forms on
 $\widetilde{\Sp}(W)(\A)$.
An automorphic form $f$ on $\widetilde{\Sp}(W)(\A)$ is said to be genuine if it satisfies
\[  f(\epsilon \cdot  g) = \epsilon \cdot f(g) \quad \text{for $\epsilon \in \mu_2$.} \]
We denote the space of genuine automorphic forms on $\widetilde{\Sp}(W)(\A)$ by 
$\mathcal{A}(\widetilde{\Sp}(W))$.
\vskip 10pt

If $B = T \cdot U$ is a Borel subgroup of $\Sp(W)$, then the double covering splits uniquely over $U(F_v)$ for each $v$. Hence, in the adelic setting, there is a unique splitting of the double cover over $U(\A)$, and more generally over the adelic  group of the unipotent radical of any parabolic subgroup of 
$\Sp(W)$. As a result, one can define the notion of cusp forms as in the linear case, and we denote the space of such cusp forms by $\mathcal{A}_0(\widetilde{\Sp}(W))$.  
Moreover, if $\theta$ is a generic automorphic character of $U$, then one can define the $\theta$-Fourier coefficient of $f$ in the same way as before. 
\vskip 10pt

For the global analog of our restriction problems, we also need to discuss the notion of automorphic forms on the non-reductive group 
\[  J_W = \Sp(W) \ltimes H(W) \]
where $H(W)$ is the Heisenberg group associated to $W$.
 The group $J_W$ is called the Jacobi group associated to 
 $W$ and 
we shall consider its double cover $\widetilde{J}_W(\A) = \widetilde{\Sp}(W)(\A) \cdot H(W)(\A)$.
The space of automorphic forms $\mathcal{A}_{\psi}(\widetilde{J}_W)$ 
on $\widetilde{J}_W(\A)$ is called the space of Jacobi forms. It consists of certain smooth functions on $J_W(F) \backslash \widetilde{J}_W(\A)$ with central character $\psi$.
\vskip 10pt
 
For our applications, we are interested in a particular automorphic representation of $\widetilde{J}_W(\A)$, namely the automorphic realization of the global Weil representation associated to 
$\psi = \prod_v \psi_v$. Recall that for each place $v$, the group 
\[ \widetilde{J}_W(F_v)  = \widetilde{\Sp}(W)(F_v) \cdot H(W)(F_v) \]
has a local Weil representation $\omega_{\psi_v}$ whose restriction to $H(W)(F_v)$ is the unique irreducible representation with central character $\psi_v$. The restricted tensor product 
\[  \omega_{\psi} = \hat{\otimes}_v \omega_{\psi_v} \]
is the global Weil representation associated to $\psi$. One of the main results of Weil [We] is that there is a unique continuous embedding
\[  \theta_{\psi}: \omega_{\psi} \hookrightarrow  \mathcal{A}_{\psi}(\widetilde{J}_W). \]
Composing $\theta_{\psi}$ with the restriction of functions from $\widetilde{J}_W(\A)$ to $\widetilde{\Sp}(W)(\A)$ gives a
$\widetilde{\Sp}(W)(\A)$-equivariant (but not injective) map
\[  \omega_{\psi} \longrightarrow \mathcal{A}(\widetilde{\Sp}(W)). \]
\vskip 10pt

We now come to the global $L$-functions and epsilon factors  associated to a cuspidal representation $\pi$, following Langlands. To define an $L$-function or epsilon factor, one needs the extra data of a finite dimensional representation $R$ of the $L$-group ${^L}G$. 
If $\pi = \otimes_v \pi_v$ is an automorphic representation and we assume the local Langlands-Vogan correspondence for $G(F_v)$, then each $\pi_v$ determines a local $L$-parameter
\[  \phi_v: WD(F_v) \longrightarrow {^L}G \]
Hence, one has the local $L$-factors $L(R  \circ \phi_v, s)$ and 
one defines the global $L$-function
\[  L(\pi, R, s) = \prod_v L(R  \circ \phi_v, s), \]
which converges when $Re(s)$ is sufficiently large.
One expects that this $L$-function has meromorphic continuation to the whole complex plane and satisfies a functional equation of a standard type. 
Similarly, one has the local epsilon factors
\[  \epsilon_v(\pi, R, \psi, s) = \epsilon(R \circ \phi_v, \psi_v ,s), \]
and one defines the global epsilon factor by
\[  \epsilon(\pi, R, s)   = \prod_v  \epsilon_v(\pi, R,\psi, s). \]
It is a finite product independent of the additive character $\psi$ of $\A/F$. 
\vskip 10pt

The following table gives some examples of $R$ and their associated 
$L$-functions which appear in this paper. When the cuspidal representation $\pi$ is globally generic, the meromorphic continuation of these $L$-functions are known.
\vskip 10pt

\begin{center}
\begin{tabular}{|c|c|c|}
\hline
& &  \\
$G$ & ${^L}G$  & $R$    \\
\hline 
 & &   \\
$\GL(V)$ & $\GL(M)$ & $\Sym^2 M$   \\
\hline
 & &   \\
$\GL(V)$ & $\GL(M)$ & $\wedge^2 M$  \\ 
 \hline 
 & &   \\
$\GL(V/E)$, $E/F$ quadratic & $(\GL(M) \times \GL(M)) \cdot \Gal(E/F)$ & ${\rm As}^{\pm}(M)$  \\  
\hline
 & & \\
$\GL(V) \times \GL(W)$ & $\GL(M) \times \GL(N)$ &  $M \otimes N$   \\
\hline 
& &   \\
$\SO(W) \times \SO(V)$, $\dim W^{\perp}$ odd & $\O(M) \times \Sp(N)$, $\dim M$ even & $M \otimes N$   
\\
\hline
& &   \\
$\Sp(W) \times \widetilde{\Sp}(V)$ & $\SO(M) \times \Sp(N)$, $\dim M$ odd & $M \otimes N$   \\
\hline
& &  \\
$\U(W) \times \U(V)$ & $(\GL(M) \times \GL(N)) \cdot \Gal(E/F)$ & $\Ind_{\widehat{G}}^{{^L}G} (M \otimes N)$     \\
& &  \\
\hline 

\end{tabular}
\end{center}

\vskip 20pt

\section{Global Restriction Problems} \label{S:global}

We are now ready to formulate the global restriction problems. 
We shall change notations slightly from the earlier part of the paper, by replacing the pair of fields
$k_0 \subset k$ in the local setting by $F \subset E$ in the global setting. Hence $\sigma$ is an involution (possibly trivial) on $E$ with $E^{\sigma} = F$, and $V$ is a vector space over $E$ equipped with a sesquilinear form $\langle-,-\rangle$ of the relevant type. The group $G = G(V)$ is then an algebraic group over $F$. Also, we shall include the case $E = F \times F$ in our discussion.

\vskip 10pt

Suppose that we have a pair of vector spaces $W \subset V$
over $E$ equipped with a sesquilinear form of sign $\epsilon$, such that $W^{\perp}$ is split and $\epsilon \cdot (-1)^{\dim W^{\perp}} =-1$. Then we have the groups 
 \[  \begin{cases}
 G = G(V) \times G(W); \\
 H = N \cdot G(W) \end{cases} \]
 over $F$, as defined  in \S \ref{S:classical}. 
 \vskip 5pt
 
 The groups of $F$-points $G(F)$ and $H(F)$ are discrete subgroups of the locally compact adelic groups $G(\A)$ and $H(\A)$ respectively. In the orthogonal case, we assume that if $V$ or $W$ has dimension 2, then it is not split. Then the quotient spaces $G(F) \backslash G(\A)$ and $H(F)\backslash H(\A)$ both have finite measure. We may then consider the space of automorphic forms and cusp forms for these groups, as in \S \ref{S:automorphic}.
 \vskip 10pt
 
In this section, we will consider irreducible tempered representations $\pi$ of $G(\A)$ which occur in the space of cusp forms 
 $\mathcal{A}_0(G)$ on $G(F) \backslash G(\A)$ and study their restriction to $H(\A)$. As in the local setting, when $G$ is quasi-split, we need to introduce an automorphic generic character
 \[  \theta:  U(F) \backslash U(\A) \longrightarrow \SS^1 \]
for the group $G$; this serves to fix the local Langlands-Vogan parameterization at all places $v$ of $F$. In addition, we need to construct an automorphic representation
$\nu$ on $H(F) \backslash H(\A)$ for the restriction problem.  
\vskip 10pt

Assume that $G = G(V) \times G(W)$ is quasi-split. 
 In the orthogonal or symplectic case, we use the spaces $W \subset V$ to naturally define a generic 
 $F$-homomorphism
 \[  \lambda: U \longrightarrow \G_a \]
as in \S \ref{S:nu}. Composing $\lambda_{\A}$ with a nontrivial additive character
\[  \psi: \A/F \longrightarrow \SS^1 \]
gives an automorphic generic character 
\[  \theta = \psi \circ \lambda. \]
In the hermitian or skew-hermitian case, we use the spaces $W \subset V$ to construct a generic homomorphism
\[  \lambda: U \longrightarrow \text{Res}_{E/F}(\G_a) \]
with image in $\text{Res}_{E/F}(\G_a)^{\sigma = -\epsilon}$.  Then, 
in the hermitian case, we compose $\lambda_{\A}$ with a fixed nontrivial additive character
\[  \psi_0: \A_E/ (E+\A) \longrightarrow \SS^1 \]
to obtain an automorphic generic character  $\theta_0 = \psi_0 \circ \lambda_{\A}$. We then translate $\theta_0$ by $-2$ times the discriminant of the odd hermitian space to obtain an automorphic generic character $\theta$ of the even unitary group.
In the skew-hermitian case, we compose $\lambda_{\A}$ with a nontrivial additive character 
\[  \psi: \A/F \longrightarrow  \SS^1 \]
to obtain $\theta = \psi \circ \lambda_{\A}$. 
\vskip 10pt

 Next, we need to define an automorphic version of the representation $\nu$ of $H(\A) = N(\A) \cdot G(W)(\A)$. In the orthogonal and hermitian cases, we define $\nu$ by composing the generic 
 $G(W)$-invariant map
 \[  l : N \longrightarrow \text{Res}_{E/F}(\G_a), \]
 constructed in \S \ref{S:nu} using the spaces $W \subset V$, with a nontrivial additive character $\psi: \A_E/E \rightarrow \SS^1$ and then extending this trivially on $G(W)(\A)$;
 \[  \nu = \psi \circ l_{\A}. \] 
 As in the local case, the choice of $\psi$ is unimportant. Then we define:
 \[  F(\nu) : \mathcal{A}_0(G)  \longrightarrow \CC(\nu) \]
  by
 \[  f \mapsto  \int_{H(F) \backslash H(\A)} f(h) \cdot \overline{\nu(h)} \cdot dh. \]
  The map $F(\nu)$ is called a Bessel coefficient.
  
  \vskip 10pt
  
  In the symplectic and skew-hermitian cases, the representation $\nu$ is infinite-dimensional; so the situation is slightly more involved. Recall from \S \ref{S:nu} that, using the spaces $W \subset V$, we have defined a $G(W)$-invariant generic linear form
  \[  l: N \mapsto \G_a. \]
 Composing this with a nontrivial additive character $\psi: \A/F \to \SS^1$, and extending trivially to $G(W)(\A)$,  we obtain an automorphic character
 \[  \lambda = \psi \circ l_{\A} \]
 of $H(\A)$. On the other hand, we have also defined a homomorphism 
 \[   N \longrightarrow H(W) \]
 where
 \[  H(W) = F \oplus \text{Res}_{E/F}W \]
 is the Heisenberg group associated to $\text{Res}_{E/F}(W)$. 
Thus, we have a homomorphism
\[  H = G(W) \cdot N \longrightarrow J_W := \Sp(\text{Res}_{E/F}(W)) \cdot H(W). \]
As discussed in \S \ref{S:automorphic}, the group $\widetilde{J}_W(\A)$ has a global Weil representation $\omega_{\psi}$ with central character $\psi$, and one has a canonical automorphic realization 
\[  \theta_{\psi}: \omega_{\psi} \hookrightarrow \mathcal{A}(\widetilde{J}_W). \]
It will now be convenient to consider the symplectic and skew-hermitian cases separately. 
 
 \vskip 10pt
 
 In the symplectic case, the map $H \longrightarrow J_W$ defined above gives rise to a map
 \[  \widetilde{H}(\A) \longrightarrow \widetilde{J}_W(\A). \]
By pulling back, one can thus regard $\omega_{\psi}$ as a representation of $\widetilde{J}(\A)$. 
Moreover, the above map gives rise to a natural inclusion
\[  \mathcal{A}(\widetilde{J}_W) \hookrightarrow \mathcal{A}(\widetilde{H}). \]
Composing the automorphic realization $\theta_{\psi}$ with this inclusion realizes $\omega_{\psi}$ as a submodule in $ \mathcal{A}(\widetilde{H})$. Multiplying by the automorphic character $\lambda$ of $H$, one obtains an automorphic realization 
\[   \theta_{\psi}:  \nu_{\psi} = \omega_{\psi} \otimes \lambda  \hookrightarrow \mathcal{A}(\widetilde{H}).  \]
Now we can define the map
\[  F(\nu_{\psi}): \mathcal{A}_0(G) \otimes \overline{\nu_{\psi}} \longrightarrow \CC \]
by
\[  f  \otimes \phi \mapsto  \int_{H(F)\backslash H(\A)} f(h) \cdot \overline{\theta_{\psi}(\phi)} \, dh. \]
The map $F(\nu_{\psi})$ is called a Fourier-Jacobi coefficient. 
\vskip 10pt

In the skew-hermitian case, we choose an automorphic character 
\[  \mu: \A_E^{\times}/E^{\times} \longrightarrow \CC^{\times} \]
satisfying 
\[  \mu|_{\A^{\times}} = \omega_{E/F}. \]
Then one obtains a splitting homomorphism 
\[  s_{\psi,\mu} : H(\A) \longrightarrow \widetilde{J}_W(\A). \]
Using $s_{\psi,\mu}$, one may pull back the global Weil representation $\omega_{\psi}$ to obtain a representation $\omega_{\psi,\mu}$ of $H(\A)$. As above, one also obtains an automorphic realization
\[ \theta_{\psi,\mu}:  \nu_{\psi,\mu} = \omega_{\psi, \mu} \otimes \lambda \longrightarrow \mathcal{A}(H). \]
Thus, we can define the map
\[  F(\nu_{\psi, \mu}): \mathcal{A}_0(G) \otimes \overline{\nu_{\psi, \mu}} \longrightarrow \CC \]
by
\[  f  \otimes \phi \mapsto  \int_{H(F)\backslash H(\A)} f(h) \cdot \overline{\theta_{\psi,\mu}(\phi)} \, dh. \]
The map $F(\nu_{\psi,\mu})$ is called a Fourier-Jacobi coefficient in the context of unitary groups.
\vskip 10pt

Now the global restriction problem is:
\vskip 5pt

\noindent {\em Determine whether the map $F(\nu)$ defined in the various cases above is nonzero when restricted to a tempered cuspidal representation $\pi$ of $G(\A)$.}
\vskip 15pt

\section{Global conjectures: central values of $L$-functions} \label{S:global-conj}

To formulate our global conjectures to the restriction problem of the previous section, we need to introduce a distinguished symplectic representation $R$ of the $L$-group ${^L}G$ over $F$. We set
\[  R = \begin{cases}
M \otimes N \quad \text{ in the orthogonal and symplectic cases;} \\
\text{  } \\
\text{Ind}_{\widehat{G}}^{{^L}G} (M \otimes N), \text{ in the hermitian case;} \\
\text{ } \\
\text{Ind}_{\widehat{G}}^{{^L}G} (M \otimes N(\mu^{-1})), \text{ in the skew-hermitian case.} 
\end{cases} \]
This representation $R$ was already introduced in the table at the end of \S \ref{S:automorphic}, except that in the skew-hermitian case, we
incorporate the character $\mu$ used in the definition of $\nu$. It is this twist by $\mu$ that makes $R$ a symplectic representation in the skew-hermitian case (by Lemmas \ref{L:1-dim}  and \ref{L:ind-tensor}).  As explained in \S \ref{S:automorphic}, 
we can then speak of the $L$-function $L(\pi, R, s)$ and the global epsilon factor $\epsilon(\pi, R, s)$ for any automorphic representation $\pi$ of $G$.
\vskip 10pt

The first form of our global conjecture is:

\vskip 15pt

  \begin{con} \label{conj:global1}
  Let $\pi$ be an irreducible tempered representation of $G(\A)$ which occurs with multiplicity one in the space $\mathcal{A}_0(G)$ of cusp forms on $G(F) \backslash G(\A)$. Then the following are equivalent:
  \vskip 5pt
  
  \noindent (i) the restriction of the linear form $F(\nu)$ to $\pi$ is nonzero;
  \vskip 5pt
  
  \noindent (ii) the complex vector space $\Hom_{H(\A)}(\pi,\nu)$ is nonzero and the $L$-function
  $L(\pi, R, s)$ does not vanish at $s = 1/2$, which is the center of the critical strip;
  \vskip 5pt
  
  \noindent (iii) the complex vector spaces $\Hom_{H(F_v)}(\pi_v,\nu_v)$ are nonzero for all $v$ and $L(\pi,R, 1/2) \ne 0$.
  \end{con}
  \vskip 10pt

Let us make some remarks about this conjecture.
  \vskip 10pt
  
When $E = F \times F$, with $G \cong \GL(V_0) \times \GL(W_0)$, then the $L$-function $L(\pi, R,s)$ is the product
\[  L(\pi, R,s) = L(\pi_V \otimes \pi_W, s) \cdot L(\pi_V^{\vee} \times \pi_W^{\vee}, s)  \]
of two Rankin-Selberg $L$-functions, so that
\[  L(\pi, R, 1/2) = |L(\pi_v \times \pi_W, 1/2)|^2. \]
In this case, the conjecture is known when $\dim W^{\perp}  =1$. Indeed, this is an immediate consequence of the integral representation  of the global Rankin-Selberg  $L$-function $L(\pi_V \times \pi_W, 1/2)$ [JPSS]. The general case seems to be open.
  \vskip 10pt
  
More generally,  under the assumption that $\pi$ is globally generic, the implication (i) $\Longrightarrow$ (ii)  has been shown by Ginzburg-Jiang-Rallis in a series of papers for the various cases [GJR1,2,3].
  Moreover, in the hermitian case with $\dim W^{\perp} = 1$, an approach to this conjecture via the relative trace formula has been developed by Jacquet and Rallis [JR].
 \vskip 10pt

Further,  one expects a refinement of Conjecture \ref{conj:global1}   in the form of an exact formula relating $|F(\nu)|^2$ with the central value $L(\pi,R, 1/2)$. 
 Such a refinement has been formulated by Ichino-Ikeda [II] in the orthogonal case, with $\dim W^{\perp} = 1$. In the analogous setting for the hermitian case, the formulation of this refined conjecture is the ongoing thesis work of N. Harris at UCSD.
\vskip 15pt

As formulated, the global conjecture \ref{conj:global1} is essentially independent of the local conjectures \ref{conj-mult1} and 
\ref{conj-character}. Rather, they complement each other, since the local non-vanishing in Conjecture \ref{conj:global1} is governed by our local conjectures. From this point of view, the appearance of the particular central $L$-value $L(\pi, R, 1/2)$ may not seem very well-motivated. However, as we shall explain in the next two sections, if we examine the implications of our local conjectures in the framework of the Langlands-Arthur conjecture on the automorphic discrete spectrum, the appearance of 
$L(\pi, R, 1/2)$ is very natural.
\vskip 10pt

For example,
observe that the global conjecture \ref{conj:global1} in the symplectic/metaplectic case involves the central $L$-value
of the symplectic representation $R = M \otimes N$ with $M$ symplectic and $N$ odd orthogonal, whereas in the local conjecture \ref{conj-character}, it is the epsilon factor associated to $M \otimes (N \oplus \CC)$ which appears. So in some sense, the global conjecture is less subtle than the local one. The explanation of this can be found in \S \ref{S:central}, as a consequence of the Langlands-Arthur conjecture (or rather its extension to the metaplectic case).

\vskip 10pt

Finally,
we highlight a particular case of the conjecture. As we explained in \S \ref{S:nu}, special cases of the data $(H,\nu)$ are
automorphic generic characters $\nu = \theta$ on $H = U$. These cases are highlighted in the following table, and arise when the smaller space $W$ is either $0$ or $1$-dimensional.  
\vskip 10pt

\begin{center}
\begin{tabular}{|c|c|c|c|}
\hline 
& & & \\
$G(V)$ & $\dim W$ & $\widehat{G(W)}$ & $N$ \\
\hline
& & & \\
odd orthogonal & $0$ & $\SO(0)$ & $0$ \\
\hline 
 & & & \\
 even orthogonal & $1$ & $\Sp(0)$ & $0$ \\
 \hline
 & & & \\
 symplectic & $0$ & $\Sp(0)$  & $0$ \\
 \hline
 & & & \\
 metaplectic & $0$ & $\SO(1)$ & $\CC$ \\
 \hline
 & & & \\
 odd hermitian & $0$ & $\GL(0)$ & $0$ \\
 \hline
 & & & \\
 even skew hermitian & $0$ & $\GL(0)$ & $0$ \\
 \hline
 \end{tabular}
 \end{center}
\vskip 10pt

As one sees from the table, in all except the metaplectic case, $N = 0$ so that $R = 0$ and $L(\pi, R, s)$ is identically 1. In the metaplectic case, $N = \CC$ so that $R = M$ and $L(\pi, R, s)$ is the standard $L$-function $L(\pi, s)$.
Hence Conjecture \ref{conj:global1} specializes to the following conjecture in these degenerate cases.   
  \vskip 10pt

 \begin{con}  \label{con:generic}
 Let $\pi$ be an irreducible tempered representation of $G(\A)$ which occurs with multiplicity one in the 
 space $\mathcal{A}_0(G)$ of cusp forms on $G(F) \backslash G(\A)$ and let $\theta$ be an automorphic generic character for $G$. Then, when $G$ is linear, the following are equivalent:
 \vskip 5pt
 
 \begin{enumerate}[(i)]
\item  the restriction of the map $F(\theta)$ to
 $\pi$ is nonzero;
 \vskip 5pt
 
 \item the complex vector space $\Hom_{U(\A)}(\pi, \theta)$ is nonzero.
 \vskip 5pt
 
 \item the complex vector spaces $\Hom_{U(F_v)}(\pi_v, \theta_v)$ are nonzero for all places $v$ of $F$.
 \end{enumerate}

When $G = \widetilde{\Sp}(V)(\A)$ is metaplectic, we fix an additive character $\psi$ of $\A/F$ which determines an automorphic generic character $\theta$ and also gives the notion of  Langlands-Vogan parameters. For any element $c \in F^{\times}/ F^{\times 2}$, let 
let $\chi_c$ be the associated quadratic character of $\A^{\times}/ F^{\times}$ and let  
$\theta_c$ denote the translate of $\theta$ by $c$. Then the following are equivalent:
\vskip 10pt

\begin{enumerate}[(a)] 
\item  the restriction of the map $F(\theta_c)$ to
 $\pi$ is nonzero;
 \vskip 5pt
 
 \item  the complex vector space $\Hom_{U(\A)}(\pi, \theta_c)$ is nonzero and 
 $L(\pi \otimes \chi_c, 1/2) \ne 0$;
 \vskip 5pt
 
\item the complex vector spaces $\Hom_{U(F_v)}(\pi_v, \theta_{c,v})$ are nonzero for all places $v$ of $F$ and $L(\pi \otimes \chi_c, 1/2) \ne 0$.
 \end{enumerate}
  \vskip 5pt

  \noindent If the above conditions hold, the space $\Hom_{U(\A)}(\pi,\theta)$ has dimension $1$ and $F(\theta)$ is a basis.  Moreover, the adjoint $L$-function $L(\pi, Ad, s)$ of $\pi$ is regular and nonzero at $s = 1$ (which is the edge of the critical strip). 

 \end{con}
  \vskip 10pt
  
  In the metaplectic case, with $\dim V = 2$, the above conjecture is known by the work of Waldspurger [Wa1].
 \vskip 15pt

\section{Global $L$-parameters and Multiplicity Formula}  \label{S:arthur}

  In this section, we review the notion of global $L$-parameters in the context of a fundamental
conjecture of Langlands and Arthur [A1, A2], concerning multiplicities of representations in the automorphic discrete spectrum. We will only present this conjecture for tempered representations, and will also discuss its simplification for the classical groups considered in this paper. 
In the next section, we shall re-examine the global conjecture \ref{conj:global1} in the framework of the Langlands-Arthur conjecture. We assume that $E$ is a field henceforth.
\vskip 10pt

Let $G$ be a connected reductive group over the global field $F$, 
and assume that the quotient space $G(F)\backslash G(\A)$ has 
finite volume. The Langlands-Arthur conjecture gives a description 
of the decomposition of the discrete spectrum $L^2_{\rm disc}(G(F) \backslash G(\A))$ or equivalently the space $\mathcal{A}^2(G)$ of square-integrable automorphic forms. We shall only describe this conjecture for  the tempered part of the discrete spectrum, which we denote by $L^2_{\rm disc, temp}(G)$. Note that $L^2_{\rm disc, temp}(G)$ is necessarily contained in the cuspidal spectrum by a result of Wallach [W3].
\vskip 10pt

Suppose that $G_0$ is the quasi-split inner form of $G$ over $F$, with a Borel-subgroup $B = T \cdot U$. Fix  an automorphic generic character $\theta = \otimes_v \theta_v$ of $U$ as in the previous section. We fix an integral structure on $G$, which determines a hyperspecial maximal compact subgroup $J_v \subset G_0(F_v)$ for almost all finite places $v$, as in \S \ref{S:unramified}. If $G=G(V)$ is a classical group, such an integral structure is given by fixing a lattice $L \subset V$.
\vskip 10pt

\noindent{\bf \underline{The Langlands-Arthur Conjecture}}
\vskip 10pt

\noindent (1) For any pure inner form $G = G(V)$ of $G_0$, there is a decomposition
\[  L^2_{\rm disc, temp}(G) = \widehat{\bigoplus}_{\phi} L^2_{\phi}(G), \]
where $\phi$ runs over the discrete global $L$-parameters and each $L^2_{\phi}$ is a
$G(\A)$-submodule. The precise definitions of these objects are given as follows.
\vskip 5pt

By definition, a discrete global $L$-parameter is a homomorphism
\[  \phi:  L_F \longrightarrow {^L}G = {^L}G_0 = \widehat{G} \rtimes W_F\]
such that its centralizer in the dual group $\widehat{G}$ is finite.
These parameters  are taken up to conjugacy by the dual group $\widehat{G}$.
Moreover, $L_F$ is the hypothetical Langlands group of $F$ -- the global analog of the Weil-Deligne group $WD(F_v)$ -- whose existence is only conjectural at this point. One postulates however that 
there is a natural surjective map
\[  L_F \longrightarrow W_F \quad (\text{the Weil group of $F$}), \]
and the projection of $\phi$ to the second factor $W_F$ in ${^L}G$ is required to be this natural surjection. Moreover, one postulates that
for each place $v$ of $F$, there is a natural conjugacy class of embedding
\[ WD(F_v)  \longrightarrow L_F. \]
\vskip 10pt

Assuming the above, one may attach the following data to a given discrete global $L$-parameter $\phi$:
\vskip 10pt

\begin{enumerate}
\item[(i)] a global component group 
\[  A_{\phi} =  Z_{\widehat{G}}({\rm Im}(\phi)), \]
which is finite by assumption.
\vskip 5pt

 \vskip 5pt

\item[(ii)] for each place $v$ of $F$, a local $L$-parameter
\[ \begin{CD}
 \phi_v : WD(F_v) @>>> L_F @>\phi>> {^L}G_0
 \end{CD}  \]
for the local group $G_{0,v}$, such that for almost all $v$, $\phi_v$ is unramified.
This gives rise to a natural map of component groups:
\[  A_{\phi} \longrightarrow A_{\phi_v}.\]
One thus has a diagonal map
\[  \Delta: A_{\phi} \longrightarrow \prod_v A_{\phi_v}. \]
\vskip 5pt

\item[(iii)]  for each place $v$,  the local Vogan packet $\Pi_{\phi_v}$ of irreducible representations of the pure inner forms $G(F_v)$, together with a bijection
\[  J(\theta_v): \Pi_{\phi_v} \leftrightarrow \text{Irr}(A_{\phi_v}) \]
specified by the local component $\theta_v$ of  the automorphic generic character 
$\theta$. For an irreducible character $\eta_v$ of $A_{\phi_v}$, we denote the corresponding representation in $\Pi_{\phi_v}$ by $\pi_{\eta_v}$.
In particular, the representation corresponding to the trivial character of $A_{\phi_v}$ is a representation of $G_0(F_v)$ and 
for almost all $v$,  it is spherical with respect to the hyperspecial maximal compact subgroup $J_v$.
 \vskip 5pt

 \item[(iv)] a global Vogan packet 
\[  \Pi_{\phi}= \{ \pi_{\eta} = \bigotimes_v \pi_{\eta_v}:  \, \text{$\pi_{\eta_v}  \in \Pi_{\phi_v}$ and $\eta_v$ is trivial for almost all $v$} \}. \]
In particular, the representations in the global packet are indexed by irreducible characters 
\[  \eta = \otimes_v \eta_v  \quad \text{of } \quad \prod_v A_{\phi_v}. \]
If $\pi_{\eta_v}$ is a representation of $G_{\eta_v}(F_v)$, then $\pi_{\eta}$ is a representation of 
the restricted direct product 
\[  G_{\eta} := \prod_{J_v} G_{\eta_v}(F_v). \]
Note, however, that the group $G_{\eta}$ need not be the 
adelic group of a pure inner form of $G_0$. For example, in the classical group case, $G_{\eta}$ need not be associated to a space $V$ equipped with a relevant sesquilinear form over $F$. If $G_{\eta} = G(V)(\A)$ for some space $V$ over $F$, 
we shall call the representation $\eta$ coherent.
\vskip 5pt

\item[(v)] for each $\pi_{\eta} \in \Pi_{\eta}$, a non-negative integer
\[  m_{\eta} =  \langle \Delta^*(\eta), 1 \rangle_{A_{\phi}}, \]
where the expression on the right denotes the inner product of the two characters of the finite group $A_{\phi}$. Thus $m_{\eta}$ is the multiplicity of the trivial character of $A_{\phi}$, in the representation
obtained by restriction of the tensor product of the representations $\eta_v$ to the diagonal.
If $\eta$ is not coherent, then one can show that $m_{\eta}$ is equal to zero.  When $\eta$ is coherent, so the adelic group $G_{\eta}$ is defined over $F$, the Langlands-Arthur conjecture for tempered representations predicts that $m_{\eta}$ is the multiplicity of the representation $\pi_{\eta}$ in the discrete spectrum of $G_{\eta}$
 \end{enumerate}

\vskip 5pt
With the above data, we have:
\vskip 10pt

\noindent (2) As $G$ runs over all pure inner forms of $G_0$ over $F$, there is an equivariant decomposition:
\[  \bigoplus_G L^2_{\phi}(G) =  \bigoplus_{\eta}  m_{\eta} \cdot \pi_{\eta}. \]
\vskip 5pt

\noindent  
We denote the representation in $\Pi_{\phi}$ associated to the trivial character by $\pi_0$. It is a representation of $G_0(\A)$  and is the unique representation in $\Pi_{\phi}$ which is abstractly 
$\theta$-generic. According to the multiplicity formula in (v), its multiplicity  in $L^2_{\phi}(G_0)$ is $1$, and the conjecture \ref{con:generic} then says that $\pi_0$ has a nonzero $\theta$-Fourier coefficient.
\vskip 10pt

Though the above conjecture of Langlands and Arthur is extremely elegant, it has a serious drawback: the group $L_F$ is not known to exist. However, in the case of the classical groups considered in this paper, one can present the conjecture on multiplicities in a way that avoids mentioning the group $L_F$. We do this below. In the case of classical groups, there is a further simplification, as the component groups $A_{\phi_v}$ are all elementary abelian $2$-groups. In particular, the representations $\eta_v$ are all $1$-dimensional, so their restricted tensor product $\eta$ also has dimension $1$. Hence the predicted multiplicity $m_{\eta}$ is either zero or one, the latter case occurring when $\eta$ has trivial restriction to the diagonal. In the general case, the groups $A_{\phi_v}$ can be non-abelian, and both the dimension of the representation $\eta$ and the dimension $m_{\eta}$ of its $A_{\phi}$-invariants can be arbitrarily large.
\vskip 5pt

We now specialize to the case where $G = G(V)$ is a classical group.  Let $G_0 = G(V_0)$ be the quasi-split inner form. Arguing exactly as we did in the local case, one sees that giving a global $L$-parameter for $G$
 \[  \phi: L_F \longrightarrow {^L}G_0 \]
is equivalent to giving a representation
\[  \varphi: L_F \longrightarrow \GL(M) \]
which is selfdual or conjugate-dual with a specific sign $b$.
 The requirement that $\phi$ is discrete then translates to the requirement that as a representation of $L_E$, 
\[  M \cong  \bigoplus_i M_i \]
where each $M_i$ is selfdual or conjugate-dual with the same sign $b$ as $M$ and $M_i \ncong M_j$ if $i \ne j$. In this case, the global component group is the 2-group
\[  A_{\phi} =A_{\varphi} = \prod_i (\Z/2\Z)_{M_i}, \]
with a natural basis indexed by the $M_i$'s.
\vskip 5pt

Now to remove the mention of the hypothetical group $L_E$, observe that when specialized to the case $G = \GL(V)$, with $\dim V = n$, the Langlands-Arthur conjecture simply says that there is a natural bijection
\[ \xymatrix { \left \{ \text{cuspidal representations of $\GL(V)$}\right \} 
\ar@{<->}[d] \\
\left \{ \text{ irreducible $n$-dimensional representations of $L_F$} \right 
\}.} \]
Thus, one may suppress the mention of $L_F$ by replacing the latter set with the former. 
Hence, in the context of the classical groups $G(V)$,
one replaces the data of each $M_i$ by a cuspidal representation $\pi_i$ of $\GL_{n_i}(\A_E)$, with $n_i = \dim M_i$. Moreover,
in view of Proposition \ref{P:cd-invariant} and its analog for symplectic and orthogonal groups, 
the selfduality or conjugate-duality of $M_i$ with sign $b$ can be described invariant theoretically and hence can be captured by the following $L$-function condition:
\vskip 5pt
\noindent (a)  a cuspidal representation $\pi$ of $\GL_n(\A)$ is selfdual of sign
\[  \begin{cases}
 \text{$+1$ if its symmetric square $L$-function $L(s, \pi, \text{Sym}^2)$ has a pole at $s = 1$;} \\
\text{$-1$ if its exterior square $L$-function $L^S(s, \pi, \wedge^2)$ has a pole at $s =1$.} \end{cases} \]
\vskip 10pt

\noindent (b) a cuspidal representation $\pi$ of $\GL_n(\A_E)$ is conjugate-dual of sign
\[  \begin{cases}
\text{$+1$, if the Asai $L$-function $L(s, \pi, {\rm As})$ has a pole at $s = 1$;} \\
\text{$-1$, if the Asai $L$-function $L(s, \pi, {\rm As}^-)$ has a pole at $s = 1$.} 
\end{cases} \]

\vskip 10pt
To summarize, a discrete global $L$-parameter $\varphi$ for $G_0 = G(V_0)$ is the data of a number of 
cuspidal representations of $\GL_{n_i}(\A_E)$, with $\sum_i n_i = \dim V_0$, satisfying the above $L$-function conditions.   The point of this reformulation is that given such a global $L$-parameter, one still has the data given in (i) -(v) above. More precisely, one has:
\vskip 10pt

\begin{enumerate}
\item[(i)]  The global component group $A_{\varphi}$ is simply the 2-group 
$\prod_i (\Z/2\Z)_{\pi_i}$ with a canonical basis indexed by the $\pi_i$'s. 
\vskip 5pt

\item[(ii)] For each $v$, the associated local $L$-parameter is the representation 
\[  \varphi_v = \bigoplus_i  \varphi_{\pi_{i,v}} \]
of $WD(k_v)$,  where $\varphi_{\pi_{i,v}}$ is the local $L$-parameter of the local component $\pi_{i,v}$ of $\pi_i$.
The $L$-function condition presumably forces each $\varphi_{i,v}$ to be selfdual or conjugate-dual with the given sign $b$. Moreover, one has a natural homomorphism 
\[  A_{\varphi} \longrightarrow A_{\varphi_v} = \prod_i A_{\varphi_{\pi_{i,v}}} \]
arising from the natural map
\[  (\Z/2\Z)_{\pi_i} \rightarrow C_{\varphi_{\pi_{i,v}}} \rightarrow A_{\varphi_{\pi_{i,v}}}, \]
obtained by sending $(\Z/2\Z)_{\pi_i}$ to the central subgroup $\langle \pm 1 \rangle$ in the centralizer $C_{\varphi_{\pi_{i,v}}}$. Thus, one continues to have the diagonal map $\Delta$.
\vskip 5pt
\item[(iii)] For each place $v$, the local parameter thus gives rise to a local Vogan packet 
$\Pi_{\varphi_v}$ as before.
\vskip 5pt

\item[(iv)] One can now define the global Vogan packet as before.
\vskip 5pt

\item[(v)] The formula for $m_{\eta}$ is as given before.
\end{enumerate}
\vskip 10pt

The formulation of the Langlands-Arthur conjecture given above amounts to a description of the discrete spectrum of classical groups in terms of the automorphic representations of $\GL_n$. The proof has been promised in a forthcoming book [A3]. 
\vskip 10pt

In the remainder of this section, we formulate
an extension of the Langlands-Arthur conjecture to the case of the metaplectic groups $\widetilde{\Sp}(W)$.

\vskip 10pt

Motivated by Theorem \ref{T:KR}, one expects that discrete global $L$-parameters for $\widetilde{\Sp}(W)$
should be discrete global $L$-parameters for $\SO(2n+1)$ with $2n = \dim W$.
Thus, a discrete global $L$-parameter of $\widetilde{\Sp}(W)$ should be a multiplicity free $2n$-dimensional symplectic representation of $L_F$:
\[ M= M_1 \oplus \cdots\oplus M_r \] 
with each irreducible summand $M_i$ also symplectic. 
In the reformulation of $L$-parameters given above, it is thus given by the data of a collection of pairwise inequivalent cuspidal representations $\pi_i $ of $\GL_{2n_i}(\A)$ with
\vskip 5pt
\begin{enumerate}[(a)]
\item $\sum_i n_i = n$ and
\vskip 5pt
\item  $L(s, \pi_i, \wedge^2)$ having a pole at $s = 1$ for each $i$. 
\end{enumerate}
Using Theorem \ref{T:KR} and Corollary \ref{C:KR}, one sees that such a global $L$-parameter   
$\varphi$ continues to give rise to the data (i) - (iv) above in the context of $\widetilde{\Sp}(W)$.
In particular, one obtains a global Vogan packet $\Pi_{\varphi}$ of irreducible genuine representations of $\widetilde{\Sp}(W)(\A)$ with a bijection
\[  \Pi_{\varphi} \longleftrightarrow \text{Irr}\left( \prod_v A_{\varphi_v} \right). \]
However, the multiplicity formula given in (v) above needs to be modified. Motivated by results of Waldspurger in the case when $\dim W = 2$, we make the following conjecture.

\vskip 10pt

\begin{con} \label{conj:arthur-meta}
Let $(\varphi, M)$ be a discrete global $L$-parameter for $\widetilde{\Sp}(W)$ with associated global Vogan packet $\Pi_{\varphi}$. Let $\chi_{\varphi}$ be the character on the global component group $A_{\varphi}$ defined by
\[   \chi_{\varphi}(a) = \epsilon(1/2, M^a). \]
More concretely, if $a_i \in A_{\varphi}$ is the basis element associated to the factor $(M_i,\pi_i)$ in $M$, then $\chi_{\varphi}(a_i) = \epsilon(\pi_i, 1/2)$. Then
\[  L^2_{\varphi}(\widetilde{\Sp}(W)) \cong \bigoplus_{\eta} m_{\eta} \pi_{\eta} \] 
where
\[  m_{\eta} = \langle \Delta^*(\eta), \chi_{\varphi} \rangle. \]
\end{con}
\vskip 10pt

  We note that Arthur has also introduced nontrivial quadratic characters of
the global component group $A_{\varphi}$  in his conjectures for the multiplicities of non-tempered representations of linear groups. 
\vskip 10pt

We conclude this section with some ramifications of Conjecture \ref{conj:arthur-meta}.
Given a discrete global L-parameter $(\varphi, M)$ (relative to a fixed additive character $\psi$) for $\widetilde{\Sp}(W)$, with $\dim W = 2n$, note that $M$ is also a discrete global L-parameter  for $\SO(V)$ with $\dim V = 2n+1$. For each place $v$, the elements in the associated Vogan packets 
\[ \begin{cases}
\text{$\Pi_{\varphi_v}(V)$ of $\SO(V)$} \\
\text{$\Pi_{\varphi_v}(W)$ of $\widetilde{\Sp}(W)$} 
\end{cases}  \]
 are both indexed by $\text{Irr}(A_{\varphi_v})$. For a character $\eta_v$ of $A_{\varphi_v}$, let 
\[  \pi_{\eta_v} \in \Pi_{\varphi_v}(V) \quad \text{and} \quad \sigma_{\eta_v} \in  \Pi_{\varphi_v}(W) \]
be the corresponding representations. By construction, $\pi_{\eta_v}$ and $\sigma_{\eta_v}$ are local theta lifts (with respect to $\psi$) of each other.
One might expect that, globally, the submodule $L^2_{\varphi}(\widetilde{\Sp}(W))$ of the discrete spectrum can be obtained from $\bigoplus_{V'} L^2_{\varphi}(\SO(V'))$ using global theta correspondence. As we explain below, this is not always the case.
 \vskip 10pt
 
More precisely, if $\eta = \otimes_v \eta_v$ is a character of $\prod_v A_{\varphi_v}$, then
the corresponding representations
\[  \pi_{\eta} \in \Pi_{\varphi}(V) \quad \text{and} \quad \sigma_{\eta} \in \Pi_{\varphi}(W) \]
may or may not be global theta lifts of each other. Indeed, 
\[  \text{$\pi_{\eta}$ occurs in the discrete spectrum of some $\SO(V')$} \Longleftrightarrow \Delta^*(\eta) = 1 \]
whereas
\[  \text{$\sigma_{\eta}$ occurs in the discrete spectrum of $\widetilde{\Sp}(W)$} \Longleftrightarrow
\Delta^*(\eta) = \chi_{\varphi}. \]
 Thus, if $\chi_{\varphi}$ is non-trivial, then the subset of $\eta$'s which indexes automorphic representations for $\SO(V)$ will be disjoint from that which indexes automorphic representations of $\widetilde{\Sp}(W)$. In such cases, there is clearly no way of obtaining the automorphic elements in $\Pi_{\varphi}(W)$ from those of $\Pi_{\varphi}(V)$ via global theta correspondence (with respect to $\psi$). 
 \vskip 10pt
 
 Suppose, on the other hand, that $\chi_{\varphi}$ is the trivial character, 
 so that 
 \[  \epsilon(M_i, 1/2) = 1 \]
  for all the irreducible symplectic summands $M_i$ of $M$. Then
 the automorphic elements in $\Pi_{\varphi}(W)$ and $\Pi_{\varphi}(V)$ are indexed by the same subset of $\eta$'s and are abstract theta lifts of each other.
 However, to construct $L^2_{\varphi}(\widetilde{\Sp}(W))$ from $L^2_{\varphi}(\SO(V))$ via global theta correspondence, there is still an issue with the  non-vanishing of global theta liftings. 
 In this case, the non-vanishing of the global theta lifting is controlled by the non-vanishing of the central L-value 
 \[  L(M, 1/2) = \prod_i  L(M_i, 1/2). \]
Only when $L(M, 1/2)$ is nonzero does one know that $L^2_{\varphi}(\widetilde{\Sp}(W))$ can be obtained from $\bigoplus_{V'} L^2_{\varphi}(\SO(V'))$ by global theta lifting (with respect to $\psi$).  
 \vskip 10pt
 
Another observation is that while the packet $\Pi_{\varphi}(V)$ always contains automorphic elements (for example the representation corresponding to $\eta = 1$), it is possible that none of the elements in the packet $\Pi_{\varphi}(W)$ are automorphic.

\vskip 10pt

We  give two examples which illustrate Conjecture \ref{conj:arthur-meta} and the phenomena noted above,  in the case $\dim W = 2$. In this case, the conjecture is known by the work of Waldspurger [Wa1, 2]. 
\vskip 10pt

\noindent{\bf Example 1}:  Suppose that $(\varphi,M)$ is a discrete global L-parameter for $\SO(3) \cong \PGL(2)$ and $\widetilde{\Sp}(2)$, so that 
\[  A_{\varphi} = \Z/2\Z.  \] 
Suppose that for 3 places $v_1$, $v_2$ and $v_3$,  the local L-parameter $\varphi_{v_i}$ corresponds to the Steinberg representation of $\PGL(2)$, and $\varphi_v$ is unramified for all other $v$. Then 
\[  \epsilon(M, 1/2) = -1, \]
so that $\chi_{\varphi}$ is the non-trivial character of $A_{\varphi}$.  
\vskip 5pt

In this case, the local Vogan packets (for both $\SO(3)$ and $\widetilde{\Sp}(2)$) have size $2$
at the 3 places $v_i$, and we label the representations by
\[  \Pi_{\varphi_v}(\SO(3))  = \{ \pi_v^+, \pi_v^- \} \quad \text{and} \quad \Pi_{\varphi_v}(\widetilde{\Sp}(2)) = \{ \sigma_v^+, \sigma_v^- \}, \]
with the minus sign  indicating the non-trivial character of $A_{\varphi_v}$. 
At all other places, the local packets are singletons.
Thus, the global L-packet of $\SO(3)$ has $8$ elements, which we can label as $\pi^{+++}$, 
$\pi^{++-}$ and so on. Similarly, the global L-packet for $\widetilde{\Sp}(2)$ also has 8 elements, denoted by $ \sigma^{+++}$, $\sigma^{++-}$ and so on.
\vskip 5pt

Now observe that a representation in the global Vogan packet for $\SO(3)$ is automorphic if and only if it has an even number of minus signs in its label,
whereas a representation in the global Vogan packet for $\widetilde{\Sp}(2)$ is automorphic  if and only if  it has an odd number of minus signs in its label.
\vskip 10pt

\noindent{\bf Example 2}:  Suppose again that $(\varphi, M)$ is a discrete global parameter for $\SO(3)$ and $\widetilde{\Sp}(2)$, but now assume that
$\varphi_v$ is reducible for all v. Moreover,  suppose that 
\[  \epsilon(M,1/2) =  -1, \]
so that $\chi_{\varphi}$ is the non-trivial character of $A_{\varphi} = \Z/2\Z$.
These conditions can be arranged. 
\vskip 5pt

In this case,  the global  Vogan packets for $\SO(3)$ and $\widetilde{\Sp}(2)$ are both singletons, containing the representation $\pi$ and $\sigma$ respectively, which are indexed by the trivial character $\eta$ of 
\[  \prod_v A_{\varphi_v} = {1}. \]
Hence,  $\Delta^*(\eta)$ is the trivial character of $A_{\varphi}$. In particular, $\pi$ is automorphic for SO(3), whereas $\sigma$ is not automorphic for $\widetilde{\Sp}(2)$.

 \vskip 15pt

 \section{Revisiting the global conjecture} \label{S:central}
 
 In this section, we shall revisit the global conjecture formulated in \S \ref{S:global-conj}. In particular, we shall approach the restriction problem using the framework of the Langlands-Arthur conjecture reviewed in \S \ref{S:arthur}. 
 \vskip 10pt

We start with a quasi-split group $G_0$ over $F$ associated to the pair of spaces $W_0 \subset V_0$ and fix an automorphic generic character $\theta$ of $U$ as in the previous section; in particular, 
$\theta$ may depend on the choice of an appropriate additive character $\psi$ in various cases.
Given a discrete global $L$-parameter $(\varphi, M, N)$ for $G_0$, there is a corresponding submodule $\mathcal{A}^2_{\varphi}$ in the automorphic discrete spectrum and we are interested in 
the restriction of the linear functional $F(\nu)$ to this submodule.
 \vskip 10pt

Recall that a natural symplectic representation $R$ of ${^L}G$ plays a prominent role in the global conjecture \ref{conj:global1}. Using the global component group $A_{\varphi}$ of the parameter $\varphi$, we 
may refine the associated $L$-function $L(\pi, R,s)$ for $\pi \in \Pi_{\varphi}$, as follows.
If $a \in A_{\varphi}$, we may consider it as an element of $A_{\varphi_v}$ for any $v$ and then choose any element in $C_{\varphi_v}$ projecting to it. Denoting any such element in 
$C_{\varphi} \subset \widehat{G}$ by $a$ again, we see that the subspace $R^a$ is a submodule of $WD(F_v)$ under $\varphi_v$. Thus, one has the associated $L$-function  
\[  L(\pi, R^a, s) = \prod_v L(R^a \circ \varphi_v, s) \]
and epsilon factor 
\[  \epsilon(\pi, R^a, s) = \prod_v \epsilon(R^a \circ \varphi_v, \psi_v, s). \]
\vskip 15pt

We are now ready to revisit the global restriction problem. 
Let us first draw some implications of the various local conjectures we have made so far.
\vskip 5pt

\noindent (i) According to our local conjectures \ref{conj-mult1} and \ref{conj-character}, there is a unique representation $\pi_v$ in the local Vogan packet $\Pi_{\varphi_v}$ such that  
\[  \Hom_{H(F_v)}(\pi_v \otimes \overline{\nu_v}, \CC) \ne 0, \]
 and this distinguished representation is indexed by a distinguished (relevant) character 
 \[  \chi_v  \quad \text{of $A_{M_v} \times A_{N_v}$}. \]
For each $v$, the representation $\pi_{\chi_v}$ is a representation of a pure inner form $G_v$ of $G_0$ over $F_v$, associated to a pair of spaces $W_v \subset V_v$. 
\vskip 5pt

\noindent (ii) According to our unramified local conjecture \ref{conj-unramified}, 
for almost all $v$, the distinguished character $\chi_v$ is trivial and the representations $\pi_{\chi_v}$,
$\theta_v$ and $\nu_v$ are all unramified. 
At these places, the pair of spaces $W_v \subset V_v$ is
simply $W_{0,v} \subset V_{0,v}$ and the group $G_v$ is simply $G_0(F_v)$.
Thus, we can form the restricted direct product groups
\[   \begin{cases}
  G_{\A} = \prod_{J_v} G_v; \\
  H_{\A} = \prod_{J_v \cap H_v} H_v,\end{cases}  \]
and  representations
\[  \begin{cases}
\pi_{\chi} = \hat{\otimes}_v \pi_{\chi_v} \, \text{of $G_{\A}$;} \\
\nu = \hat{\otimes}_v \nu_v \, \text{  of $H_{\A}$,} \end{cases} \]
which are restricted tensor products, defined using the unique line of $J_v$-invariant  or $J_v \cap H_v$-invariant vectors for almost all $v$.  The representation $\pi_{\chi}$ is simply the element in the global Vogan packet $\Pi_{\varphi}$ indexed by the distinguished character 
\[  \chi = \otimes_v \chi_v  \]
of the compact group $\prod_v A_{M_v} \times A_{N_v}$.
It is the  only (relevant) element in $\Pi_{\varphi}$ such that
 \[  \Hom_{H(\A)}(\pi_{\chi} \otimes \overline{\nu},\CC) \ne 0.   \]
\vskip 15pt
 
 With these preliminaries out of the way, there are now three questions to address.
 \vskip 10pt
 
 \noindent (1) Are $H_{\A} \hookrightarrow G_{\A}$ the adelic points of algebraic groups 
 \[  H = N \cdot G(W)  \hookrightarrow G = G(V) \times G(W) \]
 defined over $F$, associated to a relevant pair of spaces $W \subset V$ over $E$?
  
 \vskip 5pt
 
  In the symplectic case, this question clearly has a positive answer, since there is a unique symplectic vector space in any even dimension over any field. Hence, we focus on the other cases, where the issue is whether the collection $(W_v \subset V_v)$ of local spaces is coherent in the terminology of Kudla.  Now these local spaces have the same rank as $W_0 \subset V_0$ and the same discriminant in the orthogonal case. Hence, they form a coherent collection if and only if we have changed the Hasse-Witt invariant or the hermitian/skew-hermitian discriminants at an even number of places $v$. This is equivalent to the identity
  \[  1 = \prod_v \chi_v(-1,1) = \begin{cases}
  \prod_v \epsilon (M_v \otimes N_v, \psi_v), \text{ in the orthogonal and hermitian cases;} \\
  \text{} \\
   \prod_v \epsilon(M_v \otimes N_v(\mu^{-1}_v), \psi_v),  \text{ in the skew-hermitian case.} \end{cases}
 \]
 Thus, we will have a global pair of spaces $W \subset V$ with these localizations if and only if
 \[  \epsilon(\pi_0, R, 1/2)  =1. \]
 \vskip 5pt
 
 \noindent  Assuming this is the case, the second question is:
 \vskip 10pt
 
 \noindent (2) Does the representation $\pi_{\chi}$ in the global Vogan packet $\Pi_{\varphi}$ 
  occur in the space $\mathcal{A}_0(G)$ of cusp forms? 
  
  \vskip 10pt
  
  To answer this question, we exploit the Langlands-Arthur conjecture discussed in \S \ref{S:arthur}. Thus, in the 
  orthogonal, hermitian and skew-hermitian cases, we need to see if the distinguished character $\chi$ is 
    trivial when restricted to the global component group $A_{\varphi}$ via the diagonal map $\Delta$. 
  This amounts to the assertion that, for all $a \in A_{\varphi}$,
  \[  1 = \prod_v \chi_v(a) = \begin{cases} 
  \prod_v \epsilon((M_v \otimes N_v)^a, \psi_v), \text{ in the orthogonal or hermitian cases;} \\
  \text{} \\
 \prod_v  \epsilon(M_v \otimes N_v(\mu_v^{-1}))^a, \psi_v), \text{  in the skew-hermitian case,} \end{cases} \] 
 or equivalently, that
\[    \epsilon(\pi_0, R^a, 1/2) = 1  \quad \text{for all $a \in A_{\varphi}$.} \]
    On the other hand, in the symplectic case, using the multiplicity formula given in Conjecture 
 \ref{conj:arthur-meta}, we see that we need the distinguished character to be equal on $A_{\varphi} = A_M \times A_N^+$ to the character $\chi_0 \times 1$ (assuming that $M$ is symplectic and $N$ orthogonal). This translates to the same condition
 \[  1 = \prod_v \epsilon((M_v \otimes N_v)^a, \psi_v) = \epsilon(\pi_0, R^a, 1/2). \]

  \vskip 5pt
Finally, assuming that $\epsilon(\pi_0, R^a, 1/2) = 1$ for all $a \in A_{\varphi}$, so that $\pi_{\chi}$ occurs in the space of cusp forms, we can ask:
   
   \vskip 10pt
  
 \noindent (3) Does the linear form  $F(\nu)$ have nonzero restriction to $\pi_{\chi}$?
  \vskip 10pt

The point is that, since the above conditions on epsilon factors are assumed to hold, there are no trivial reasons for the central critical $L$-value  $L(\pi_0, R, 1/2)$ to vanish. Here then is the second form of our global conjecture:
 \vskip 10pt
 
 \begin{con} \label{conj:global2}
 Let $\pi_{\chi}$ be the representation in the global Vogan packet  $\Pi_{\varphi}$ corresponding to the distinguished character $\chi$. Then the following are equivalent:
 \vskip 5pt
 
 \begin{enumerate}[(i)]
 \item $\pi_{\chi}$ occurs in $\mathcal{A}_0(G)$ and the linear form $F(\nu)$  is nonzero on $\pi_{\chi}$
 \vskip 5pt
 
 \item $L(\pi_{\chi}, R, 1/2) \ne 0$. 
  \end{enumerate}
  \end{con}
  
         \vskip 15pt
    
\section{The first derivative}  \label{S:derivative}

We maintain the notation and setup of the previous section, so that $W_0 \subset V_0$ is a pair of spaces over $E$ with quasi-split  group $G_0 = G(V_0) \times G(W_0)$ over $F$. For a given discrete global
 $L$-parameter $(\varphi, M,N)$ of $G_0$, we have a distinguished representation 
 \[  \pi =\pi_{\chi} = \hat{\otimes}_v \pi_v \]
 in the global Vogan packet $\Pi_{\varphi}$, which is a representation of a restricted direct product 
 \[  G_{\A} = \prod_{J_v} G_v(F_v) \]
  and is the unique element in the packet such that
 \[  \Hom_{H_{\AA}}(\pi \otimes \overline{\nu}, \CC) \not= 0. \]
 \vskip 5pt
 
\noindent  In this final section, we specialize to the orthogonal and hermitian cases (i.e. where $\epsilon = 1$)
and assume that 
\[
\ep(\pi_0, R, 1/2) = -1 \quad \text{so that $L(\pi_0, R, 1/2) = 0$,} 
\]  
where $\pi_0$ is the generic automorphic representation of $G_0(\AA)$ with parameter $\varphi$.
In this case the group $G_{\AA}$  does not arise from a pair of orthogonal or hermitian spaces $W \subset V$ over $E$.  In Kudla's 
terminology, the local data $(W_v \subset V_v)$ is incoherent. 
Nevertheless, we can formulate a global conjecture in this case, provided that the following condition holds:
\vskip 5pt

\begin{itemize}
\item[($\ast$)]  There is a non-empty  set $S$ of places of 
$F$, containing all archimedean primes, such that
the groups $G_v(F_v)$ and $H_v(F_v)$ are {\it compact} for all places $v 
\in S$. 
\end{itemize}
\vskip 10pt

\noindent This condition has the following implications:
\vskip 5pt

\begin{enumerate}[(i)]
\item  $\dim W_0^{\perp} = 1$. Indeed, 
this follows from the fact that $H_v = N_v .G(W_v)$ and a nontrivial unipotent subgroup $N_v$ cannot be compact.  Hence, for all places $v$, we have
$$\dim V_v = \dim W_v + 1.$$
If  we consider the orthogonal decomposition
\[  V_0 = W_0 \oplus L \]
over $E$, then  since $W_v \subset V_v$ is relevant for all $v$,  we have 
\[ L_v = W^{\perp}_v.\] 
Thus, though the collection $(W_v \subset V_v)$ is not coherent, the collection $(W_v^{\perp})$ is.
\vskip 10pt

\item  Any archimedean place $v$ of $F$ is real  and the space 
$V_v$ must be definite.  In the hermitian case, we must have $E_v = \CC$.  
Hence, in the number field case, $F$ is totally real and, in the hermitian case, $E$ is a $CM$ field.  
Moreover, at all archimedean places $v$ of $F$, the generic representation $\pi_{0,v}$ of 
$G_0(F_v)$ is in the discrete series, and $\pi_v$ is a finite 
dimensional representation of the compact group 
$$G_v(F_v) = \SO(n) \times \SO(n-1) \, {\rm or}\    \U(n) \times 
\U(n-1)$$
with a unique line fixed by $H(F_v) = \SO(n-1)$ or $ \U(n-1)$. 
\vskip 10pt

\item At finite primes $v \in S$, we must have $\dim (V_v/F_v) \le 4$.  Indeed, 
a 
quadratic form of rank $ \geq 5$ over $F_v$ represents $0$.  Hence, for 
function fields $F$, we have the following nontrivial cases:

\[   
(\dim V_v, \dim W_v) = \begin{cases}
\text{$(3,2)$ or $(4,3)$ in the orthogonal case;} \\
\text{$(2,1)$ in the unitary case.}  
\end{cases} \]
\end{enumerate}
\vskip 10pt

For simplicity, we will assume that $F$ is a totally real number field and
$S$ consists only of the archimedean places.  In the hermitian 
case, the quadratic extension $E$ of $F$ is a $CM$ field.  
\vskip 10pt

Suppose first that the 
spaces $V_v$ are orthogonal  of dimension $n \geq 3$. Fix a real place 
$\alpha$. If $V_{\alpha}$ has signature $(n, 0)$, let 
\[  W^*_{\alpha} 
\subset V^*_{\alpha} \]
be the unique orthogonal spaces over $F_{\alpha} = \RR$ with signatures 
\[  (n-3,2) \subset (n-2, 2). \]
 If $V_{\alpha}$ has signature $(0, n)$, let 
\[ W^*_{\alpha} \subset V^*_{\alpha} \]
 have signatures 
 \[  (2, n-3) \subset (2, n-2). \]
 Since we have modified the 
Hasse-Witt invariant at a single  place of $F$, and kept the discriminant 
of $W_{\alpha}^{*\perp} \simeq W_v^{\perp}$ equal to the discriminant of 
$L_{\alpha}$, there are unique global spaces 
\[  W^{\alpha} \subset V^{\alpha} \quad \text{over $E$} \]
with localizations 
\[ \begin{cases}
\text{$W_v \subset V_v$ for all $ v \not= \alpha$,} \\ 
\text{$W^*_{\alpha} \subset V_{\alpha}^*$ at $\alpha$.} \end{cases} \]
 We note that we can make such a pair of global spaces for  {\it any} place $\alpha$ of $F$, having localizations $W_v \subset V_v$ 
for all $v \not= \alpha$, provided that $\dim W_{\alpha} \geq 3$.  When 
$\dim W_{\alpha} = 2$, we can make such a global space provided that 
$W_{\alpha}$ is {\it not  split} over $F_{\alpha}$, i.e. for all primes 
$\alpha$ which are ramified or inert in the splitting field $E$ of the 
$2$-dimensional space $W^0$.  The proof is similar to [Se, Prop 7].
\vskip 10pt

We can use the global spaces $W^{\alpha} \subset V^{\alpha}$ so constructed  to define 
the groups 
\[  H^{\alpha} \hookrightarrow G^{\alpha} = \SO(V^{\alpha}) 
\times \SO(W^{\alpha}) \] 
over $F$.  These have associated Shimura varieties 
$$\Sigma (H^{\alpha}) \hookrightarrow \Sigma(G^{\alpha})$$
over $\CC$, of dimensions $n-3$ and $2n-5$ respectively, which are 
defined over the reflex field $E= F$, embedded in $\CC$ via the place 
$\alpha$.  The varieties over $F$ are independent of the choice of the  real place $\alpha$, so we denote them simply by
$$\Sigma (H) \hookrightarrow \Sigma(G),$$ 
suppressing the mention of $\alpha$.
\vskip 5pt

Next, suppose that the spaces $V_v$ are hermitian over $E_v$ of dimension $n 
\ge 2$.  Fix a real place $\alpha$ and a complex embedding $z : 
E_{\alpha} \to \CC$.  If $V_{\alpha}$ has signature $(n, 0)$, let 
\[  W_{\alpha}^* \subset V^*_{\alpha} \]
be the unique hermitian spaces over  $E_{\alpha}$ with signature 
\[  (n-2, 1) \subset (n-1, 1). \]  
If  $\Va$ has signature $(0, n)$, let 
\[  W_{\alpha}^* \subset V_{\alpha}^* \]
be the 
unique hermitian spaces over $E_{\alpha}$ with signatures 
\[ (1, n-2) \subset (1, n-1). \]
 Again, since we have modified the hermitian 
discriminants  at a single  place $\alpha$ of $F$, and kept 
$(W_{\alpha}^*)^{\perp} \simeq W_v^{\perp}$ constant, there is a unique 
pair of global spaces 
\[  W^{\alpha} \subset V^{\alpha} \] 
over $E$ with  localizations 
\[  \begin{cases}
\text{$W_v \subset V_v$, for all $v \not= \alpha$;} \\ 
 \text{$W^*_{\alpha} \subset \Va^*$ at $\alpha$.} \end{cases} 
 \]  
Again,  we can make such a modification at any place $\alpha$ of $F$ which is not split in 
the quadratic extension $E$. 
\vskip 5pt

As before,  we use the global spaces $W^{\alpha} 
\subset V^{\alpha}$ to define groups 
\[  H^{\alpha} \hookrightarrow G^{\alpha} = \U(V^{\alpha}) \times \U(W^{\alpha}) \]
over $F$.  These have associated Shimura varieties 
\[ \Sigma(H^{\alpha}) \hookrightarrow \Sigma(G^{\alpha}) \]
over $\CC$, of dimensions $(n-2)$ and $(2n-3)$ respectively, which are 
defined over the reflex field $= E $, embedded in $\CC$ via the 
extension $z$ of the place $\alpha$.  These varieties over $E$ are 
independent of the choice of real place $\alpha$ of $F$, so we denote 
them simply by:
$$\Sigma(H) \hookrightarrow \Sigma (G).$$
\vskip 10pt

We sketch the definition of the Shimura variety $\Sigma$ of dimension $n-1$ associated to incoherent 
hermitian data $\{W_v\}$ of dimension $n$ which is definite at all real 
places $v$ of $F$; the orthogonal case is similar.  Take the modified 
space $W^{\alpha}$ at a real place $\alpha$, and let $G = Res_{F/\Q} 
\U(W^{\alpha})$.  We define a homomorphism 
$$h : S_{\RR} = Res_{\CC/\RR} \GG_m \to G_{\RR} = 
{\displaystyle{\prod_{v|\infty}}} \U(W^{\alpha}_v)$$ 
as follows. Let $\langle e_1, \cdots, e_n\rangle $ be an orthogonal basis of 
$W^*_{\alpha}$, such that the definite space $e_1^{\perp}$ has the same 
sign as the definite space $W_{\alpha}$.  We set 
\[  h(z) = \begin{pmatrix}z/\overline{z}\\ & \ddots  &   \\
& & 1_{ \ddots 1} 
\end{pmatrix} \quad
 \text{in $\U(W^*_{\alpha})$} \]  
 and  
 \[  \text{$h(z) = 1$ in all the other (compact) components 
${\displaystyle{\prod_{v \not= \alpha}}} \U(W_v)$.} \]
 Let $X$ be the 
$G_{\RR}$-conjugacy class of $h$, which is isomorphic to the unit ball 
$\U_{n-1,1} / \U_{n-1} \times \U_1$ in $\CC^{n-1}$.  The pair $(G, X)$ 
satisfies the axioms for a Shimura variety [De3, $\S$2.1].  The 
composite homomorphism 
$$w: (\G_m)_{\RR} \to S_{\RR} \to G_{\RR}$$ 
is trivial, and the reflex field of $\Sigma(G) = M(G, X)$ is equal to $E$, 
embedded in $\CC$ via the homomorphism $z$  extending $\alpha$.  Indeed, 
the miniscule co-character  $\mu : (\G_m)_{\CC} \to G_{\CC}$ is 
defined over $E$:
$$\mu(\alpha) = \begin{pmatrix}\alpha \\ & \ddots & \\
 \end{pmatrix} \times 1 .$$
The complex points of $\Sigma(G)$ are:
$$\Sigma (G,\CC) = G(\Q)\backslash [X \times G (\widehat{\Q})].$$
Over $E$, the variety $\Sigma(G)$ and its action of $G(\widehat{\Q}) = 
{\displaystyle{\prod_{v {\rm\ finite}}}} \U(W_v)$ depend only on the 
incoherent family $\{W_v\}$.  If 
\[   \pi_{\infty} = {\displaystyle{\otimes_{v \ {\rm real}}}} \pi_v\] 
is any finite dimensional representation of the compact group 
${\displaystyle{\prod_{v \ {\rm real}}}} \U(W_v)$, there is a local system 
${\mathcal F}$ on $\Sigma(G)$ over $E$ associated to $\pi_{\infty}$.  
\vskip 10pt

We now return to 
the study of the $L$-function $L(\pi_0, R, s)$ at $s = 1/2$, 
using the arithmetic geometry of the cycle 
$$\Sigma(H) \hookrightarrow \Sigma(G)$$
associated to the incoherent family $(W_v \subset V_v)$.  The 
representation 
\[ \text{$\pi_{\infty} = {\displaystyle{\otimes_{v\ {\rm real}}}}\pi_v$ of $\prod_{\text{$v$ real}} G_v(F_v)$} \]
 gives a local system ${\mathcal F}$ on $\Sigma(G)$ which contains the 
trivial local system $\CC$ when restricted to the cycle $\Sigma(H)$.   
\vskip 10pt

To get the appropriate representation $\pi_f = 
{\displaystyle{\widehat{\otimes}_{v~{\rm finite}}}} \pi_v$ of 
$G(\widehat{\Q})$ on the Chow group of $\Sigma(G)$ with coefficients in 
${\mathcal F}$, we need to find this representation in the middle 
dimensional cohomology of $\Sigma(G)$ with coefficients in ${\mathcal F}$ 
(which is the ``tangent space'' of the Chow group).  Hence we need
$$\Hom_{G(\widehat{\Q})} (\pi_f, H^d(\Sigma(G), {\mathcal F})) \not= 0$$
with $d = \dim \Sigma(G)$.  We put 
\[  n = \begin{cases}
\dim V^{\alpha}, \text{  in the hermitian case;}\\
 \dim V^{\alpha} -1, \text{  in the orthogonal case,} \end{cases} \]
so that $n \geq 2$ in all cases.  We have
\[  \begin{cases}
\dim \Sigma(G) = d =  2n-3 \\
\dim \Sigma (H)  =  n-2 \\
{\rm codim} \  \Sigma(H)  =  n-1.
\end{cases} \]
\vskip 5pt

\noindent  Now Matsushima's formula for cohomology shows that 
\[  \dim  \Hom_{G(\widehat{\Q})}(\pi_f, H^{d} (\Sigma(G),{\mathcal F})) \]
is equal to the sum of multiplicities in the cuspidal spectrum 
$$\sum_{\pi^{\alpha}}  m (\pi^{\alpha} \otimes (\otimes_{\text{$\beta \ne \alpha$ real}} \pi_{\beta})\otimes \pi_f)  \cdot  (G 
(\RR_\alpha) : G(\RR_\alpha)^0)$$
over the discrete series representations $\pi^{\alpha}$ of 
$G(\RR_\alpha)$ with the same infinitesimal and central character as 
$\pi_{\alpha}$.  If all of the multiplicities are $1$, the middle 
cohomology of $\Sigma(G)$ with coefficients in ${\mathcal F}$ will contain the 
motive $M \otimes N$ over $F$ or $E$, associated to the parameter of the 
$L$-packet of $\pi_0$. 
\vskip 10pt

On the other hand, the conjecture of Birch and Swinnerton-Dyer, as 
extended by Bloch and Beilinson, predicts that
$$\dim \Hom_{G(\widehat{\Q})} (\pi_f, CH^{n-1} (\Sigma(G), {\mathcal F}))$$
is equal to the order of vanishing of the $L$-function
$$L(\pi_0, R, s)$$
at the central critical point $s = 1/2$.  If the first derivative is nonzero, we 
should have an embedding, unique up to scaling
$$\pi_f \hookrightarrow CH^{n-1} (\Sigma(G), {\mathcal F})$$
and the Chow group of codim$(n-1)$ cycles plays the role of the space 
of automorphic forms in \S \ref{S:central}. 
\vskip 10pt

The height pairing against the codimension $(n-1)$ cycle $\Sigma(H)$, on 
which ${\mathcal F}$ has a unique trivial system, should give a nonzero 
linear form
$$ F : CH^{n-1} (\Sigma(G), {\mathcal F}) \to  \CC$$ 
analogous to the integration of automorphic forms over $H(F) \backslash 
H(\AA)$.   This form is $H(\widehat{\Q})$-invariant, and our global conjecture in this setting is:
\vskip 10pt

\begin{con} \label{conj:global3} 
The following are equivalent:
\vskip 5pt

\begin{enumerate}[(i)]
\item The representation $\pi_f$ occurs in $CH^{n-1}(\Sigma(G), \mathcal{F})$ with multiplicity one and the linear form $F$ is nonzero on $\pi_f$;
\vskip 5pt

\item $L'(\pi_0, R, 1/2) \ne 0$.
\end{enumerate}
\end{con}
\vskip 10pt

\noindent{\bf Remark :} Just as the cohomology of a pro-Shimura variety associated to a reductive group
$G$ over $\QQ$  carries an admissible, automorphic action of $G(\widehat{\QQ})$, it is reasonable to expect that the Chow groups of cycles defined over the reflex field $E$ will also be admissible and automorphic. We note that this is true for the Shimura curves associated to inner forms $G$ of
to $\GL_2(\QQ)$: the action of $G(\widehat{\QQ})$ on the Chow
group of zero cycles of degree 0 is the Hecke action on the Mordell-Weil group of the Jacobian over $\QQ$, which factors through the action of endomorphisms on the differential forms. Here the multiplicity of a representation $\pi_f$ of $G(\widehat{\QQ})$ on the Chow group in the tower is conjecturally equal to the order of zero of the standard $L$-function associated
to $\pi_f$ at the central critical point.

\vskip 10pt

As in the global conjecture in central value case, one expects a refinement of the above conjecture, in the form of an exact formula relating the pairing
$$\langle  \Sigma (H) (\pi_f), \Sigma (H) (\pi_f)\rangle $$ 
to the first derivative $L' (\pi_0, R, 1/2)$. This would 
generalize the formula of Gross-Zagier [GZ], as completed by Yuan-Zhang-Zhang [YZZ], which is the case $n = 2$ where the codimension of 
the cycle is $1$. Such a refined formula  in higher dimensions has been proposed in a recent preprint of W. Zhang [Zh].


\end{document}